\documentclass[12pt, a4paper]{amsart}
\usepackage{url,graphicx}
\usepackage{multirow, xcolor}
\usepackage{fullpage}
\usepackage{amsfonts,amscd,amssymb,amsmath,amsthm,mathrsfs,multirow,lscape,xfrac}
\usepackage{pbox,lipsum, stmaryrd,url, epstopdf,float, enumerate}
\usepackage{array}
\usepackage{xspace, comment}
\usepackage[all,cmtip]{xy}
\usepackage[english]{babel}
\usepackage{color}

\usepackage{ucs}
\usepackage[utf8]{inputenc}
\usepackage{dsfont}
\usepackage[T1]{fontenc} 
\usepackage{lmodern}
\usepackage[thinlines]{easytable}

\usepackage{phaistos}
\usepackage{microtype}

\usepackage{enumitem}
\usepackage{scalerel}
\usepackage{todonotes}
\usepackage{stackengine,wasysym}
\usepackage{todonotes}
\usepackage[T1]{fontenc}
\usepackage[colorlinks,linkcolor=blue,anchorcolor=blue,citecolor=blue,backref=page]{hyperref}
\usepackage{cleveref}
\usepackage[english]{babel}

\newtheorem{theorem}{Theorem}[section]
\newtheorem{cor}[theorem]{Corollary}
\newtheorem{lemma}[theorem]{Lemma}
\newtheorem{prop}[theorem]{Proposition}

\newtheorem{problem}[theorem]{Problem}
\newtheorem{algorithm}[theorem]{Algorithm}
\theoremstyle{definition}
\newtheorem{definition}[theorem]{Definition}
\newtheorem{remark}[theorem]{Remark}

\newcommand{\CC}{\mathbb{C}}
\newcommand{\QQ}{\mathbb{Q}}
\newcommand{\RR}{\mathbb{R}}
\newcommand{\ZZ}{\mathbb{Z}}
\newcommand{\frakp}{\mathfrak{p}}
\newcommand{\calM}{\mathcal{M}}

\newcommand{\N}{\mathscr{N}}

\DeclareMathOperator{\coker}{coker}
\DeclareMathOperator{\Gal}{Gal}
\DeclareMathOperator{\image}{image}
\DeclareMathOperator{\lcm}{lcm}
\DeclareMathOperator{\Norm}{Norm}

\DeclareMathOperator{\Trace}{Trace}

\def\tareesidedbox#1{\setbox0=\hbox{$#1$}\dimen0=\wd0 \advance\dimen0 by3pt\rlap{\hbox{\vrule height9pt width.4pt depth2pt \kern-.4pt\vrule height9.4pt width\dimen0 depth-9pt\kern-.4pt \vrule height9pt width.4pt depth2pt}} \relax \hbox to\dimen0{\hss$#1$\hss}}
\def\house#1{\tareesidedbox{#1}}

\title[The exceptional set in Cassels's theorem]{The exceptional set in Cassels's theorem on small cyclotomic integers}

\author{Jitendra Bajpai}
\address{Department of Mathematics, Christian-Albrechts University of Kiel, 24118 Kiel, Germany and Einstein Institute of Mathematics, Edmond J. Safra Campus, The Hebrew University of Jerusalem, Jerusalem 9190401, Israel}
\email{jitendra@math.uni-kiel.de}

\author{Srijan Das}
\address{Department of Mathematics, Indian Institute of Science Education and Research, Pune, India}
\email{das.srijan@students.iiserpune.ac.in}

\author{Kiran S. Kedlaya}
\address{Department of Mathematics, University of California San Diego, 9500 Gilman Drive \#0112, La Jolla, CA, 92122, USA}
\email{kedlaya@ucsd.edu}

\author{Nam H. Le}
\address{Department of Mathematics and Statistics, Florida Atlantic University, Florida, USA}
\email{namle2024@fau.edu}

\author{Meghan Lee}
\address{Department of Mathematics, University of California Santa Barbara, California, USA}
\email{meghanlee@ucsb.edu}

\author{Antoine Leudi\`ere}
\address{Department of Mathematics and Statistics, University of Calgary, Alberta, Canada}
\email{antoine.leudiere@ucalgary.ca}

\author{Jorge Mello}
\address{Department of Mathematics and Statistics, Oakland University, Michigan, USA}
\email{jorgedemellojr@oakland.edu}


\begin{document}
\date{\today}

\begin{abstract}
    In a 1965 paper, R. Robinson made five conjectures about the classification of cyclotomic algebraic integers for which the maximum absolute value in any complex embedding (the \emph{house}) is small, modulo the equivalence relation generated by Galois conjugation and multiplication by roots of unity. In response to one of these conjectures, Cassels showed in 1969 that when the house is at most $\sqrt{5}$, one obtains three parametric families plus an effectively computable finite set of equivalence classes of exceptions. Building on the work of Jones, Calegari--Morrison--Snyder, and Robinson--Wurtz, we determine this exceptional set. By specializing to the case where the house is strictly less than 2, we resolve the final outstanding conjecture from Robinson's 1965 paper.
\end{abstract}

\maketitle

\tableofcontents

\section{Introduction}

\subsection{General setup}

It is a fundamental problem in algebraic number theory to classify ``small''
algebraic integers. To make this question well-posed, for $\alpha$ an algebraic integer, we define the \emph{house} of $\alpha$, denoted $\house{\alpha}$, to be the maximum of $|\beta|$ as $\beta$ varies over conjugates of $\alpha$ in $\CC$. (The terminology is presumably based on the shape of the notation.) In this notation, the arithmetic-geometric mean inequality implies that $\house{\alpha} \geq 1$ when $\alpha \neq 0$, and a theorem of Kronecker implies that $\house{\alpha} = 1$ if and only if $\alpha$ is a root of unity. 

For a general algebraic integer $\alpha$, one cannot say much more without introducing a dependence on degree.
This can be quantified by citing Dimitrov's recent resolution of the Schinzel--Zassenhaus conjecture \cite{dimitrov, dimitrov2} (see also \cite[Chapter 4]{mckee-smyth}), which gives the lower bound $\house{\alpha} \geq 2^{1/(4[\QQ(\alpha):\QQ])}$ which is known to be essentially best possible.

On the other hand, it was first suggested by Robinson \cite{robinson} that one can hope to formulate much stronger results for \emph{cyclotomic} algebraic integers, i.e., those belonging to the subring of $\CC$ generated by roots of unity. 
This restriction occurs naturally in several mathematical contexts,
including the study of Fermat equations and their generalizations, 
the properties of character tables of finite groups, and the theory of fusion categories and their Frobenius--Perron dimensions \cite{calegari-morrison-snyder}. 

The primary purpose of this paper is to explicate a theorem of Cassels which resolves one of the conjectures raised by Robinson \cite[Conjecture~2]{robinson} concerning the classification of cyclotomic integers $\alpha$ with $\house{\alpha}^2 \leq 5$.
By specializing to the case $\house{\alpha} < 2$,
we resolve the last outstanding conjecture of Robinson
\cite[Conjecture~1]{robinson}; we will encounter the other conjectures from \cite{robinson} throughout this introduction, but see \cite[Chapter~5]{mckee-smyth} for a more thorough survey.

\subsection{The theorem of Cassels}

We next formulate the theorem of Cassels mentioned above \cite[Theorem~I]{cassels} and our main result which complements this theorem. Following Cassels, we declare two algebraic integers $\alpha,\beta$ to be equivalent, and write $\alpha \equiv \beta$, if $\beta$ equals a root of unity times a conjugate of $\alpha$; the house is constant on equivalence classes.

\begin{theorem}[Cassels] \label{T:cassels}
For $\alpha$ a cyclotomic algebraic integer, $\house{\alpha}^2 < 5.01$ if and only if $\alpha$ satisfies one\footnote{Conditions (1)--(3) are not mutually exclusive, but the overlaps between them are finite and known. For instance, the overlaps between (2) and (3) are the subject of \cite[Conjecture~3]{robinson}, for which see \cite[\S 5.8.3]{mckee-smyth} for a proof; see also \Cref{rem:cassels degenerate cases}. We regard (4) as mutually exclusive from (1)--(3) by definition.} of the following conditions.
\begin{enumerate}
\item It is the sum of at most two roots of unity.
\item It is equivalent to $1 + \zeta- \zeta^{-1}$ for some root of unity $\zeta$.
\item It is equivalent to $(\zeta_5+\zeta_5^4) + (\zeta_5^2 + \zeta_5^3)\zeta$ for some root of unity $\zeta$.
\item It belongs to a certain effectively computable finite set of equivalence classes.
\end{enumerate}
\end{theorem}
As a convenient shorthand and in honor of this result, we refer to the quantity $\house{\alpha}^2$ as the \emph{castle} of $\alpha$.
For later reference, we spell out the castles of the specializations of cases (1)--(3) of~\Cref{T:cassels} to the root of unity  $\zeta_N := e^{2 \pi i/N}$ for any positive integer $N$ (see \Cref{algo:cassels test}):
\begin{align}
\nonumber
|\alpha|^2 = \house{\alpha}^2 = 4 \cos^2 \frac{\pi}{N} & \qquad \alpha = 1 + \zeta_N \\
\label{eq:house for cassels families}
|\alpha|^2 = 1 + 4 \sin^2 \frac{2\pi}{N}, \quad \house{\alpha}^2 = 1 + 4 \cos^2 \frac{\pi}{N'(N)} & \qquad \alpha = 1 + \zeta_N - \zeta_N^{-1} \\
\nonumber
|\alpha|^2 =  1 + 4 \sin^2 \frac{\pi}{N}, \quad \house{\alpha}^2 = 1 + 4 \cos^2 \frac{\pi}{N'(2N)} & \qquad \alpha = (\zeta_5+\zeta_5^4) + (\zeta_5^2 + \zeta_5^3)\zeta_N
\end{align}
where $N'(N)$ denotes the denominator of $\tfrac{1}{2} - \frac{2}{N} = \tfrac{N-4}{2N}$ in lowest terms:
\begin{gather}
\label{eq:denom for cassels families}
N'(N) = \begin{cases}
2N & N \equiv 1 \pmod{2} \\
N & N \equiv 2 \pmod{4} \\
N/4 & N \equiv 4 \pmod{8} \\
N/2 & N \equiv 0 \pmod{8}.
\end{cases}
\end{gather}
From \Cref{T:cassels} plus this calculation, we see that 4 and 5 are the only limit points in $[0, 5.01)$ of the set of castles of cyclotomic integers.

Our contribution is to identify the exceptional set in part (4) of \Cref{T:cassels}.

\begin{theorem} \label{T:main}
The equivalence classes of cyclotomic algebraic integers $\alpha$
for which $\house{\alpha}^2 < 5.01$, but $\alpha$ does not have one of the forms (1)--(3) of \Cref{T:cassels}, are precisely those represented by the entries of~\Cref{table:exceptional classes}. (These entries are all necessary; see \Cref{cor:no redundant exceptions}.)
\end{theorem}

The restricted version of \Cref{T:main} with $\house{\alpha}^2 < 3$
was established by Cassels (see \Cref{lem:cassels low height}).
By contrast, the restricted version with $\house{\alpha}^2 < 4$ is itself a novel result
that suffices to resolve \cite[Conjecture~1]{robinson}; see \Cref{cor:robinson1} for details. 

\begin{table}[ht]
\caption{Representatives of the equivalence classes in part (4) of \Cref{T:cassels}, according to \Cref{T:main}. Here $\calM(\alpha)$ denotes the Cassels height (see \S\ref{subsec:metrics}), ``Level'' the minimal level (see \S\ref{subsec:minimal level}), and $\zeta$ a root of unity of order equal to the minimal level.}
\label{table:exceptional classes}
\[
\begin{array}{c|c|c|c|c}
\house{\alpha}^2 & & \calM(\alpha) & \text{Level} & \alpha \\
\hline
\multirow{8}{*}{$1 + 4 \cos^2 \tfrac{\pi}{1}$} & \multirow{8}{*}{$= 5$}& \multirow{8}{*}{5} & 11 & 1 + \zeta + \zeta^2 + \zeta^4 - \zeta^5 + \zeta^{7} \\
& & & 19 & 1+\zeta + \zeta^4 + \zeta^7 + \zeta^8 + \zeta^9 + \zeta^{10} + \zeta^{12} + \zeta^{14} \\
& & & 20 & 1+ \zeta + \zeta^3 - \zeta^4 \\
& & & 24 & 1+ \zeta + \zeta^5 - \zeta^6 \\
& & & 31 & 1+\zeta+\zeta^3+\zeta^8+\zeta^{12}+\zeta^{18} \\
& & & 51 & 1+\zeta^3 - \zeta^{10} + \zeta^{15} + \zeta^{21} + \zeta^{24} + \zeta^{30} + \zeta^{33} + \zeta^{39} \\
& & & 84 & 1 - \zeta^4 - \zeta^{13} - \zeta^{16} + \zeta^{19} + \zeta^{21} + \zeta^{22} + \zeta^{31}
   \\
& & & 91 & 1 + \zeta^7 + \zeta^{14} - \zeta^{17} + \zeta^{21} + \zeta^{42} - \zeta^{69} + \zeta^{70} - \zeta^{82} \\
\hline
1 + 4 \cos^2 \tfrac{\pi}{33} & \approx 4.964 & 3 \tfrac{1}{10} & 33 & 1 + \zeta^6 - \zeta^8 + \zeta^{21} \\
1 + 4 \cos^2 \tfrac{\pi}{28} & \approx 4.950 & 3 & 28 & 1 + \zeta - \zeta^3 - \zeta^{11} \\
1 + 4 \cos^2 \tfrac{\pi}{22} & \approx 4.918 & 3 \tfrac{1}{5} & 33 & 1 + \zeta^6 - \zeta^7 - \zeta^{10} \\
1 + 4 \cos^2 \tfrac{\pi}{21} & \approx 4.911 & 3 \tfrac{1}{6} & 21 & 1 - \zeta - \zeta^5 + \zeta^{18} \\
1 + 4 \cos^2 \tfrac{\pi}{14} & \approx 4.802 & 3 \tfrac{1}{3} & 28 & 1 + \zeta - \zeta^3 + \zeta^4 \\
\tfrac{5 + \sqrt{21}}{2} & \approx 4.791 & 2 \tfrac{1}{2} & 21 & 1 - \zeta + \zeta^{6} + \zeta^{18} \\
1 + 4 \cos^2 \tfrac{\pi}{12} & \approx 4.732 & 3 & 60 & 1 - \zeta^3 - \zeta^6 - \zeta^8 \\
1 + 4 \cos^2 \tfrac{\pi}{11} & \approx 4.682 & 2 \tfrac{4}{5} & 11 & 1+\zeta+\zeta^2 + \zeta^5 \\
\multirow{3}{*}{$1 + 4 \cos^2 \tfrac{\pi}{10}$} & \multirow{3}{*}{$\approx 4.618$} & \multirow{3}{*}{$3 \tfrac{1}{2}$} & 35 & 1-\zeta^6 + \zeta^7 + \zeta^{10} +\zeta^{15} + \zeta^{17} + \zeta^{22} \\
 &  &  & 40 & 1 - \zeta^3 + \zeta^{7} + \zeta^{10} \\
 & & & 60 & 1 - \zeta^3 - \zeta^5 - \zeta^8 \\
1 + 4 \cos^2 \tfrac{\pi}{8} & \approx 4.412 & 3 & 24 & 1 + \zeta + \zeta^7 \\
\tfrac{5+\sqrt{13}}{2} & \approx 4.302 & 2 \tfrac{1}{2} & 13 & 1 + \zeta + \zeta^4 \\
1 + 4 \cos^2 \tfrac{\pi}{7} & \approx 4.247 & 2 \tfrac{2}{3} & 21 & 1 - \zeta - \zeta^4 + \zeta^{12} \\
\hline
\multirow{6}{*}{$1 + 4 \cos^2 \tfrac{\pi}{6}$} & \multirow{6}{*}{$= 4$}& \multirow{6}{*}{4} & 7 & (1 + \zeta_7 + \zeta_7^3)^2 = \zeta + \zeta^3 + \zeta^4 - \zeta^5 \\
& & & 28 & (1 + i)(1 + \zeta_7 + \zeta_7^3) \\
& & & 39 & 1- \zeta^2 - \zeta^5 - \zeta^8 - \zeta^{11} - \zeta^{20} -\zeta^{32}\\
& & & 55 & 1 - \zeta - \zeta^{16} + \zeta^{22} - \zeta^{26} - \zeta^{31} - \zeta^{36} \\
& & & 60 & (1 + i)(1 - \zeta_{15} + \zeta_{15}^{12}) \\
& & & 105 & (1 + \zeta_7 + \zeta_7^3)(1 - \zeta_{15} + \zeta_{15}^{12}) \\
\hline
 4 \cos^2 \tfrac{\pi}{14} & \approx 3.802 & 2 \tfrac{1}{3} & 21 & 1 - \zeta - \zeta^{13}  \\
\hline
\multirow{3}{*}{$1 + 4 \cos^2 \tfrac{\pi}{4}$} & \multirow{3}{*}{$= 3$} & \multirow{3}{*}{3} & 11 & 1 + \zeta + \zeta^2 + \zeta^4 + \zeta^7 \\
& & & 13 & 1 + \zeta + \zeta^3 + \zeta^9 \\
& & & 35 & 1 - \zeta + \zeta^{7} - \zeta^{11} -\zeta^{16} \\
\hline
1 + 4 \cos^2 \tfrac{\pi}{3} & = 2& 2 & 7 & 1 + \zeta + \zeta^3 \\
\end{array}
\]
\end{table}

\subsection{Additional metrics}
\label{subsec:metrics}

In preparation for a survey of prior intermediate results and a description of the proof of \Cref{T:main}, we introduce two more key measures of the ``complexity'' of a cyclotomic integer $\alpha$. Like the house, these are invariant under equivalence.
\begin{itemize}
\item
Let $\N(\alpha)$ be the smallest nonnegative integer $n$ such that $\alpha$ can be written as a sum of $n$ roots of unity (not necessarily distinct). This function does not seem to have a standard name; by analogy with the term \emph{weight} in coding theory, we refer to $\N(\alpha)$ as the \emph{minimal weight} of $\alpha$.
\item
Let $\calM(\alpha)$ be the average of $|\beta|^2$  as $\beta$ varies over conjugates of $\alpha$ in $\CC$; since $\alpha$ belongs to a CM field, this average can be interpreted as the reduced trace of $\house{\alpha}^2$, and hence is a rational number. The quantity $\calM(\alpha)$ was introduced by Cassels \cite{cassels} and is accordingly often called the \emph{Cassels height} of $\alpha$.
(If we interpret $\house{\alpha}$ as the $L^\infty$ norm with respect to the product of the complex embeddings,
then $\calM(\alpha)^{1/2}$ is the $L^2$ norm with respect to the same product.)
\end{itemize}
Comparing the definitions and applying the triangle inequality, we see at once that
\[
\N(\alpha) \geq \house{\alpha} \geq \calM(\alpha)^{1/2}.
\]
In the opposite direction, Loxton \cite{loxton} established the following lower bound on $\calM(\alpha)$ in terms of $\N(\alpha)$.
We will use some explicit versions of this result for small $n$;
see \Cref{L:short sums}.
\begin{theorem}[Loxton] \label{T:loxton}
For any $k > \log 2$, there exists an effectively computable (in terms of $k$) constant $c$ such that for all nonzero cyclotomic integers $\alpha$,
\[
\calM(\alpha) \geq c n \exp (-k \log n/ \log \log n), \qquad n := \N(\alpha).
\]
On the other hand, no such constant exists for $k = \log 2$.
\end{theorem}
\begin{proof}
For the first assertion, the corresponding inequality  with $\calM(\alpha)$ replaced with $\house{\alpha}^2$
is \cite[Theorem~1]{loxton}; however, the proof given in \cite[\S 6]{loxton} proves the stronger inequality we have asserted. (The effective computability of the constant is not remarked upon in \cite{loxton} but can be inferred from the proof; it is asserted explicitly in the proof of \cite[Theorem~5]{loxton2}.)
For the second assertion, see \cite[Theorem~2]{loxton}.
\end{proof}

The main reason to consider $\N(\alpha)$ and $\calM(\alpha)$ in addition to $\house{\alpha}$ is that they behave much more predictably with respect to decomposing an element of $\QQ(\zeta_N)$ as a sum of powers of $\zeta_N$ times elements of $\QQ(\zeta_{N/p})$ for some prime factor $p$ of $N$ (see \Cref{L:rep-prime-power-case} and \Cref{L:rep-prime-case}). We will return to this point in \S\ref{subsec:context}.

\subsection{Prior intermediate results}

We next recall some prior results in the direction of \Cref{T:main}, all of which will play roles in our proof of same.
We start with the aforementioned lemma of Cassels \cite[Lemma~6]{cassels} 
that handles the first few cases of \Cref{T:main}.
\begin{lemma}[Cassels] \label{lem:cassels low height}
Let $\alpha$ be a nonzero cyclotomic integer with $\house{\alpha}^2 < 3$.
Then $\alpha$ is either a root of unity or equivalent to one of
$1+i, 1+\zeta_5, \tfrac{-1+\sqrt{-7}}{2}, \tfrac{\sqrt{5}+\sqrt{-3}}{2}$.
\end{lemma}

We continue with a result of Jones \cite{jones} that resolves \cite[Conjecture~5]{robinson}. 
\begin{theorem}[Jones] \label{T:jones}
Let $\alpha$ be a cyclotomic integer with $\N(\alpha) \leq 3$ and $\house{\alpha}^2 \leq 5$.
Then $\alpha$ is equivalent to an element of
\begin{gather*}
\{1+\zeta_N\colon N=1,2,\dots\} \bigcup \{1+\zeta_N-\zeta_N^{-1}\colon N =1,2,\dots\} \\
\bigcup \{1+\zeta_7+\zeta_7^3, 1+\zeta_{13}+\zeta_{13}^4, 1+\zeta_{24}+\zeta_{24}^7,
1+\zeta_{30}+\zeta_{30}^{12}, 1+\zeta_{42}+\zeta_{42}^{13}\}.
\end{gather*}
\end{theorem}
As a result of \Cref{T:jones} (and part (b) of \Cref{T:robinson-wurtz}, see below),
throughout the proof of \Cref{T:main} we will be free to assume that $\N(\alpha) \geq 4$.

We next state a restricted form of a result of Calegari, Morrison and Snyder \cite[Theorem~1.0.5]{calegari-morrison-snyder}. The full statement allows $\house{\alpha} < \tfrac{76}{33}$ with a few more exceptional cases.
\begin{theorem}[Calegari--Morrison--Snyder] \label{P:totally real case}
Let $\alpha$ be a totally real cyclotomic integer with $\house{\alpha}^2 < 5.01$.
Then $\alpha$ is equivalent either to $2 \cos \frac{\pi}{N}$ for some positive integer $N$ or to one of $\frac{\sqrt{7}+\sqrt{3}}{2}$ or $\sqrt{5}$.
\end{theorem}

We finally state a result of Robinson--Wurtz \cite[Theorem~1.2]{robinson-wurtz}, which resolves \cite[Conjecture~4]{robinson}. The case $\house{\alpha}^2 \leq 4$ was treated by Robinson; see \cite[Lemma~4]{cassels}.

\begin{theorem}[Robinson--Wurtz] \label{T:robinson-wurtz}
Let $\alpha$ be a cyclotomic integer.
\begin{enumerate}
\item[(a)]
If $\house{\alpha}^2 \leq 5$, then
\begin{gather*}
\house{\alpha}^2 \in \left\{ 4 \cos^2 \tfrac{\pi}{N}\colon N = 1,2, \dots \right\} \bigcup \left\{ 1 + 4 \cos^2 \tfrac{\pi}{N}\colon N = 1,2,\dots \right\} 
\bigcup \left\{ \tfrac{5 + \sqrt{13}}{2}, \tfrac{5 + \sqrt{21}}{2} \right\}.
\end{gather*}
\item[(b)]
If $\house{\alpha}^2 \in (5, 5.04]$, then
$\house{\alpha}^2 = |1 + \zeta_{70} + \zeta_{70}^{10} + \zeta_{70}^{29}|^2 \approx 5.01766$.
\end{enumerate}
\end{theorem}
\begin{proof}
See \cite[Theorem~1.2, Theorem~1.3]{robinson-wurtz}.
\end{proof}

\subsection{Overview of the proof}
\label{subsec:context}

We now give a high-level description of the proof of \Cref{T:main}, which also serves to provide a navigation guide through the paper.
To streamline terminology, we say that a cyclotomic integer $\alpha$ is \emph{covered by \Cref{T:main}} if $\alpha$ is either of one of the forms (1)--(3) of \Cref{T:cassels} or equivalent to an element of \Cref{table:exceptional classes}. 
Our goal is then to show that every $\alpha$ with $\house{\alpha}^2 < 5.01$ is covered by \Cref{T:main};
by \Cref{T:robinson-wurtz}(b), it suffices to consider $\alpha$ with
$\house{\alpha}^2 \leq 5$.

As the formulation of \Cref{T:cassels} indicates, it was already known to Cassels that the exceptional set 
is effectively computable: it follows from the arguments in \cite{cassels} that every exceptional equivalence class is represented by an element of $\QQ(\zeta_N)$ for some particular $N$
(see the introduction to \cite{calegari-morrison-snyder} for a particular value). However, it is not straightforward to translate this bound into a \emph{feasible} finite computation; for this, we draw on ideas from Cassels and Robinson--Wurtz, and also introduce some further innovations.

We start by extracting a corollary from \Cref{T:robinson-wurtz} that guides our approach.

\begin{cor}[Robinson--Wurtz] \label{cor:RW truncation}
Let $\alpha$ be a cyclotomic integer with $\house{\alpha}^2 < 5.01$. Then one of the following holds.
\begin{enumerate}
    \item[(a)] We have $\calM(\alpha) < 3 \tfrac{1}{4}$. When this holds, in fact 
    $\calM(\alpha) \leq 3\tfrac{1}{5}$.
    \item[(b)] We have
    \begin{equation} \label{eq:rw cutoff}
(\house{\alpha}^2, \calM(\alpha)) \in \{
(1 + 4 \cos^2 \tfrac{\pi}{15}, 3 \tfrac{1}{4}), 
(1 + 4 \cos^2 \tfrac{\pi}{14}, 3 \tfrac{1}{3}),
(1 + 4 \cos^2 \tfrac{\pi}{10}, 3 \tfrac{1}{2}) 
\}.
\end{equation}
\item[(c)] We have $\house{\alpha}^2 = \calM(\alpha) = 4$.
\item[(d)] We have $\house{\alpha}^2 = \calM(\alpha) = 5$.
\end{enumerate}
\end{cor}
\begin{proof}
This follows from \Cref{T:robinson-wurtz} as follows.
Recall first that $\calM(\alpha)$ is uniquely determined by $\house{\alpha}^2$.
If $\house{\alpha}^2 \in \left\{ \frac{5 + \sqrt{13}}{2}, \frac{5 + \sqrt{21}}{2} \right\}$,
then $\calM(\alpha) = 2 \tfrac{1}{2}$.
If $\house{\alpha}^2 = 4 \cos^2 \tfrac{\pi}{N}$, then either $\calM(\alpha) = 4$ or $\calM(\alpha) \leq 3$. If $\house{\alpha}^2 = 1 + 4 \cos^2 \tfrac{\pi}{N}$,
then for $\mu$ the M\"obius function and $\varphi$ the totient function, we have
\begin{equation} \label{eq:formula for cassels height}
    \calM(\alpha) = 3 + \frac{2 \mu(N)}{\varphi(N)}
\end{equation}
(see \cite[Remark~9.0.2]{calegari-morrison-snyder} or \Cref{rem:exact castle}). In particular, $\calM(\alpha) \leq 3$ unless $N$ is squarefree with an even number of prime factors, and $\calM(\alpha) \leq 3 \tfrac{1}{5}$ unless in addition $\varphi(N) 
\leq 8$. 
We may thus take $N = 1,6,10,14,15$, yielding the listed values.
\end{proof}

For each of the cases (a)--(d) of \Cref{cor:RW truncation}, we subdivide the proof of \Cref{T:main} into three steps as follows. Here we write $N$ for the \emph{minimal level} of $\alpha$, meaning the smallest positive integer for which $\alpha$ is equivalent to an element of $\QQ(\zeta_N)$; see \S\ref{subsec:minimal level} for more details on this concept.
We also write $N_0$ and $N_1$ for specific positive integers depending only on which case we are in; see \Cref{table:main breakdown}. By a \emph{higher prime power}, we mean a prime power $p^n$ with $p$ prime and $n>1$; we refer to $p$ as its \emph{underlying prime}.

\begin{itemize}
    \item \textbf{Step 1: higher prime powers in the minimal level.} 
    In this step, we prove that $\alpha$ is covered by \Cref{T:main} whenever $N$ is divisible by a higher prime power not dividing $N_0$. (Note that the underlying prime may or may not divide $N_0$.) In other words, in cases (b)--(d), Step~1 treats all nonsquarefree
minimal levels; in case (a), Step~1 treats those minimal levels that are divisible by $8$ or by the square of an odd prime.
    \item \textbf{Step 2: large primes in the minimal level.} In this step, we prove that $\alpha$ is covered by \Cref{T:main} in all remaining cases where $N$ has a prime factor greater than $N_1$, i.e., when $N/\gcd(N,N_0)$ is squarefree and coprime to $N_0$ but has a prime factor greater than $N_1$.
    \item \textbf{Step 3: small primes in the minimal level.}
    In this step, we prove that $\alpha$ is covered by \Cref{T:main} in all remaining cases, i.e., when $N$ divides $N_0$ times the product of all primes $p \leq N_1$ coprime to $N_0$.
\end{itemize}

In case (a), the three steps are roughly parallel to three of the steps in the proof of \Cref{T:robinson-wurtz} (namely \cite[Lemma~4.1, Lemma~5.1, Lemma~6.1]{robinson-wurtz}), and 
indeed our proof amounts to carefully retracing through these steps.
\begin{itemize}
    \item 
    Step 1 is a refinement of \cite[Lemma~10]{cassels} but uses a key geometric lemma \cite[Lemma~2.5]{robinson-wurtz} to sharpen the arguments.
    \item 
     Step 2 is in turn a refinement of \cite[Lemma~11]{cassels}, in which one studies the combinatorics of Cassels heights in the decomposition against powers of $\zeta_p$ where $p$ is the largest prime factor of $N$. Compared to \cite[Lemma~5.1]{robinson-wurtz}, we shift some complexity into Step 3 (see below).
    \item 
Step 3 is a variation of Step 2 augmented by an optimized floating-point computation to exhaust over all $\alpha$ in a fixed cyclotomic field with a fixed value of $\N(\alpha)$, in order to identify candidates with castle less than 5.01. Such a computation is alluded to in the proof of \cite[Lemma 6.1]{robinson-wurtz} but no source code was preserved; we reproduce the computation in Rust (with source code available, see below) and push it somewhat further to eliminate some complexity in the theoretical arguments (including Step 2 in this case, see above).
    
\end{itemize}

This all takes place in \S\ref{sec:rw} of the paper, after some preliminaries in the earlier sections.
As a corollary, we prove \cite[Conjecture~1]{robinson} (see \Cref{cor:robinson1}) and in so doing confirm a prediction\footnote{To quote  \cite[\S 5.8.1]{mckee-smyth}: ``\cite[Conjecture~1]{robinson} seems not to have been proved, in spite of the assertion in \cite{robinson-wurtz}
that it had been. However, the calculations in that paper may prove useful for constructing a proof.''} made in \cite[\S 5.8.1]{mckee-smyth}. 

We now return to cases (b)--(d).
For Step 1, we replace the approach used in case (a)
with a  more algebraic argument taking advantage of the fact that $\house{\alpha}^2$ is known to belong to a cyclotomic field with \emph{squarefree} level; this means that when we decompose $\alpha$ as a linear combination of powers of $\zeta_{p^n}$, the induced decomposition of $\alpha \overline{\alpha}$ must have the property that the coefficient of every nonzero power of $\zeta_{p^n}$ vanishes. This implies that the set of indices with a nonzero coefficient in the decomposition of $\alpha$ has no unique differences modulo $p$, and a combinatorial argument can be used to rule this out unless $p$ is very small. 
See \S\ref{sec:combinatorial} for the algebraic argument in a general form and \S\ref{sec:higher prime powers} for its application to Step 1, treating all three cases (b), (c), (d) uniformly (and borrowing an idea from Step 3, see below).

For Step 2, we combine the approach used in case (a) with a variant of the algebraic argument, also presented in \S\ref{sec:combinatorial}; the latter serves to remove the most complicated branches of the analysis of Cassels heights.
(As in Step 1, we also streamline by borrowing an idea from Step 3, for which see below.) 
See \S\ref{sec:large primes} for this argument, broken down across cases (b), (c), (d) in succession.

For Step 3, we can ignore case (b) as it is treated uniformly with case (a). To handle cases (c) and (d), we use an argument loosely inspired by a suggestion of Noam Elkies: we pin down the possible fractional ideals of $\alpha$ in $\QQ(\zeta_N)$ and observe that they are all accounted for by known quantities. While this method in general requires computing class groups of possibly large cyclotomic fields, in this paper we are fortunate enough not to require any such computations.
See \S\ref{sec:class groups} for the argument presented in a general form (together with some lemmas used to reduce complexity in Steps 1 and 2)
and \S\ref{sec:small primes} for the application to Step 3.

\begin{table}
    \caption{Breakdown of cases and steps in the proof of \Cref{T:main}.}
    \label{table:main breakdown}

    \begin{tabular}{c|c|c|c|c|c|c}
    Case & $\calM(\alpha)$ & $N_0$ & $N_1$ & Step 1 & Step 2 & Step 3 \\
    \hline
    (a) & $< 3 \tfrac{1}{4}$ & $2^2$ & $7$ & \Cref{P:robinson wurtz extract1} & \Cref{P:robinson wurtz extract2} & \Cref{prop:list 420 part1} \\
    (b) & $3 \tfrac{1}{4}, 3 \tfrac{1}{3}, 3 \tfrac{1}{2}$ & $1$ & $7$ & \Cref{prop:Step 1} & \Cref{prop:step 2 case b} & \Cref{prop:list 420 part1}\\
    (c) & $4$ & $1$ & $11$ & \Cref{prop:Step 1} & \Cref{prop:step 2 case c} & \Cref{prop:step3 case 4}\\
    (d) & $5$ & $1$ & $13$ & \Cref{prop:Step 1} & \Cref{prop:step 2 case d} & \Cref{prop:step3 case 5}
    \end{tabular}
\end{table}

At various points, the arguments refer to some computations made in Jupyter notebooks running SageMath (specifically version 10.6); in addition, the proof of \Cref{lem:exhaust} depends on a program written in the language Rust. All source code can be found in the following GitHub repository: 
\begin{center}
    \url{https://github.com/castle-gray-rnt6/cassels}
\end{center}

\subsection{Future directions}

We expect that the techniques used here can be used to say more about the structure of small cyclotomic integers.
For instance, Calegari \cite{calegari-house} observed that the set of castles of cyclotomic integers is not closed, and asked for the smallest external limit point; the answer can be no greater than $6 \tfrac{1}{4}$ as this is the limit point of $\house{\alpha}^2$ for $\alpha = 1 +\zeta + \zeta^3 - \zeta^4$ as $\zeta$ ranges over all roots of unity. (This value is somewhat smaller than the external limit point $7.0646 \cdots$ identified in \cite{calegari-house}.)
It may in fact be feasible to classify \emph{all} cyclotomic integers with castle at most 6.25.

\section{The calculus of Cassels heights}
\label{sec:prelim}

We summarize and extend some lemmas that appear in various forms in \cite{cassels} and \cite{robinson-wurtz}.

For $K$ a number field, let $W_K$ denote the group of roots of unity in $K^\times$.

\subsection{Minimal level}
\label{subsec:minimal level}

We spell out a key definition in somewhat more detail than in some prior references.

\begin{definition} \label{def:minimal level}
For $\alpha$ a cyclotomic algebraic number, the \emph{minimal level} of $\alpha$ is the smallest positive integer $N$ such that $\alpha$ is equivalent to some element of $\QQ(\zeta_N)$; by definition, $N$ is either odd or divisible by 4.
Since $\QQ(\zeta_N)$ is a Galois extension of $\QQ$, $\alpha = \zeta \alpha_0$ for some root of unity $\zeta$ and some $\alpha_0 \in \QQ(\zeta_N)$. 
\end{definition}

\begin{definition}
For $\alpha \in \QQ(\zeta_N)$ and $p$ a prime dividing $N$, a \emph{$p$-decomposition} of $\alpha$
is a presentation
\[
\alpha = \sum_{i=0}^{p-1} \eta_i \zeta_{p^n}^i 
\]
where $n$ is the $p$-adic valuation of $N$ and $\eta_i \in \QQ(\zeta_{N/p})$ for all $i$.
When $n > 1$, the $p$-decomposition of $\alpha$ is unique; when $n=1$, the $p$-decomposition is unique up to adding a common term to each $\eta_i$.
\end{definition}

We make \cite[Exercise~5.24]{mckee-smyth} explicit.
\begin{lemma} \label{L:minimal level divides}
For any $\alpha \in \QQ(\zeta_N)$, the minimal level of $\alpha$ divides $N$.
\end{lemma}
\begin{proof}
Let $N_0$ be the minimal level of $\alpha$ and suppose by way of contradiction that $N_0$ does not divide $N$. We can then find a prime $p$ at which $N_0$ has higher $p$-adic valuation than $N$;
let $n$ denote the $p$-adic valuation of $N_0$, and set $N' := \lcm(N, N_0)$.
As indicated in \Cref{def:minimal level}, we have $\alpha = \zeta \alpha_0$ for some root of unity $\zeta$ and some $\alpha_0 \in \QQ(\zeta_{N_0})$.
Write $\zeta = \eta \zeta_{p^n}^{-j}$ for some $\eta \in \QQ(\zeta_{N'/p})$ and some $j \in \{0,\dots,p-1\}$.

Form a $p$-decomposition $\alpha_0 = \sum_{i=0}^{p-1} \eta_i \zeta_{p^n}^i$, then apply the uniqueness property to
\[
\sum_{i=0}^{p-1} \eta_i \zeta_{p^n}^{i-j}
= 
\zeta_{p^n}^{-j} \alpha_0 = \eta^{-1} (\zeta \alpha_0) \in \QQ(\zeta_{N'/p})\colon
\]
if $n > 1$ then we must have $\zeta_{p^n}^{-j} \alpha_0 = \eta_j$,
whereas if $n = 1$ then $\zeta_{p^n}^{-j} \alpha_0 = \eta_j - \eta_i$ for any $i \neq j$.
In either case we get $\zeta_{p^n}^{-j} \alpha_0 \in \QQ(\zeta_{N_0/p})$, contradicting the definition of $N_0$.
\end{proof}

The following is used frequently in \cite{calegari-morrison-snyder} and \cite{robinson-wurtz} but not stated in either.
\begin{lemma} \label{lem:sum in minimal level}
Let $\alpha$ be a cyclotomic integer of minimal level $N$. Then $\alpha$ is equivalent to a sum of $\N(\alpha)$ roots of unity in $\QQ(\zeta_N)$ (rather than a larger field).
\end{lemma}
\begin{proof}
See \cite[Theorem~5.29(a)]{mckee-smyth}.
\end{proof}

\subsection{Calculation of some minimal levels}

We next compute the minimal level of the cyclotomic integers named in \Cref{T:cassels}.
\begin{lemma} \label{lem:containment of totally real subfields}
Let $N_1, N_2$ be two positive integers such that $\QQ(\zeta_{N_1}+\zeta_{N_1}^{-1}) \subseteq \QQ(\zeta_{N_2}+\zeta_{N_2}^{-1})$. Then either $N_1$ divides $\lcm(2, N_2)$ or $N_1 \in \{1,2,3,4,6\}$ (in which case $\QQ(\zeta_{N_1}+\zeta_{N_1}^{-1}) = \QQ$).
\end{lemma}
\begin{proof}
We may assume that both $N_1$ and $N_2$ are even.
Set $N := \lcm(N_1, N_2)$, so that both
$\QQ(\zeta_{N_1}+\zeta_{N_1}^{-1})$
and $\QQ(\zeta_{N_2}+\zeta_{N_2}^{-1})$
are subfields of $\QQ(\zeta_{N}+\zeta_{N}^{-1})$.
Identify $\Gal(\QQ(\zeta_N)/\QQ)$ with $(\ZZ/N\ZZ)^\times$ via Artin reciprocity; then
\[
\Gal(\QQ(\zeta_N)/\QQ(\zeta_{N_i}+\zeta_{N_i}^{-1})) = \ker((\ZZ/N\ZZ)^\times \to (\ZZ/N_i \ZZ)^\times / \{\pm 1\}),
\]
so the condition $\QQ(\zeta_{N_1}+\zeta_{N_1}^{-1}) \subseteq \QQ(\zeta_{N_2}+\zeta_{N_2}^{-1})$ means that
\[
\ker((\ZZ/N\ZZ)^\times \to (\ZZ/N_2 \ZZ)^\times / \{\pm 1\})
\subseteq
\ker((\ZZ/N\ZZ)^\times \to (\ZZ/N_1 \ZZ)^\times / \{\pm 1\}).
\]
Hence the image of 
$H := \ker((\ZZ/N\ZZ)^\times \to (\ZZ/N_2 \ZZ)^\times)$
in $(\ZZ/N_1 \ZZ)^\times$ is either the trivial group or $\{\pm 1\}$. In the former case, $N_1$ divides $N_2$; so suppose hereafter that the latter occurs. Then $\varphi(N) = 2 \varphi(N_2)$, so either
$N = 2N_2$, or $N = 3N_2$ and $N_2$ is coprime to $3$.

If $N = 2N_2$, then $H$ is generated by $1 + N_2$, so $1 + N_2 \equiv -1 \pmod{N_1}$. Now $N_1$ divides $2(2+N_2) - N = 4$, which implies $N_1 \in \{2,4\}$.

If $N = 3N_2$ and $N_2$ is coprime to 3, then $H$ is generated by $1+N_2^2$, so $1+N_2^2 \equiv -1 \pmod{N_1}$. Now $N_1$ divides $3(2+N_2^2) - N_2 N = 6$, which implies $N_1 = 6$.
\end{proof}

\begin{cor} \label{cor:minimal level of cosine}
Let $\zeta$ be a root of unity such that $\zeta + \zeta^{-1} \in \QQ(\zeta_N)$ for some positive integer $N$. Then either $\zeta \in \QQ(\zeta_N)$ or $\zeta \in \{\zeta_3, \zeta_4, \zeta_6\}$.
\end{cor}
\begin{proof}
Since $\QQ(\zeta_N + \zeta_N^{-1})$ is the maximal totally real subfield of $\QQ(\zeta_N)$, it contains $\QQ(\zeta+\zeta^{-1})$.
The claim thus follows at once from \Cref{lem:containment of totally real subfields}.
\end{proof}

\begin{prop} \label{P:minimal level for Cassels families}
Let $N$ be a positive integer. Set
\[
N' := \begin{cases} N/2 & \textit{if } N \equiv 2 \pmod{4} \\
N & \text{otherwise.}
\end{cases}
\]
For $\alpha$ as below, the minimal level of $\alpha$ is as indicated.
\begin{enumerate}
\item[(a)] For $\alpha := 1+\zeta_N$: $1$ if $N=3$ and $N'$ otherwise.
\item[(b)] For $\alpha := 1 + \zeta_N - \zeta_N^{-1}$: $1$ if $N \in \{3,6\}$, $4$ if $N = 12$, and $N'$ otherwise.
\item[(c)] For $\alpha := (\zeta_5+\zeta_5^4) + (\zeta_5^2 + \zeta_5^3)\zeta_N$: $1$ if $N=1$ and $\lcm(5, N')$ otherwise.
\end{enumerate}
\end{prop}
\begin{proof}
For (a), we treat the cases $N \in\{3,4,6\}$ by direct calculation.
In the remaining cases, 
\Cref{L:minimal level divides} implies that 
the minimal level of $\alpha$
divides $N'$. Meanwhile, $|\alpha|^2 = 2 + \zeta_N + \zeta_N^{-1}$, so by \Cref{cor:minimal level of cosine}, $\zeta_N + \zeta_N^{-1}$ has minimal level $N'$. Hence $\alpha$ has minimal level $N'$.

For (b),  we treat the cases $N \in\{3,4,6,8,12\}$ by direct calculation.
In the remaining cases, \Cref{L:minimal level divides} implies that
the minimal level of $\alpha$ divides $N'$. Meanwhile, $|\alpha|^2 = 3 - \zeta_N^2 - \zeta_N^{-2}$;
by applying \Cref{cor:minimal level of cosine} as in the proof of (a), we get the following.

\begin{itemize}
\item For $N$ not divisible by 4 with $N \notin \{3,6\}$,
the minimal level is divisible by $N'$.

\item 
For $N$ divisible by $4$ with $N \notin \{8,12\}$,
the minimal level is divisible by
\[
N'' := \begin{cases} N/2 & 8 \mid N \\
N/4 & 8 \nmid N.
\end{cases}
\]
Equality would mean that $\zeta_N^j + \zeta_N^{j+1} - \zeta_N^{j-1} \in \QQ(\zeta_{N''})$ for some $j \in \{0,1\}$.
For $j=1$, this cannot occur as exactly two of the three summands are in $\QQ(\zeta_{N''})$.
For $j=0$, we would have $\zeta_N - \zeta_N^{-1} \in \QQ(\zeta_{N''})$ which also cannot occur for $N \neq 4$: on one hand $\zeta_N - \zeta_N^{-1} \neq 0$, and on the other hand
$\Gal(\QQ(\zeta_N)/\QQ(\zeta_{N''}))$ is generated by $\zeta_N \mapsto -\zeta_N$ which carries $\zeta_N - \zeta_N^{-1}$ to its negation.

\end{itemize}
We conclude in both cases that $\alpha$ has minimal level $N'$.

For (c), we treat the cases $N \in \{1,2,3,4,6\}$ 
by direct calculation.
In the remaining cases, 
\Cref{L:minimal level divides} implies that 
the minimal level of $\alpha$ divides $\lcm(5, N')$.
Meanwhile, $|\alpha|^2 = 3 -  \zeta_N - \zeta_N^{-1}$,
so by \Cref{cor:minimal level of cosine}, the minimal level is divisible by $N'$;
in particular, there is nothing more to check unless $N$ is coprime to 5. In that case, equality would mean that
$(\zeta_5^{1+i}+\zeta_5^{4+i}) + (\zeta_5^{2+i} + \zeta_5^{3+i})\zeta_N \in \QQ(\zeta_N)$ for some $i \in \{0,\dots,4\}$; since $[\QQ(\zeta_{5N}):\QQ(\zeta_N)] = 4$, this forces $i=0$ and $\zeta_N = 1$ by the uniqueness of $p$-decompositions.
\end{proof}

\begin{remark} \label{rem:cassels degenerate cases}
Note that in \Cref{P:minimal level for Cassels families}, in every exceptional case $\alpha$ is a sum of at most two roots of unity. For an application of this observation, see \Cref{algo:cassels test}.
\end{remark}

\subsection{Decomposition and the Cassels height: prime power case}
\label{subsec:decomposition}

We now recall the behavior of the Cassels height with respect to a $p$-decomposition. We start with the case where the minimal level is divisible by a higher power of a prime; here the geometric arguments of \cite{cassels} were somewhat systematized in \cite[Lemma~2.5]{robinson-wurtz}, which we also recall.

\begin{remark} \label{R:kronecker}
We will use frequently (often implicitly) the following consequence of Kronecker's theorem: for any nonzero cyclotomic integer $\alpha$, $\calM(\alpha) \geq 1$ with equality if and only if $\alpha$ is a root of unity (see \cite[\S 2]{cassels}). For a refinement, see \Cref{L:short sums}.
\end{remark}

\begin{lemma} \label{L:rep-prime-power-case}
Let $\alpha \in \QQ(\zeta_N)$ be a cyclotomic integer of minimal level $N$, with 
$N$ exactly divisible by the prime power $p^n$ with $n > 1$.
Form the $p$-decomposition $\alpha = \sum_{i=0}^{p-1} \eta_i \zeta_{p^n}^i$ and let $X$ be the number of indices $i$ for which $\eta_i \neq 0$.
Then 
\[
\calM(\alpha) = \sum_{i=0}^{p-1} \calM(\eta_i);
\]
hence by \Cref{R:kronecker}, $\calM(\alpha) \geq X$
with equality iff each nonzero $\eta_i$ is a root of unity.
\end{lemma}
\begin{proof}
This is \cite[(3.16)]{cassels}; see also \cite[Lemma~5.46]{mckee-smyth}.
\end{proof}

\begin{lemma}\label{conjLemma}
Suppose $\beta$ is equivalent to $\alpha + \zeta_{p^n} \gamma$, where $\alpha \in \QQ(\zeta_{M'})$ and $\gamma \in \QQ(\zeta_{M''})$. Let $m$ be the largest integer such that $\zeta_{p^m} \in \QQ(\zeta_{M'})$ or $\QQ(\zeta_{M''}
)$. Then if $m<n$,
\begin{equation} \house{\beta}^2 \geq |\alpha|^2 + |\gamma|^2 + 2 \ |\alpha| \cdot |\gamma| \cdot \cos(\theta) \end{equation}
where 
\[ 
\theta = \begin{cases}
2 \pi / p^n & \mathrm{if\ } m=0 \\
\pi / p^{n-m}& \mathrm{if\ } m>0.
\end{cases}
\]
Moreover, if $(M', M'')=1$, then
\begin{equation} \label{eq:conjLemma strong case}
\house{\beta}^2 \geq \house{\alpha}^2 + \house{\gamma}^2 + 2 \ \house{\alpha} \cdot \house{\gamma} \cdot \cos(\theta).
\end{equation}
\end{lemma}

\begin{proof}
See \cite[Lemma~2.5]{robinson-wurtz}.
\end{proof}

\subsection{Decomposition and the Cassels height: prime case}

We next turn to the case of removing an isolated prime from the minimal level.

\begin{lemma} \label{L:rep-prime-case}
Let $\alpha \in \QQ(\zeta_N)$ be a cyclotomic integer of minimal level $N$, with $N$ exactly divisible by the odd prime $p$. Choose a $p$-decomposition $\alpha = \sum_{i=0}^{p-1} \eta_i \zeta_{p}^i$ 
that minimizes the number $X$ of indices $i$ for which $\eta_i \neq 0$.
Let $S$ be the set of such indices.
\begin{enumerate}
\item[(a)]
We have $2 \leq X \leq p-1$.
\item[(b)]
We have 
\begin{equation} \label{eq:rep-prime-case}
(p-1) \calM(\alpha) = (p-X) \sum_{i \in S} \calM(\eta_i) + \sum_{i<j \in S} \calM(\eta_i - \eta_j);
\end{equation}
hence by \Cref{R:kronecker}, $\calM(\alpha) \geq \frac{X(p-X)}{p-1}$ with equality iff all of the nonzero $\eta_i$ are equal to the same root of unity (this forces $X \leq \tfrac{p-1}{2}$).
\item[(c)]
We have
\begin{equation} \label{eq:rep-prime-case extended}
\calM(\alpha) \geq \frac{p}{2} - \frac{1}{p-1} \left( \left\lfloor \frac{p}{p-X} \right\rfloor \binom{p-X}{2} + \binom{p-(p-X) \left\lfloor \frac{p}{p-X} \right\rfloor}{2} \right).
\end{equation}
(Note that the right-hand side of \eqref{eq:rep-prime-case extended} is strictly increasing as a function of $X$.)
\item[(d)]
If $\calM(\alpha) = \frac{p+3}{4}$, then $X \leq \frac{p+1}{2}$. Moreover, if equality holds, then up to equivalence, we can ensure that all but one $\eta_i$ equals $1$ and the other equals $\zeta_6$.
\item[(e)]
If $\calM(\alpha) < \frac{p+3}{4}$, then $X \leq \frac{p-1}{2}$.
\end{enumerate}
\end{lemma}
\begin{proof}
Part (a) is evident.
For (b), we follow \cite[(3.4)]{cassels}
(see also \cite[Lemma~5.43]{mckee-smyth}): write
\begin{equation} \label{eq:rep-prime-case symmetric}
(p-1) \calM(\alpha) = \sum_{0 \leq i < j \leq p-1} \calM(\eta_i - \eta_j)
\end{equation}
and sort terms based on the size of $\{i,j\} \cap S$
to obtain \eqref{eq:rep-prime-case}.

For (c), 
note that for any $p$-decomposition $\alpha = \sum_{i=0}^{p-1} \eta_i \zeta_{p}^i$, we may partition the index set $\{0,\dots,p-1\}$ according to the value of $\eta_i$; this partition is independent of the choice of the $p$-decomposition, and minimizing $X$ amounts to normalizing so that the most common value of $\eta_i$ is 0. This means that the maximum size of any part is $p-X$;
consequently, if we bound the right-hand side of \eqref{eq:rep-prime-case symmetric} from below by taking 1 for each pair $(i,j)$ in distinct parts (per \Cref{R:kronecker}), this bound is minimized by taking as many parts as possible of size $p-X$, yielding \eqref{eq:rep-prime-case extended}.

For (d) and (e), 
we follow \cite[Corollary of Lemma~1]{cassels} (see also  \cite[Lemma~5.44]{mckee-smyth}).
Note that for $X = \tfrac{p+1}{2}$, the right-hand side of \eqref{eq:rep-prime-case extended} specializes to $\tfrac{p+3}{4}$: in the language of the proof of (c), the optimal partition consists of two parts of size $X-1$ and one of size 1,
so there are $(X-1)^2 + 2(X-1) = \tfrac{(p-1)(p+3)}{4}$ pairs $(i,j)$ in distinct parts.
This yields (d) and (e) except for the second assertion of (d); for this, we note from the proof of (c) that equality in \eqref{eq:rep-prime-case extended} means that:
\begin{itemize}
    \item for every $i \in S$, $\eta_i$ is a root of unity;
    \item there is a single $j \in S$ such that for $i \neq j$, the $\eta_i$ are all equal;
    \item $\eta_i - \eta_j$ is also a root of unity.
\end{itemize}
We may assume that $\eta_i = 1$ for all $i \in S \setminus \{j\}$; then $\eta_j$ is a root of unity such that $\eta_j-1$ is also a root of unity, which forces $\eta_j \in \{\zeta_6, \zeta_6^{-1}\}$ (compare \Cref{R:too many sides of a triangle}).
\end{proof}

\begin{remark} \label{R:prime case next}
We gain more flexibility in \Cref{L:rep-prime-case} by accounting not just for the most common value 
among the $\eta_i$, but also for the second most common value using similar logic. To wit, 
let $Y$ be the number of nonzero $\eta_i$ which are \emph{not} equal to the most frequently occurring nonzero value.
Then as in \eqref{eq:rep-prime-case extended}, we obtain
\begin{equation} \label{eq:prime case next}
\sum_{i < j \in S} \calM(\eta_i - \eta_j) \geq 
\binom{X}{2} - \left( \left\lfloor \frac{X}{X-Y} \right\rfloor \binom{X-Y}{2} + \binom{X - (X-Y) \lfloor \frac{X}{X-Y} \rfloor}{2}\right).
\end{equation}
Again, the right-hand side is strictly increasing as a function of $Y$
and bounded below by $Y(X-Y)$, with equality when $Y \leq \tfrac{X}{2}$. In particular,
\begin{gather} \label{eq:prime case next Y0}
    \sum_{i<j \in S} \calM(\eta_i - \eta_j) < X-1 \Longrightarrow Y=0 \\
 \label{eq:prime case next Y1}
 \sum_{i<j \in S} \calM(\eta_i - \eta_j) < 2X-2 \Longrightarrow Y \leq 1.    
\end{gather}
\end{remark}

\begin{remark} \label{R:too many sides of a triangle}
We observe from the proof of \Cref{L:rep-prime-case}(d) that
\eqref{eq:rep-prime-case extended} cannot be sharp when $X > \tfrac{2p}{3}$: if $\eta_i, \eta_j, \eta_k$ are distinct roots of unity and $\calM(\eta_i-\eta_j) = \calM(\eta_j-\eta_k) = 1$, then up to equivalence we have $\eta_i = 1, \eta_j = \zeta_6, \eta_k = \zeta_6^{-1}$ and hence $\calM(\eta_i-\eta_k) = 3$.
Similarly, \eqref{eq:prime case next} cannot be sharp when $Y > \tfrac{X}{2}$.
\end{remark}

\subsection{Lower bounds on the Cassels height}
\label{subsec:bounds-cassels-height}

Recall from \Cref{T:loxton} that the Cassels height $\calM(\alpha)$ can be effectively bounded below in terms of the minimum weight $\N(\alpha)$.
We pick up a theme from \cite{cassels} and make this explicit for a few small values.

\begin{lemma} \label{L:short sums}
Let $\alpha$ be a cyclotomic integer.
\begin{enumerate}
\item[(a)]
If $\N(\alpha) \geq 1$, then $\calM(\alpha) \geq 1$ with equality iff $\alpha$ is a root of unity.
\item[(b)]
If $\N(\alpha) \geq 2$, then $\calM(\alpha) \geq 1\tfrac{1}{2}$ with equality iff $\alpha$ is equivalent to $1+\zeta_5$.
Moreover, if equality fails, then one of the following occurs:
\begin{itemize}
\item $\calM(\alpha) = 1 \tfrac{2}{3}$ and $\alpha$ is equivalent to $1+\zeta_7$;
\item $\calM(\alpha) = 1 \tfrac{3}{4}$ and $\alpha$ is equivalent to $1+\zeta_{30}$;
\item $\calM(\alpha) = 1 \tfrac{4}{5}$ and $\alpha$ is equivalent to $1+\zeta_{11}$;
\item $\calM(\alpha) = 1 \tfrac{5}{6}$ and $\alpha$ is equivalent to $1+\zeta_{13}$ or $1+\zeta_{42}$;
\item
or $\calM(\alpha) \geq 1\tfrac{7}{8}$.
\end{itemize}
\item[(c)]
If $\N(\alpha) \geq 3$, then $\calM(\alpha)\geq 2$ with equality iff $\alpha$ is equivalent to $1+\zeta_7 +\zeta_7^i$ for some $i \in \{2,3\}$ or $\zeta_3 - \zeta_5 - \zeta_5^i$ for some $i \in \{2,4\}$. Moreover, if equality fails, then $\calM(\alpha) \geq 2 \tfrac{1}{4}$.
\item[(d)]
If $\N(\alpha) \geq 4$, then $\calM(\alpha) \geq 2\tfrac{1}{2}$
with equality iff
$\alpha$ is equivalent to
$(1+\zeta_5)(1+\zeta_7)$ or
$-\zeta_3 + \zeta_7^i + \zeta_7^j + \zeta_7^k$ for some $1 \leq i < j < k \leq 6$.
\end{enumerate}
\end{lemma}
\begin{proof}
Part (a) is a restatement of \Cref{R:kronecker}.
For (b), apply \cite[Lemma 2]{cassels} and \cite[Lemma~2.2]{robinson-wurtz} (see also \eqref{eq:formula for cassels height}).
For (c), apply \cite[Lemma 3]{cassels} (or \cite[Lemma~5.47]{mckee-smyth})
and \cite[Theorem~9.0.1]{calegari-morrison-snyder}.

For (d), suppose that $\alpha$ is a counterexample of minimal level $N$;
using \Cref{L:rep-prime-power-case}, we see that $N$ is odd and squarefree. Now set notation as in \Cref{L:rep-prime-case}, taking $p$ to be the largest prime factor of $N$ (if it exists). 
\begin{itemize}
\item 
Suppose that $p \geq 11$. Since $\tfrac{p+3}{4} > 2 \tfrac{1}{2}$, by \Cref{L:rep-prime-case}(e) we have $X \leq \tfrac{p-1}{2}$.
By (b), (c), and \Cref{L:rep-prime-case}(b),
if $X \geq 4$, then $\calM(\alpha) \geq \tfrac{4(p-4)}{p-1} \geq 2 \tfrac{4}{5}$;
if $X = 3$, then $\calM(\alpha) \geq \tfrac{3.5(p-3)}{p-1} \geq 2 \tfrac{4}{5}$;
if $X = 2$, then $\calM(\alpha) \geq \tfrac{3(p-2)}{p-1} \geq 2 \tfrac{7}{10}$.
\item
Suppose that $p = 7$.
By \Cref{L:rep-prime-case}(d)(e), if $\calM(\alpha) \leq 2 \frac{1}{2}$, then $X\leq 4$ with equality only if
$\alpha$ is equivalent to $-\zeta_3 + \zeta_7^i + \zeta_7^j + \zeta_7^k$ for some $1 \leq i < j < k \leq 6$.
If $X = 3$, then $\calM(\alpha) \geq \tfrac{3.5(p-3)}{p-1} + \tfrac{2}{p-1} > 2 \tfrac{1}{2}$.
If $X = 2$, then $\calM(\alpha) \geq \tfrac{3(p-2)}{p-1} = 2 \tfrac{1}{2}$ with equality iff $\alpha$ is equivalent to $(1+\zeta_5)(1+\zeta_7)$.
\item
Suppose that $p = 5$.
Any $\eta \in \ZZ[\zeta_3]$ with $\calM(\eta) > 1$ in fact satisfies $\calM(\eta) \geq 3$
because $\calM(\eta) = |\eta|^2$ (see also \Cref{lem:cassels low height}).
If $X = 4$, then we have equality in \eqref{eq:rep-prime-case extended} which is impossible by \Cref{R:too many sides of a triangle}.
If $X = 3$, then we must have $\N(\eta_i) > 1$ for some $i$;
now \Cref{L:rep-prime-case}(b) implies
$10 = 2\sum_{i \in S} \calM(\eta_i) + \sum_{i<j \in S} \calM(\eta_i-\eta_j)$, but the inequalities $\sum_{i \in S} \calM(\eta_i) \geq 1+1+3=5$
and $\sum_{i<j \in S} \calM(\eta_i-\eta_j) \geq 0$ cannot be sharp at the same time.
If $X=2$, then $\calM(\eta_0) + \calM(\eta_1) \geq 4$
but \Cref{L:rep-prime-case}(b) implies
$10 = 3(\calM(\eta_i) + \calM(\eta_j)) + \calM(\eta_i-\eta_j) \geq 12$.
\item 
Suppose that $N \in \{1,3\}$. By \cite[Lemma~7.0.3]{calegari-morrison-snyder} we must have $N = 1$,
but then $\N(\alpha) \geq 4$ implies $\calM(\alpha) = |\alpha| \geq 4$.
\qedhere
\end{itemize}
\end{proof}

\begin{remark} \label{rem:more short sums}
While \Cref{L:short sums} suffices for the proof of \Cref{T:main}, 
it would be of interest for other applications to push this result further.
For example, Dvornicich--Zannier have showed \cite[Corollary~1]{dvornicich-zannier}
that one can use an effective version of \Cref{T:loxton} to solve equations of the form $P(x_1,\dots,x_n,y) = 0$ where $P$ is a polynomial over a number field $K$; $x_1,\dots,x_n$ are constrained to be roots of unity; and $y$ is constrained to generate a cyclotomic extension of $K$. In the special case $K = \QQ$, we may see this by writing $\house{y} \leq c$ where $c$ is the supremum of $|y|$ over all $(x_1,\dots,x_n,y) \in \CC^{n+1}$ with $|x_1| = \cdots = |x_n| = 1$ and $P(x_1,\dots,x_n,y) = 0$; using \Cref{T:loxton} to deduce a bound on $\N(y)$ then yields a polynomial equation in a fixed number of roots of unity, for which there are various algorithms available.

With this as motivation, define $f(n)$ to be the infimum of $\calM(\alpha)$ over all $\alpha$ with $\N(\alpha) \geq n$; then \Cref{L:short sums} asserts that
\[
f(1) = 1, \qquad f(2) = 1 \tfrac{1}{2}, \qquad f(3) = 2,  \qquad f(4) = 2 \tfrac{1}{2}.
\]
We give some upper bounds for $f(n)$ for small $n$ in \Cref{table:upper bounds}. 
The implicitly claimed values of $\N(\alpha)$ can be verified using \Cref{rem:determine minimal length}.
\begin{table}[ht]
\caption{Some upper bounds on $f(n)$.}
\label{table:upper bounds}
\begin{tabular}{c|c|c}
$n$ & $\calM(\alpha)$ & $\alpha$ with $\N(\alpha) \geq n$\\
\hline
1 & 1 & $1$ \\
2 & $1 \tfrac{1}{2}$ & $1+\zeta_5$ \\
3 & 2 & $1 + \zeta_7 + \zeta_7^2$, $\zeta_3 - \zeta_5 - \zeta_5^2$ \\
4 & $2 \tfrac{1}{2}$ & $(1+\zeta_5)(1 + \zeta_7)$, $\zeta_3 - \zeta_7 - \zeta_7^2 - \zeta_7^3$ \\
\hline
5 & $3$ & $(1+\zeta_5)(1+\zeta_7) + \zeta_7^2, 1 + \zeta_{11} + \zeta_{11}^2 + \zeta_{11}^3 + \zeta_{11}^4$ \\
6 & $3$ & $(1 + \zeta_5)(1 + \zeta_{7} + \zeta_7^2)$  \\
7 & $3 \tfrac{1}{2}$ & $1-\zeta_{35}^6 + \zeta_{35}^7 + \zeta_{35}^{10} +\zeta_{35}^{15} + \zeta_{35}^{17} + \zeta_{35}^{22}$ \\
8 & $3 \tfrac{3}{4}$ & $(1+\zeta_5)(\zeta_3 - \zeta_7 - \zeta_7^2 - \zeta_7^3)$ \\
9 & $4$ & $(1+\zeta_7+\zeta_7^2)(\zeta_3-\zeta_5-\zeta_5^2)$ \\ 
10 & $4 \tfrac{1}{2}$ & $(1+\zeta_5)(1+\zeta_{11}+\zeta_{11}^2+\zeta_{11}^3+\zeta_{11}^4)$ \\
11 & $5 \tfrac{1}{8}$ & $\zeta_3 - (1+\zeta_5)(\zeta_{11} + \zeta_{11}^2 + \zeta_{11}^3 + \zeta_{11}^4 + \zeta_{11}^5)$
\end{tabular}
\end{table}
\end{remark}

\section{Computational methods}

We describe some computational methods relevant to our work.

\subsection{Computing the basic quantities}

We describe some basic algorithms for handling cyclotomic integers. Implementations of these as SageMath functions can be found in the file \verb+utils.sage+ in the GitHub repository.

\begin{algorithm} \label{algo:minimal level}
Given a positive integer $N$ and a cyclotomic integer
$\alpha \in \ZZ[\zeta_N]$, compute the minimal level $N_0$ of $\alpha$ and a root of unity $\zeta$ such that $\alpha \zeta \in \QQ(\zeta_{N_0})$.
\begin{enumerate}
    \item For each prime divisor $p$ of $N$ and each $i \in \{0,\dots,p-1\}$, check whether $\alpha \zeta_p^i \in \QQ(\zeta_{N/p})$. If so, recurse on the input $N/p, \alpha \zeta_p^i$ to obtain the output $N_0', \zeta'$, then return $N_0', \zeta_p^i \zeta'$.
    \item
    Otherwise, return $N, 1$.
\end{enumerate}
\end{algorithm}
\begin{proof}
By \Cref{L:minimal level divides}, $N_0$ is a divisor of $N$. If $N_0 = N$, then step 1 will fail to find a prime $p$ and so the output $N, 1$ is correct. If $N_0 \neq N$, then step 1 will find some prime divisor $p$ of $N/N_0$, so termination and correctness can be deduced by induction on $N$.
\end{proof}

\begin{remark}  \label{rem:exact castle}
Given a positive integer $N$ and a cyclotomic integer
$\alpha \in \ZZ[\zeta_N]$, we may compute the castle
$\house{\alpha}^2$ by first computing the minimal polynomial
$P(T)$ of $\alpha$ over $\QQ$, then computing $|\beta|^2$ for each root of $P(T)$ in $\CC$ and taking the maximum. In SageMath, this can be done rigorously using the object \texttt{AlgebraicField}, which gives an exact representation of the subfield of algebraic numbers in $\CC$ using interval arithmetic.

Computing the Cassels height $\calM(\alpha)$ is somewhat easier: we may simply write
\[
\calM(\alpha) = \frac{1}{\varphi(N)} \Trace_{\QQ(\zeta_N)/\QQ}(\alpha \overline{\alpha}).
\]
\end{remark}

\begin{remark} \label{rem:minimal level}
Using \Cref{algo:perfect hash} and \Cref{rem:exact castle}, we may verify that every entry of \Cref{table:exceptional classes} is presented at minimal level and that the castles
and heights are listed correctly.
This makes it apparent that no two entries are equivalent, as they are distinguished pairwise by the combination of castle and minimal level.
\end{remark}

\begin{algorithm} \label{algo:perfect hash}
Given a positive integer $N$ and a cyclotomic integer
$\alpha \in \ZZ[\zeta_N]$, compute a perfect hash function for equivalence of cyclotomic integers. That is, $\alpha$ and $\beta$ are equivalent iff this function returns the same output on both.
\begin{enumerate}
    \item Apply \Cref{algo:minimal level} to return a pair $N_0, \zeta$. 
    \item For each $\zeta' \in W_{\QQ(\zeta_{N_0})}$, compute the minimal polynomial of $\alpha \zeta \zeta'$ over $\QQ$.
    \item
    Sort the results of the previous step first in increasing order by degree, then in lexicographic order by coefficients from highest to lowest degree.
    \item
    Return the first entry of the resulting list.
\end{enumerate}
\end{algorithm}
\begin{proof}
    By construction, $\alpha$ is equivalent to every root of the return value; consequently, if $\alpha$ and $\beta$ yield the same output then they are equivalent.
    Conversely, if $\alpha$ and $\beta$ are equivalent, then they have the same minimal level; hence they yield identical lists in step 3, and hence identical outputs.
\end{proof}

\begin{remark}
    By \cite[Lemma~8]{mckee-oh-smyth}, for $\alpha \in \ZZ[\zeta_N]$ nonzero, there exists some $i \in \ZZ$ such that $\zeta_N^i \alpha$ has nonzero trace. It follows that for the polynomial $P(T)$ returned by \Cref{algo:perfect hash}, the coefficient of $T^{\deg(P)-1}$ is always nonzero.
\end{remark}

We do not know whether in \Cref{algo:perfect hash}, skipping the step of rewriting $\alpha$ using its minimal level yields the same result. We formulate this as a problem.
    
\begin{problem}
Let $\alpha$ be a cyclotomic integer which has minimal degree over $\QQ$ among the elements of its equivalence class.
Must we have $\alpha \in \QQ(\zeta_N)$ for $N$ the minimal level of $\alpha$?
\end{problem}

We next address \cite[Problem 2]{robinson}.
\begin{algorithm} \label{algo:minimal weight}
Given a positive integer $N$ and a cyclotomic integer
$\alpha \in \ZZ[\zeta_N]$, compute the minimal weight $\N(\alpha)$ and a sequence $\eta_1,\dots,\eta_{\N(\alpha)} \in W_{\QQ(\zeta_N)}$ such that
$\alpha = \sum_{i=1}^{\N(\alpha)} \eta_i$.
\begin{enumerate}
    \item Set $n = 0$.
    \item For each $n$-element multisubset $\eta_1,\dots,\eta_n$ of $W_{\QQ(\zeta_N)}$,
    check whether $\sum_{i=1}^{n} \eta_i = \alpha$;
    if so, return $n, (\eta_1,\dots,\eta_n)$.
    \item Replace $n$ by $n+1$, then go to step 2.
\end{enumerate}
\end{algorithm}
\begin{proof}
It suffices to verify that there exists a minimal-weight representation of $\alpha$ in $\ZZ[\zeta_N]$. By \Cref{lem:sum in minimal level}, such a representation exists in $\ZZ[\zeta_{N'}]$ where $N'$ is the minimal level of $\alpha$.
By \Cref{L:minimal level divides}, $N'$ divides $N$ and so we also obtain a representation in $\ZZ[\zeta_N]$.
\end{proof}

\begin{remark} \label{rem:determine minimal length}
\Cref{algo:minimal weight} scales so poorly that we were unable to 
confirm that all of the entries in \Cref{table:exceptional classes} are presented with minimum weight.
See also \Cref{rem:more short sums}.

It is suggested in \cite[\S 5.7.2]{mckee-smyth} that \Cref{T:loxton} can be applied to obtain an alternate algorithm, but this would require making the statement completely explicit. Moreover, to confirm that $\N(\alpha) \geq n$ in the typical case where the Loxton bound is not sharp, one still needs to verify somehow that $\alpha$ admits no representation as a sum of $n-1$ roots of unity;
a new idea to improve upon exhaustion for this step would be very helpful.

One easy improvement is available when $N$ is not squarefree:
with notation as in \Cref{L:rep-prime-power-case}, we have
\[
\N(\alpha) = \sum_{i=0}^{p-1} \N(\eta_i).
\]
By contrast, when $p$ exactly divides $N$, there need not exist any representation as in \Cref{L:rep-prime-case} for which $\N(\alpha) = \sum_{i=0}^{p-1} \N(\eta_i)$: one might have to use a representation with \emph{more} nonzero $\eta_i$ to minimize the right-hand side. However, any representation with $X \leq \tfrac{p-1}{2}$ does satisfy  $\N(\alpha) = \sum_{i=0}^{p-1} \N(\eta_i)$:
for any nonzero $\beta \in \ZZ[\zeta_N]$,
\begin{align*}
\sum_{i=0}^{p-1} \N(\eta_i + \beta)  
&\geq \sum_{i \in S} (\N(\eta_i) - \N(\beta)) + \sum_{i \notin S} \N(\beta) \\
&= (p-2X) \N(\beta) + \sum_{i \in S} \N(\eta_i)
> \sum_{i \in S} \N(\eta_i).
\end{align*}
\end{remark}

We next give an algorithm to separate out the exceptional classes for \Cref{T:cassels}.

\begin{algorithm} \label{algo:cassels test}
Given a positive integer $N$ and a cyclotomic integer
$\alpha \in \ZZ[\zeta_N]$, return True if and only if $\alpha$ has one of the forms (1)--(3) of~\Cref{T:cassels}.
\begin{enumerate}
    \item Run \Cref{algo:minimal weight} (with an early abort) to determine whether or not $\N(\alpha) \leq 2$. If so, return True.
    \item Check whether the polynomial $T^2 - (\alpha \overline{\alpha}-3) T + 1$
    admits a root $\eta \in W_{\QQ(\zeta_N)}$. If so, fix a choice of $\eta$, otherwise return False.
    \item Check whether $\eta$ has a square root $\eta'$ in $\QQ(\zeta_N)$, and if so whether $\alpha$ and $1+\eta'-(\eta')^{-1}$ have the same hash via \Cref{algo:perfect hash}. If these both hold, return True.
    \item If $5 \mid N$, check whether $\alpha$ and $(\zeta_5+\zeta_5^4)+(\zeta_5^2+\zeta_5^3)\eta$ have the same hash via \Cref{algo:perfect hash}. If so, return True.
    \item 
    Return False.
\end{enumerate}
\end{algorithm}
\begin{proof}
It is clear that any return value of True is correct.
We thus need to prove conversely that if $\alpha$ is of one of the forms (1)--(3) of \Cref{T:cassels}, then the value True will be returned. This is obvious for form (1), as then we return True in step 1.

If $\alpha$ is of the form (2) of \Cref{T:cassels}, then
$\alpha$ equals a root of unity times $1 + \zeta - \zeta^{-1}$ for some root of unity $\zeta$.
By \Cref{L:minimal level divides},
\Cref{P:minimal level for Cassels families},
and \Cref{rem:cassels degenerate cases}, 
we must have $\zeta \in \QQ(\zeta_N)$.
By writing
\[
\alpha \overline{\alpha} = (1 + \zeta - \zeta^{-1})(1 + \zeta^{-1} - \zeta) = 3 - 2 (\zeta^2 + \zeta^{-2}),
\]
we deduce that $\zeta^2$ is a root of the polynomial $T^2 - (\alpha \overline{\alpha}-3) T + 1$.
It is now clear that we obtain a result of True in step 3.

If $\alpha$ is of the form (3) of \Cref{T:cassels}, then
$\alpha$ equals a root of unity times $(\zeta_5+\zeta_5^4)+(\zeta_5^2+\zeta_5^3)\zeta$ for some root of unity $\zeta$. 
By \Cref{L:minimal level divides},
\Cref{P:minimal level for Cassels families},
and \Cref{rem:cassels degenerate cases}, 
we must have $5 \mid N$ and $\zeta \in \QQ(\zeta_N)$.
By writing
\[
\alpha \overline{\alpha} = ((\zeta_5+\zeta_5^4)+(\zeta_5^2+\zeta_5^3)\zeta)((\zeta_5+\zeta_5^4)+(\zeta_5^2+\zeta_5^3)\zeta^{-1})
=  3 - 2 (\zeta + \zeta^{-1}),
\]
we deduce that $\zeta$ is a root of the polynomial $T^2 - (\alpha \overline{\alpha}-3) T + 1$.
It is now clear that we obtain a result of True in step 4.
\end{proof}

We now confirm that every entry of \Cref{table:exceptional classes} is in fact necessary for \Cref{T:main}.
\begin{cor} \label{cor:no redundant exceptions}
Each entry $\alpha$ of \Cref{table:exceptional classes} is a cyclotomic integer with $\house{\alpha}^2 \leq 5$
but not of one of the forms (1)--(3) of \Cref{T:cassels}. Moreover, no two entries of \Cref{table:exceptional classes} are pairwise equivalent.
\end{cor}
\begin{proof}
The first assertion may be checked using \Cref{rem:exact castle}. The second assertion may be checked using \Cref{algo:cassels test}. The third inclusion is included in \Cref{rem:minimal level}.
All of these checks can be reproduced at once using the SageMath file \verb+create-table.sage+ in the GitHub repository.
\end{proof}

\subsection{Exhaustion over short sums}

The proof of \cite[Lemma~6.1]{robinson-wurtz} involves an exhaustion step to identify all $\alpha$ with bounded minimal level and weight and small castle. 
We reimplement and extend that step here.
\begin{lemma} \label{lem:exhaust}
Let $\alpha$ be a cyclotomic integer with $\house{\alpha}^2 < 5.01$
such that for some pair
\begin{gather*}
(N,n) \in \{ 
(2^2 \times 3 \times 5 \times 7, 7),  
(31, 6), 
(3 \times 5 \times 7 \times 13, 5), 
(2^2 \times 3 \times 5 \times 7 \times 11, 5), 
\\
(5 \times 19, 4), 
(5 \times 17, 4), 
(2^2 \times 3 \times 5 \times 7 \times 11 \times 13, 4), 
(2^3 \times 3^2 \times 5 \times 7, 4) 
\},
\end{gather*}
the minimal level of $\alpha$ divides $N$ and $\N(\alpha) \leq n$.
Then $\alpha$ is covered by \Cref{T:main}.
(Specializing to $(N,n) = (2^2 \times 3 \times 5 \times 7, 6)$ recovers the computational result asserted in \cite[Lemma~6.1]{robinson-wurtz}.)
\end{lemma}
\begin{proof}
Set $N' = \lcm(2, N)$; by \Cref{lem:sum in minimal level}, we can write $\alpha = \sum_{i=1}^{n'} \zeta_{N'}^{j_i}$ with $n' = \N(\alpha) \leq n$ for some
$0 \leq j_1, \ldots, j_{n'} < N'$.
As the claim is apparent when $\alpha$ is equivalent to an element of $\ZZ$, we may assume that the $j_i$ are not all equal. We may also assume that $n' \geq 3$, as otherwise $\alpha$ is covered by \Cref{T:cassels}(1).

Without loss of generality, we may make some further assumptions on the indices $j_i$:
\begin{itemize}
    \item $j_1 = 0$
(by multiplying $\alpha$ by a root of unity);
    \item $\gcd(j_i, N') \geq \gcd(j_2, N')$ for $i=3,\dots,n'$
(by relabeling);
    \item $j_2 = d$ is a proper divisor of $N'$ (by applying a field automorphism and using that the $j_i$ are not all equal);
    \item $j_3 \leq \cdots \leq j_{n'}$ (by relabeling);
    \item $N' - j_{n'} > j_3-j_2$ (by applying $j_i \mapsto d-j_i \pmod{N'}$ if needed);
    \item there  are no indices $i < i'$ with $j_{i'}-j_i$ a nonzero multiple of $N'/2$ or $N'/3$ (as otherwise $n' > \N(\alpha)$);
    \item there are no indices $i<i'<i''$ with $j_{i'}-j_i, j_{i''}-j_{i'}$ nonzero multiples of $N'/5$ (as otherwise $n' > \N(\alpha)$);
    \item there are no indices $i<i'<i''<i'''$ with $j_{i'}-j_i, j_{i''}-j_{i'},j_{i'''}-j_{i''}$ nonzero multiples of $N'/7$ (as otherwise $n' > \N(\alpha)$);
    \item if $n' = 3$, then there is no permutation $k_1, k_2, k_3$ of $j_1, j_2, j_3$ such that $k_1+k_2 \equiv 2k_3 + N'/2$ (as otherwise $\alpha$ is covered by \Cref{T:cassels}(2));
    \item if $n' = 4$ and $5 \mid N$, then there is no permutation $k_1,k_2,k_3, k_4$ of $j_1,j_2,j_3,j_4$ such that 
    $k_1-k_2, k_3-k_4 \equiv 0 \pmod{N'/5}$ and $k_1-k_2 + k_3-k_4, k_1-k_2-k_3+k_4 \not\equiv 0 \pmod{N'}$
    (as otherwise $\alpha$ is covered by \Cref{T:cassels}(3)).
\end{itemize}
For each tuple of indices satisfying these conditions, we have 
\begin{equation} \label{eq:house as maximum}
\house{\alpha}^2 = \mathop{\max_{0 \leq k \leq N'}}_{\gcd(k,N')=1} \left\{\left( \sum_{i=1}^{n'} \cos \frac{2 \pi k j_i}{N'} \right)^2 + \left( \sum_{i=1}^{n'} \sin \frac{2 \pi k j_i}{N'} \right)^2 \right\};
\end{equation}
it thus suffices to verify the claim in all  those cases where this expression is at most 5.01.

To do this, we first compute a table of values of $\cos \tfrac{2 \pi j}{N'}$ and $\sin \tfrac{2 \pi j}{N'}$ for $j=0,\dots,N'-1$
using double precision floating-point arithmetic in Rust. By exporting these values to SageMath and using interval arithmetic there, we may confirm that each value is accurate to within $10^{-14}$.

We then compute \eqref{eq:house as maximum} using Rust floating point arithmetic for all tuples  $(j_1,\dots,j_{n'})$ satisfying the conditions listed above, recording those for which the computed answer is at most 5.1;
this takes under 1 hour on a laptop (Intel Core Ultra 5 135U, 14 cores @4.4GHz).
Since Rust floating-point arithmetic is IEEE compliant, 
we can easily bound the error in the result (say by $10^{-10}$), which guarantees that all cases with $\house{\alpha}^2 \leq 5.01$ pass through this filter.

Switching to SageMath, for each $\alpha$ represented by a tuple returned by Rust, we recompute $\house{\alpha}^2$ as an element of $\QQ(\zeta_{N}+\zeta_N^{-1})$ and check rigorously whether or not each of its $\RR$-conjugates is less than 5.01 (\Cref{rem:exact castle}). Whenever this occurs, we verify that $\alpha$ is either of one of the forms (1)--(3) of \Cref{T:cassels} or equivalent to an entry of \Cref{table:exceptional classes}.

All Rust and SageMath source code can be found in the GitHub repository.
For the latter, see the Jupyter notebooks \verb|floating-point.ipynb| and \verb|check-table-1.ipynb|.
\end{proof}

\begin{remark}
Note that in the proof of \Cref{lem:exhaust}, we cannot replace the assumption $\gcd(j_i, N') \geq d$ by the stronger assumption $d \mid j_i$.
For a concrete example, take $N = 3 \times 5 \times 7, n'=4, (j_1,\dots,j_4) = (0,2,4,7)$.
\end{remark}

\begin{remark}
As part of the same computation, we verified that if $5 < \house{\alpha}^2 \leq 5.04$ and $(N,n)$ is as in \Cref{lem:exhaust}, then $\alpha$ is equivalent to the value $1+\zeta_{70} + \zeta_{70}^{10} + \zeta_{70}^{29}$ listed in \Cref{T:robinson-wurtz}(b).
\end{remark}

\begin{remark}
As in~\cite{robinson-wurtz}, we can infer somewhat more information from the floating-point computation: for $\alpha, \alpha'$ two cyclotomic integers of ``not too large'' degree, the numerical equality of the computed approximations of $\house{\alpha}^2, \house{\alpha'}^2$ is sufficient to guarantee the exact equality of these quantities. This is stated formally in~\cite[Lemma~2.6]{robinson-wurtz}; however, there are a number of apparent misprints in the result's statement and proof. Most notably, the claim that $\delta$ is a cyclotomic integer seems to be predicated on the false assertion that for $\alpha$ a cyclotomic integer, $|\alpha|$ is also a cyclotomic integer (whereas for example $\sqrt{1+4 \cos^2 \tfrac{\pi}{N}}$ is not a cyclotomic integer for general $N$).
We include a corrected version for future reference.
\end{remark}

\begin{lemma}[\cite{robinson-wurtz}, Lemma~2.6]
    Suppose $\beta$ and $\gamma$ are two cyclotomic integers with $\house{\beta}^2, \house{\gamma}^2 \leq 5 + 1/25$. Choose a positive integer $N$ with $\beta, \gamma \in \QQ(\zeta_N)$ and set $k:=[\QQ(\zeta_N):\QQ]=\varphi(N)$. If $|\house{\beta}^2 - \house{\gamma}^2|<(10+2/25)^{-k}$, then $\house{\beta}^2=\house{\gamma}^2$.
    \begin{proof}
        By replacing $\gamma$ with a Galois conjugate if necessary, we can ensure that the cyclotomic integer $\delta:= \beta \overline{\beta}-\gamma \overline{\gamma}$ maps to $\house{\beta}^2 - \house{\gamma}^2$ via some complex embedding.
        Denote the conjugates of $\delta$ in $\CC$ by $\delta_1,\dots,\delta_i$ where $i \leq k$ and $\delta_1=\house{\beta}^2 - \house{\gamma}^2$. 
        Since the conjugates of $\beta\overline{\beta}$ and $\gamma\overline{\gamma}$ have magnitude at most $5+1/25$,
        by the triangle inequality we have $|\delta_i| \leq (10+2/25)^k$ for $i=2,\dots,k$.
        Then $|\mathrm{Norm}(\delta)|=|\delta_1 \cdots \delta_i|\leq |\delta_1| (10+2/25)^{k-1}<1$. As $\delta\in \ZZ[\zeta_N]$, we have $|\mathrm{Norm}(\delta)|<1$ if and only if $\mathrm{Norm}(\delta)=0=\delta$, so $\house{\beta}^2=\house{\gamma}^2$.
    \end{proof}
\end{lemma}

\section{Revisiting Robinson--Wurtz}
\label{sec:rw}

We next retrace the steps of the proof of \cite[Theorem~1.3]{robinson-wurtz} in order to resolve case (a) of the proof of the \Cref{T:main} (see \Cref{T:cassels partial classification}), as set out in \Cref{table:main breakdown}. As a corollary, we then deduce \cite[Conjecture~1]{robinson}, thus completing the resolution of all of the conjectures made in \cite{robinson}. We also record another corollary for later use (\Cref{lem:cor of partial classification}).

\subsection{Step 1}

The following is a refinement of \cite[Lemma~4.1]{robinson-wurtz}.

\begin{prop} \label{P:robinson wurtz extract1}
Let $\alpha$ be a cyclotomic integer with $\house{\alpha}^2 < 5.01$ and $\calM(\alpha) < 3 \tfrac{1}{4}$.
Assume further that the minimal level $N$ of $\alpha$ is divisible by a prime power $p^n$ with $n>1$ and $p^n > 4$.
Then $\alpha$ is covered by \Cref{T:main}.
\end{prop}
\begin{proof}
 We assume at once that $\alpha \in \QQ(\zeta_N)$.
 By \Cref{T:jones}, we may also assume that $\N(\alpha) > 3$,
 as otherwise $\alpha$ is covered by \Cref{T:main}.

    As in~\Cref{L:rep-prime-power-case}, form the unique $p$-decomposition $\alpha = \sum_{j=0}^{p-1}\eta_j \zeta_{p^n}^j$ and let $X$ denote the number of indices $j$ for which $\eta_j \neq 0$.
    We have $2 \leq X \leq 4$, but:
    \begin{itemize}

        \item $X\geq 4$ would imply $\calM(\alpha)\geq 4$, contradicting $\calM(\alpha) < 3 \tfrac{1}{4}$.

        \item $X=3$ would imply that $\N(\eta_j) > 1$ for some $j$ (since we are assuming $\N(\alpha) > 3$).
        By \Cref{L:short sums}(a,b) this would imply $\calM(\alpha) \geq 1 + 1 + \tfrac{3}{2}$, contradicting $\calM(\alpha) < 3 \tfrac{1}{4}$.

    \end{itemize}
Hence $X = 2$, and we may assume without loss of generality that $\alpha= \eta_0 +\zeta_{p^n}\eta_1$ with $0 < \N(\eta_0) \leq \N(\eta_1)$. Since $\calM(\eta_0)+\calM(\eta_1)=\calM(\alpha) < 3\tfrac{1}{4}= \tfrac{13}{4}$,
we must have $\min\{\calM(\eta_0), \calM(\eta_1)\} < \tfrac{13}{8}$; by \Cref{L:short sums}(b), this implies $\N(\eta_0) \leq 2$. We now distinguish cases based on $\N(\eta_0)$.

\subsubsection{Case A: $\N(\eta_0) = 1$}
We can assume without loss of generality that $\eta_0=1$. By writing $\calM(\eta_1) = \calM(\alpha) - \calM(\eta_0) = \calM(\alpha)-1$, we deduce that
        $\calM(\eta_1) < \tfrac{9}{4}$. On the other hand, since we are assuming $\N(\alpha) > 3$, we must have $\N(\eta_1) \geq \N(\alpha) - \N(\eta_0) \geq 3$; hence by \Cref{lem:cassels low height},
        either $\house{\eta_1}^2 = 2$ or $\house{\eta_1}^2 \geq 3$. We treat these two subcases separately.

            \begin{itemize}

                \item \underline{\textbf{Subcase A.1: $\house{\eta_1}^2 \geq 3$}:} 
                We set up an application of  \Cref{conjLemma} with $M' = 1, M'' = N/p$. 
                The largest integer $m$ such that $\zeta_{p^m} \in \QQ(\zeta_{M'})$ or $\QQ(\zeta_{M''})$ is then the $p$-adic valuation of $M''$, which is $n-1$; consequently, $m > 0$ and $\gcd(M',M'') = 1$.
                We are thus in the situation where \eqref{eq:conjLemma strong case} applies with $\theta = \frac{\pi}{p}$, yielding 
                  \[
                    \house{\alpha}^2 \geqslant 1 + \house{\eta_1}^2 + 2
                    \house{\eta_1}\cos(\theta).
                  \]

                  \begin{itemize}

                    \item If $p \geq 3$, we directly compute that $\house{\alpha}^2 > 5\frac 1 {25}$, which contradicts our hypothesis.

                    \item If $p = 2$, then by assumption $n>2$. Then, because $\eta_1 \in \QQ(\zeta_{N/2})$, we have a decomposition $\eta_1=\gamma^\prime+\zeta_{2^{n-1}}\gamma^{\prime\prime}$ for $\gamma^{\prime}, \gamma^{\prime\prime} \in \QQ(\zeta_{N/4})$. We have 
                    $$
                    \calM(\gamma^\prime)+\calM(\gamma^{\prime\prime}) = \calM(\eta_1) = \calM(\alpha)-1 < \tfrac{9}{4}.
                    $$
                    By \Cref{L:short sums}(a,b), this implies that either both $\gamma^{\prime}$ and $\gamma^{\prime\prime}$ are roots of unity or one of them is zero. In the latter case, we have $\alpha = 1 + \zeta_{2^n}^i(\gamma' + \gamma'')$ for some \emph{odd} integer $i$; this implies that $\zeta_{2^n}^i$ is primitive.
                    We can now apply \Cref{conjLemma} again, this time taking $m = n-2$ and hence $\theta = \frac \pi 4$. Since $\cos \theta = \frac {\sqrt 2} 2$, we obtain $\house{\alpha}^2 \geqslant 4 + \sqrt 6$, which contradicts the hypothesis.

                  \end{itemize}

                \item \underline{\textbf{Subcase A.2: $\house{\eta_1}^2=2$}:} By \Cref{lem:cassels low height}, $\eta_1$ is equivalent to either $1+\zeta_7+\zeta_7^3$ or $\zeta_3-\zeta_5-\zeta_5^{-1}$. We subdivide further based on the value of $p^n$.
                
                \begin{enumerate}
    
                    \item $p^n \notin \{2^3, 3^2\}$. In this case,  \Cref{conjLemma} applies with $\theta \leq \tfrac{\pi}{5}$. By \eqref{eq:conjLemma strong case} we have $\house{\alpha}^2 \geq 3 + 2 \sqrt{2} \cos \tfrac{\pi}{5} \approx 5.28825$, a contradiction.

                    \item $p^n=3^2$ and $\eta_1 \equiv 1+\zeta_7+\zeta_7^3$. This implies $\house{\alpha}^2>5.04$.
                    
                    \item $p^n=3^2$ and $\eta_1 \equiv \zeta_3-\zeta_5-\zeta_5^{-1}$. In this case, $\alpha$ is equivalent to the form $1+\zeta_{3^2}^i\cdot \zeta_m^l \cdot (\zeta_3^j-\zeta_5^k-\zeta_5^{-k})$ for some $m$ not divisible by $3^2$. By changing $i$, we may assume that $m$ is not even divisible by 3.
                    If $m$ is divisible by a prime or prime power $q$ not dividing $2^3 \times 5$, 
                    we may apply \Cref{conjLemma} as in (1) with $p^n$ replaced by $q$ to obtain a contradiction. 
                    We may thus assume that
                    $m$ divides $2^3 \times 5$; since $\N(\alpha) \leq 4$, we can apply \Cref{lem:exhaust} to the pair $(2^3 \times 3^2 \times 5 \times 7, 4)$ to ensure that $\alpha$ is covered by \Cref{T:main}.
                    
                    \item $p^n=2^3$ and $\gamma \equiv 1+\zeta_7+\zeta_7^3$ or $\zeta_3-\zeta_5-\zeta_5^{-1}$. Thus $\alpha \equiv 1+ \zeta_{2^3}^i\zeta_m^j\gamma^\prime$;
                    as in the previous case, we may reduce to the case where $m$ divides $3^2 \times 5 \times 7$.
                    Since $\N(\alpha) \leq 4$, we can apply \Cref{lem:exhaust} to the pair $(2^3 \times 3^2 \times 5 \times 7, 4)$ to ensure that $\alpha$ is covered by \Cref{T:main}.
                    
                \end{enumerate}
            \end{itemize}

\subsubsection{Case B: $\N(\eta_0) > 1$}
        We cannot have $\N(\eta_1) > 2$, as otherwise \Cref{L:short sums} would imply
        $\tfrac{13}{4} > \calM(\alpha) = \calM(\eta_0) + \calM(\eta_1) \geq \tfrac{3}{2} + 2$.
        Hence $\N(\eta_0) = \N(\eta_1) = 2$, and we may now assume that $\calM(\eta_0) \leq \calM(\eta_1)$. By \Cref{L:short sums}(b),  $\calM(\eta_0) = \tfrac{3}{2}$ and $\calM(\eta_1) \in \{\tfrac{3}{2},\tfrac{5}{3}\}$.

        \begin{itemize}

            \item \underline{\textbf{Subcase B.1: $\calM(\eta_1)= \tfrac{5}{3}$ }:} $\calM(\eta_1)=\tfrac{5}{3}$ implies $\eta_1 \equiv 1+\zeta_7$ and $\eta_0 \equiv 1+\zeta_5$, and again \Cref{conjLemma} implies $\house{\alpha}^2>5.04$.

            \item \underline{\textbf{Subcase B.2:  $\calM(\eta_1)=\tfrac{3}{2}$}:} Both $\eta_0$ and $\eta_1$ are equivalent to $1+\zeta_5$, so $\alpha$ is equivalent to $1+\zeta_5+\varrho(\zeta_5^i+\zeta_5^j)$ for some root of unity $\varrho$. If $i-j \equiv 1$ or 4 mod 5, we may apply \Cref{conjLemma} as in the proof of \cite[Lemma~4.1]{robinson-wurtz} to obtain $\house{\alpha}^2 \geq (2 + \sqrt{2}) \tfrac{3 + \sqrt{5}}{2}\approx 8.93853$. Otherwise, $\alpha$ is equivalent to $(1+\zeta_5)+(\zeta_5^2+\zeta_5^3)\zeta$ for some root of unity $\zeta$, and this outcome is covered by \Cref{T:main}. \qedhere
        \end{itemize}

\end{proof}

\subsection{Step 2}

The following is a refinement of  \cite[Lemma~5.1]{robinson-wurtz}.
However, using \Cref{lem:exhaust} we can reduce the complexity of the case analysis in the proof.

\begin{prop} \label{P:robinson wurtz extract2}
Let $\alpha$ be a cyclotomic integer with $\house{\alpha}^2 < 5.01$ and $\calM(\alpha) < 3 \tfrac{1}{4}$.
Assume further that the minimal level $N$ of $\alpha$
divides $4$ times an odd squarefree integer,
and that $N$ is divisible by some prime $p > 7$. Then $\alpha$ is covered by \Cref{T:main}.
\end{prop}
\begin{proof}
We may again assume that $\alpha \in \QQ(\zeta_N)$ and (by~\Cref{T:jones}) $\N(\alpha) > 3$.
Further, by \Cref{L:minimal level divides} and our hypothesis, we also have $\alpha \notin \QQ(\zeta_{420})$. Under these hypotheses, we will show that $\alpha$ is covered by \Cref{lem:exhaust}, is equivalent either to $(1+\zeta_5) + (\zeta_5^2 + \zeta_5^3)\zeta$ for some root of unity $\zeta$, or is equivalent to an entry of Table 1. 

Let $p \geq 11$ be the largest prime factor of $N$.
Since $\calM(\alpha) < \tfrac{13}{4} < \tfrac{p+3}{4}$,
by \Cref{L:rep-prime-case} we have a  $p$-decomposition
$\alpha = \sum_{i=0}^{p-1} \eta_i \zeta_{p}^i$
where the number $X$ of indices $i$ for which $\eta_i \neq 0$ is at most $\tfrac{p-1}{2}$. 
Let $\alpha_1, \dots, \alpha_X$ denote the nonzero values of $\eta_i$ sorted so that $0 < \N(\alpha_1) \leq \cdots \leq \N(\alpha_X)$. From \Cref{L:rep-prime-case}(b), we have the inequality:
\begin{equation}\label{eq:maininequality}
\tfrac{13}{4}(p-1) > (p-X) \sum_{i=1}^X \calM(\alpha_i) + \sum_{1\leq i<j\leq X} \calM(\alpha_i - \alpha_j) =: S
\end{equation}
(here we override the definition of $S$ in \Cref{L:rep-prime-case}).
We now proceed by considering different combinations of $p$ and $X$,
compiling lower bounds on $S$ based on \Cref{L:short sums}(a--d) and the inequality $\N(\alpha) \leq \sum_{i=1}^X \N(\alpha_i)$.

\subsubsection{Case: $p=11$} 
In this case, $2 \leq X \leq \tfrac{p-1}{2} = 5$
and \eqref{eq:maininequality} asserts that $S < \tfrac{65}{2} = 32 \tfrac{1}{2}$.
By applying \Cref{lem:exhaust} to the pair $(2^2 \times 3 \times 5 \times 7 \times 11, 5)$, we may assume that $\N(\alpha) \geq 6$. We obtain a contradictory lower bound on $S$ by tabulating cases:

    \begin{center}
    \setlength{\tabcolsep}{8pt}
    \begin{tabular}{ccccc|ccccc|c}
    \hline\hline
    \multicolumn{5}{c}{$\N(\alpha_i)$} & \multicolumn{5}{c}{$\calM(\alpha_i)$} & $S$ \\
    \hline
    $1, 2$ & $\geq 4$ & & & & $1$ & $5/2$ & & & & $31 \tfrac{1}{2} + 1^{*}$ \\
    $\geq 3$ & $\geq 3$ & & & & $2$ & $2$ & & & & $36$ \\
    \hline
    $1$ & $1$ & $\geq 4$ & & & $1$ & $1$ & $5/2$ & & & $36$ \\
    $1$ & $\geq 2$ & $\geq 3$ & & & $1$ & $3/2$ & $2$ & & & $36$ \\ 
    $\geq 2$ & $\geq 2$ & $\geq 2$ & & & $3/2$ & $3/2$ & $3/2$ & & & $36$ \\
    \hline
    $1$ & $1$ & $1$ & $\geq 3$ & & $1$ & $1$ &$1$ & $2$ & & $35$ \\
    $\geq 1$ & $\geq 1$ & $\geq 2$ & $\geq 2$ & & $1$ & $1$ & $3/2$ & $3/2$ & & $35$ \\
    \hline
    $\geq 1$ & $\geq 1$ & $\geq 1$ & $\geq 1$ & $\geq 2$ & $1$ & $1$ & $1$ & $1$ & $3/2$ & $33$ \\
    \hline\hline
    \end{tabular}
    \end{center}

* We get the additional 1 by noting that $\alpha_1 \neq \alpha_2$, so $\calM(\alpha_1-\alpha_2) \geq 1$.

\subsubsection{Case: $p=13$} 
In this case, $2 \leq X \leq \tfrac{p-1}{2} = 6$ and \eqref{eq:maininequality} asserts that $S < 39$.
By applying \Cref{lem:exhaust} to the pair $(2^2 \times 3 \times 5 \times 7 \times 11 \times 13, 4)$,
we may assume that $\mathcal{N}(\alpha) \geq 5$. For $X =5,6$, \Cref{L:rep-prime-case}(b)
gives $S \geq 5 \cdot 8 = 40$.
For $X \leq 4$, we obtain a contradictory lower bound on $S$ by tabulating cases:

    \begin{center}
    \setlength{\tabcolsep}{8pt}
    \begin{tabular}{cccc|cccc|c}
    \hline\hline
    \multicolumn{4}{c}{$\N(\alpha_i)$} & \multicolumn{4}{c}{$\calM(\alpha_i)$} & $S$ \\
    \hline
    1 & $\geq 4$ & & & $1$ & $5/2$ & & & $42$ \\
    $\geq 2$ & $\geq 3$ & & & $3/2$ & $2$ & & & $42$ \\
    \hline
    $1$ & $1$ & $\geq 3$ & & $1$ & $1$ & $2$ & & $40$ \\
    $1$ & $\geq 2$ & $\geq 2$ & & $1$ & $3/2$ & $3/2$ & & $40$ \\
    \hline
    $1$ & $1$ & $1$ & $\geq 2$ & $1$ & $1$ & $1$ & $3/2$ & $40 \tfrac{1}{2}$ \\
    \hline\hline
    \end{tabular}
    \end{center}

For $p \geq 17$, we continue to distinguish subcases based on $X$, but now arguing uniformly with respect to $p$.
Remember that we may assume $\N(\alpha) \geq 4$.

\subsubsection{Case: $X=2$ and $p\geq 17$}
By \eqref{eq:maininequality},
\begin{align*}
\tfrac{13}{4}\tfrac{p-1}{p-2} & > \calM(\alpha_1)+\calM(\alpha_2)+ \tfrac{1}{p-2} \calM(\alpha_1- \alpha_2).
\end{align*}
For $p \geq 17$, we have $\tfrac{p-1}{p-2} \leq \tfrac{16}{15}$, and therefore from the above inequality, we get
\begin{equation}\label{eq:inequality17,2}
\calM(\alpha_1)+\calM(\alpha_2)<\tfrac{52}{15}\approx 3.4666.
\end{equation}
In particular, if $\N(\alpha_1) > 1$, then by \Cref{L:short sums}(b,c), $\calM(\alpha_2) < \tfrac{52}{15} - \tfrac{3}{2} < 2$ and hence $\N(\alpha_1) = \N(\alpha_2) = 2$.
We may thus assume without loss of generality that $\calM(\alpha_1)\leq \calM(\alpha_2)$; then $\calM(\alpha_1)\leq \tfrac{52}{30}$, so by \Cref{L:short sums}(a,b) we have
\begin{equation*}
    \calM(\alpha_1)=1, \, \tfrac{3}{2}, \, \text{or} \, \tfrac{5}{3}.
\end{equation*}
We consider each possible value of $\calM(\alpha_1)$ below:
\begin{itemize}
\item $\calM(\alpha_1)=1$. Then $\N(\alpha_1)=1$, that is, $\alpha_1$ is a root of unity. We also have $\N(\alpha)>3$, hence $\N(\alpha_2)>2$. Applying \Cref{L:short sums}(c), we get $\calM(\alpha_2)\geq 2$ and hence
$$
\house{\eta_2}\geq \calM(\alpha_2)^{\tfrac{1}{2}}\geq \sqrt{2}.
$$ 
As $p \geq 17$ exactly divides $N$, \Cref{conjLemma} implies that 
$$
\house{\alpha}^2 \geq 3 + 2\sqrt{2} \cos \tfrac{2\pi}{17} \approx 5.6374 >5.01.
$$

\item $\calM(\alpha_1)=\tfrac{3}{2}$. 
As noted above, $\N(\alpha_2) = 2$ and so  $\alpha_2$ is equivalent to $1+\zeta_n$ for some $n \geq 4$.
 If $n=5$, then both $\alpha_1$ and $\alpha_2$ are equivalent to $1+ \zeta_5$, so $\alpha$ is equivalent to $1+\zeta_5+\varrho(\zeta_5^i+\zeta_5^j)$ for some root of unity $\varrho.$ Now, if $i-j \equiv 1$ or $4 \pmod 5$, then $\alpha$ is equivalent to $\alpha_1+ \zeta_p \alpha_2$ with $|\alpha_1|=|\alpha_2|=\tfrac{1+ \sqrt{5}}{2}$;
 applying \Cref{conjLemma} with $\theta \leq \tfrac{2 \pi}{17}$, we have $\house{\alpha}^2 > 10.11855 >5.01$. If $i-j \equiv 2$ or 3 $\pmod 5$, then $\alpha$ is equivalent to $(1+\zeta_5)+(\zeta_5^2+\zeta_5^3)\zeta$ for some root of unity $\zeta$.

 If $n$ is coprime to $5$, then $n \geq 4$ implies
 $$
 \house{1 + \zeta_n}=4 \cos^2 \tfrac{\pi}{n} \geq \sqrt{2}.
 $$ 
 Since $\alpha_2 \equiv 1 + \zeta_n \in \QQ(\zeta_{N/p})$, applying \Cref{conjLemma} with 
 $\theta \leq \tfrac{2 \pi}{17}$ and $\house{1 + \zeta_5}= \tfrac{1+ \sqrt{5}}{2}$
 yields $\house{\alpha}^2 \geq 8.88548$.

If $5 |  n$ and $n \neq 5$, then $n \geq 10$. We may conjugate $\alpha$ so that 
$$
|1 + \zeta_n|=\house{1 + \zeta_n} \geq \house{1 + \zeta_{10}}= \sqrt{(5+ \sqrt{5})/2}.
$$
The conjugate of $1+\zeta_5$ with smallest absolute value is $1 + \zeta_5^2$ with 
$|{1 + \zeta_5^2}|=\tfrac{\sqrt{5}-1}{2}$.
Applying \Cref{conjLemma} with $\theta \leq \tfrac{2 \pi}{17}$, we get $\house{\alpha}^2 \geq 6.19237$.

\item $\calM(\alpha_1)=\tfrac{5}{3}$.
Again, $\N(\alpha_2) = 2$ and so  $\alpha_2$ is equivalent to $1+\zeta_n$ for some $n \geq 4$.
If $n=7$, then both $\alpha_1$ and  $\alpha_2$ are equivalent to $1+ \zeta_7$.
Among the conjugates of $1+\zeta_7$, the ones with the smallest absolute values are $1+\zeta_7^3$ and $1+\zeta_7^4$,
and the next smallest is $1+\zeta_7^2$.
Since $\alpha$ is equivalent to both $1+\zeta_7+\varrho(1+\zeta_7^i)$ and $1+\zeta_7^2+\varrho(1+\zeta_7^{2i})$ for some $i \in \{1,\dots,6\}$ and some root of unity $\varrho$, we can ensure that neither $\alpha_1$ nor $\alpha_2$
has the smallest absolute value (taking the first option if $i \neq 3,4$ and the second option otherwise). 
We now have $|\alpha_1|, |\alpha_2| \geq |1+ \zeta_7^2|$,
so by \Cref{conjLemma} we conclude: 
$$
\house{\alpha}^2 \geq 6.00983.
$$

If $n$ is coprime to $7$, then the argument from the case $\calM(\alpha_1)=\tfrac{3}{2}$ carries over after noting that $\house{1+\zeta_7}>\house{1+\zeta_5}$.

 If $7|n$ and $n \neq 7$, then $n \geq 14$. 
 We may conjugate $\alpha$ so that 
 $$
 |1 + \zeta_n|=\house{1 + \zeta_n} \geq \house{1 + \zeta_{14}} \approx 1.949856.
 $$ 
 The conjugate of $1+\zeta_7$ with smallest absolute value is $1 + \zeta_7^3$ with $|{1 + \zeta_7^3}|\approx 0.44504$.
 Applying \Cref{conjLemma} with $\theta \leq \tfrac{2 \pi}{17}$, we get $\house{\alpha}^2 \geq 5.61756$.
\end{itemize}

\subsubsection{Case: $X=3$ and $p\geq 17$}

Since $\sum_{i=1}^3 \N(\alpha_i) \geq \N(\alpha) > 3$, we must have $\N(\alpha_3) > 1$ and hence $\calM(\alpha_3) \geq \tfrac{3}{2}$ by \Cref{L:short sums}(b).
If $\N(\alpha_2) > 1$, then $S > (p-3)(1+\tfrac{3}{2}+\tfrac{3}{2}) = 4(p-3)$ and this exceeds $\tfrac{13}{4}(p-1)$ for $p \geq 17$. Hence $\N(\alpha_2) = 1$ and in particular $\alpha_1-\alpha_3, \alpha_2-\alpha_3$ are both nonzero.
Combining this logic with \eqref{eq:maininequality}, we obtain
\begin{align*}
\tfrac{13}{4} (p-1) - \tfrac{7}{2}(p-3) - 2 &> 
(p-3)(\calM(\alpha_3)-\tfrac{3}{2}) + \calM(\alpha_1-\alpha_2) \\
&+ (\calM(\alpha_1-\alpha_3)-1)+ (\calM(\alpha_2-\alpha_3)-1) \geq 0
\end{align*}
and the outer inequality fails for $p \geq 23$. For $p=17,19$ the left-hand side of the inequality is at most 1;
using \Cref{L:short sums}, we deduce that $\alpha_3$ is equivalent to $1+\zeta_5$
(because $(p-3)(\tfrac{5}{3} - \tfrac{3}{2}) \geq 1$); $\alpha_1 = \alpha_2$; and $\N(\alpha_1-\alpha_3) = 1$
(because $\calM(\alpha_1-\alpha_3)=\calM(\alpha_2-\alpha_3) < \tfrac{3}{2}$).
Hence $\alpha$ is equivalent to $1+\zeta_p + \zeta_p^i (1+\zeta_5)$
for some $i$, so we may apply \Cref{lem:exhaust} to the pair $(5 \times p, 4)$ to ensure that $\alpha$ is covered by \Cref{T:main}.

\subsubsection{Case: $X\geq 4$ and $p\geq 17$}
From~\eqref{eq:maininequality},
\begin{equation}\label{inequalityforlarge}
    \tfrac{13}{4}(p-1)>(p-X)X,
\end{equation}
and we have $X\leq \tfrac{p-1}{2}$. 
However, the difference between the left-hand and right-hand sides of \eqref{inequalityforlarge} equals 0 for $p=17, X=4$; is strictly decreasing in $p$ for $X=4, p \geq 17$; and is strictly decreasing in $X$ for $p \geq 17$, $4 \leq X \leq \tfrac{p-1}{2}$. This difference therefore cannot be positive, which contradicts~\eqref{inequalityforlarge}.
\end{proof}

\subsection{Step 3}

The following is a refinement of \cite[Lemma~6.1]{robinson-wurtz}, but with a weaker restriction on $\calM(\alpha)$ which will allow us to reuse this result in case (b) of \Cref{table:main breakdown}.
This is made possible by the fact that \Cref{lem:exhaust}
goes beyond the corresponding computation in \cite{robinson-wurtz}.

\begin{prop} \label{prop:list 420 part1}
Let $\alpha \in \QQ(\zeta_{420})$ be a cyclotomic integer with $\house{\alpha}^2 < 5.01$ and $\calM(\alpha) \leq 3 \tfrac{1}{2}$.
Then $\alpha$ is covered by \Cref{T:main}.
\end{prop}
\begin{proof}
If $\N(\alpha) \leq 7$, we deduce the claim by  applying \Cref{lem:exhaust} to the pair
$(2^2 \times 3 \times 5 \times 7, 7)$.
It will thus suffice to prove that there is no cyclotomic integer $\alpha \in \QQ(\zeta_{420})$ with
$\house{\alpha}^2 < 5.01$, $\calM(\alpha) \leq 3 \tfrac{1}{2}$, and $\N(\alpha) > 7$.
(Note that there is no need to apply \Cref{T:jones} in this argument.)

Choose a representation $ \alpha = \sum_{i=0}^4 \zeta_5^i\eta_i$ where $\eta_i\in \QQ(\zeta_{84})$
and the number $X$ of nonzero $\eta_i$ has been minimized, as in \Cref{L:rep-prime-case} except with  $X=1$ allowed (but not $X=5$).
Let $\alpha_1, \dots, \alpha_X$ be the nonzero values of $\eta_i$ sorted so that $\N(\alpha_1) \leq \cdots \leq \N(\alpha_X)$.

It is known that $\beta\in \mathbb Q(\zeta_{84})$ satisfies the following:
\begin{itemize}
    \item If $\N(\beta)=2$, then $\calM(\beta)\ge 5/3$, by \Cref{L:short sums}(b) because $1+\zeta_5 \notin \QQ(\zeta_{84})$.
    \item If $\N(\beta)=3$, then $\calM(\beta)\ge 2$, by \Cref{L:short sums}(c).
    \item If $\N(\beta)=4$, then $\calM(\beta)\ge 5/2$, by \Cref{L:short sums}(d).
    \item If $\N(\beta)= 5$, then $\calM(\beta)\ge 17/6$, by \cite[Lemma 7.0.8]{calegari-morrison-snyder}.
    \item If $\N(\beta)\ge 6$, then $\calM(\beta)\ge 23/6$, by \cite[Lemma 7.0.9]{calegari-morrison-snyder}.
\end{itemize}
In each case, we demonstrate a contradiction to \eqref{eq:rep-prime-case} with $p=5$ and $\calM(\alpha) \leq 3\frac{1}{2}$:
\begin{align}\label{eq:bound_cassel_height}
    14 \geq (5-X)\sum_{i=1}^X \calM(\alpha_i)+\sum_{1\le i<j\le X}\calM(\alpha_i-\alpha_j):=S.
\end{align}
As in the proof of \Cref{P:robinson wurtz extract2}, we present the lower bounds contradicting \eqref{eq:bound_cassel_height} in tabular form. The column $\calM(\alpha_i-\alpha_j)$ is listed in the order $\alpha_1-\alpha_2, \dots, \alpha_1-\alpha_X, \alpha_2-\alpha_3, \dots$.

\subsubsection{Case: $X = 1$}

We may assume that $ \alpha = \alpha_1 \in \mathbb{Q}(\zeta_{84}) $. We write  $\alpha = \beta + \zeta_4 \gamma$ with $\beta, \gamma \in \mathbb{Q}(\zeta_{21})$. By \Cref{L:rep-prime-power-case}, we have
$$
\mathcal{M}(\alpha) = \mathcal{M}(\beta) + \mathcal{M}(\gamma) \leq 3\tfrac{1}{2},
$$
so we may assume without loss of generality that $ \mathcal{M}(\beta) \leq \tfrac{7}{4}$.
By \Cref{L:short sums}, $\N(\beta) \leq 2$ and so $\N(\gamma) \geq 6$. By \cite[Lemma~7.0.5]{calegari-morrison-snyder}, $\calM(\gamma) \geq \tfrac{23}{6}$
and now $\calM(\alpha) \geq 1 + \tfrac{23}{6}$, a contradiction to \eqref{eq:bound_cassel_height}.

\subsubsection{Case: $X = 2$}
We tabulate cases:

    \begin{center}
    \setlength{\tabcolsep}{8pt}
    \begin{tabular}{cc|cc|c|c}
    \hline\hline
    \multicolumn{2}{c}{$\N(\alpha_i)$} & \multicolumn{2}{c}{$\mathcal{M}(\alpha_i)$} & $\mathcal{M}(\alpha_1 - \alpha_2)$ & $S$ \\
    \hline
    $1,2$ & $\ge 6$ & $1$ & $23/6$ & $5/2$ & $17$ \\
    $3$ & $\ge 5$ & $2$ & $17/6$ & $5/3$ & $16\frac{1}{6}$ \\
    $\geq 4$ & $\geq 4$ & $5/2$ & $5/2$ & 0 & $15$ \\
    \hline\hline
    \end{tabular}
    \end{center}

\subsubsection{Case: $X = 3$}

At most one of the differences $\alpha_j-\alpha_k$ can vanish: otherwise $\alpha_1 = \alpha_2 = \alpha_3$ and there is another representation with $X < 3$
given by taking $\eta_i'=\eta_i-\alpha_1$ for all $i$.
We may further assume that among representations with $X =3$, we have chosen one to maximize the number of $i$ with $\N(\alpha_i) = 1$.
We now tabulate cases:

\begin{center}
    \setlength{\tabcolsep}{8pt}
    \begin{tabular}{ccc|ccc|ccc|c}
\hline\hline
\multicolumn{3}{c}{$\N(\alpha_i)$} & \multicolumn{3}{c}{$\mathcal{M}(\alpha_i)$} & \multicolumn{3}{c}{$\mathcal{M}(\alpha_i - \alpha_j)$} & $S$ \\
\hline
1 & 1 & $\ge 6$ & 1 & 1 & $23/6$ & 0 & 17/6 & $17/6$ & $17\frac{1}{3}$ \\
1 & 2 & $\ge 5$ & 1 & $5/3$ & $17/6$ & 1 & $5/2$ & $2$ & $16 \frac{1}{2}$ \\
1 & $\ge 3$ & $\ge 4$ & 1 & 2 & 5/2 & $5/3$ & $2$ & 0 & $14\frac{2}{3}$ \\
2 & 2 & $\ge 4$ & $5/3$ & $5/3$ & 5/2 & 0 & 5/3 & 5/3 & $15$ \\
$\ge 2$ & $\ge 3$ & $\ge 3$ & $5/3$ & 2 & 2 & 1 & 1 & 0 & $13\frac{1}{3} + 1^*$ \\
\hline\hline
\end{tabular}
\end{center}
* The lower bounds on $\calM(\alpha_i-\alpha_j)$ cannot all be sharp: otherwise we could rewrite $\alpha = \sum_{i=0}^4 \eta_i' \zeta_5^i$ with $\eta_i' := \eta_i - \alpha_2$;
the nonzero $\eta'_i$ would be $\alpha_1-\alpha_2, \alpha_2, \alpha_2$, increasing the number of roots of unity from 0 to 1. Hence either $\alpha_2 \neq \alpha_3$ and $\calM(\alpha_2-\alpha_3) \geq 1$, or $\calM(\alpha_1-\alpha_2) = \calM(\alpha_1-\alpha_3) \geq \tfrac{5}{3}$.

\subsubsection{Case: $X = 4$}
No difference $\alpha_j-\alpha_k$ can vanish: otherwise, there is another
representation with $X<4$ given by taking $\eta_i'=\eta_i-\alpha_j$ for all $i$. 

If there are at most three pairs $i<j$ with $\calM(\eta_i-\eta_j) \leq 1$, then $S \geq 4 + 7 \times \tfrac{5}{3} > 14$.
Otherwise, there must be an index $i$ such that
$\calM(\eta_i - \eta_j) = 1$ for at least two different $j$.
By replacing $\eta_i$ with $\eta_i' := \eta_i - \alpha_j$, we may reach the situation where $\alpha_1$ and $\alpha_2$ are both roots of unity.
We now tabulate cases:

\begin{center}
\begin{tabular}{cccc|cccc|cccccc|c}
\hline\hline
\multicolumn{4}{c}{$\N(\alpha_i)$} & \multicolumn{4}{c}{$\mathcal{M}(\alpha_i)$} & \multicolumn{6}{c}{$\mathcal{M}(\alpha_i - \alpha_j)$} & $S$ \\
\hline
1 & 1 & 1 & $\ge 5$ & 1 & 1 & 1 & $17/6$ & 1 & 1 & 5/2 & 1 & 5/2 & 5/2 & $16\frac{1}{3}$ \\
1 & 1 & 2 & $\ge 4$ & 1 & 1 & $5/3$ & 5/2 & 1 & 1 & $2$ & 1 & $2$ & 5/3 & $14 \frac{5}{6}$ \\
1 & 1 & $\ge 3$ & $\ge 3$ & 1 & 1 & 2 & 2 & 1 & $5/3$ & $5/3$ & $5/3$ & $5/3$ & 1 & $14\frac{2}{3}$ \\
\hline\hline
\end{tabular} 
\end{center}

This completes the proof.
\end{proof}

\subsection{Putting things together}

We assemble the three steps into a partial result towards \Cref{T:main}. 

\begin{theorem}[\Cref{T:main}, case (a) of \Cref{table:main breakdown}] \label{T:cassels partial classification}
Let $\alpha$ be a cyclotomic algebraic integer with $\house{\alpha}^2 < 5.01$ and $\calM(\alpha) < 3 \tfrac{1}{4}$. Then $\alpha$ is covered by \Cref{T:main}.
\end{theorem}
\begin{proof}
By \Cref{T:robinson-wurtz} and 
\Cref{cor:RW truncation}
we have $\house{\alpha}^2 < 5$ and $\calM(\alpha) \leq 3 \tfrac{1}{5}$.
Let $N$ be the minimal level of $\alpha$. 
If $N$ is divisible by a higher prime power $p^n > 4$, then
we deduce the claim by applying \Cref{P:robinson wurtz extract1}.
Otherwise, $N$ divides $4N_0$ for some odd squarefree $N_0$;
if $N_0$ is divisible by some prime $p > 7$, then
we deduce the claim by applying \Cref{P:robinson wurtz extract2}. Otherwise, we can take $\alpha \in \QQ(\zeta_{420})$ and apply \Cref{prop:list 420 part1}.
\end{proof}

In light of the fact that the values of $\house{\alpha}^2$  listed in \Cref{cor:RW truncation}(b)--(d) are all at least 4, we deduce the following corollary which resolves \cite[Conjecture~1]{robinson}.

\begin{cor} \label{cor:robinson1}
Every cyclotomic integer $\alpha$ with $\house{\alpha} < 2$
is either the sum of at most two roots of unity; equivalent to 
some number of the form $\tfrac{\sqrt{a}+i\sqrt{b}}{2}$ where $a,b$ are positive integers; or equivalent to one of
\[
\frac{3+\sqrt{13} + i \sqrt{26-6\sqrt{13}}}{4},
\qquad
1 + i \frac{\sqrt{5}+1}{2},
\qquad
2 \cos \frac{2\pi}{7} + \frac{1+i\sqrt{3}}{2}.
\]
\end{cor}
\begin{proof}
By \Cref{cor:RW truncation}, $\house{\alpha} < 2$ implies $\calM(\alpha) < 3 \tfrac{1}{4}$. Hence by \Cref{T:cassels partial classification}, every such $\alpha$ is either of one of the forms (1)--(3) of \Cref{T:cassels} or equivalent to an entry of \Cref{table:exceptional classes}. 
With the former in mind, we note that $1 + 4 \cos^2 \tfrac{\pi}{M}$ is monotonically decreasing for $M \geq 2$, so
\begin{equation} \label{eq:castle 4 options}
1 + 4 \cos^2 \tfrac{\pi}{M} < 4 \Longrightarrow M \in \{2,3,4,5\}.
\end{equation}
Since our desired statement expressly includes form (1) of \Cref{T:cassels}, it suffices to check the claim for $\alpha$ of form (2) or (3) or listed in \Cref{table:exceptional classes}. We now identify these values:

\begin{itemize}
   \item  \textbf{Case $(2)$ of \Cref{T:cassels}:}
Up to equivalence, we have $\alpha = 1+\zeta_N-\zeta_N^{-1}$ for some root of unity $\zeta_N$. For such $\alpha$, by \eqref{eq:house for cassels families}, we have $\house{\alpha}^2 = 1 + 4 \cos^2 \frac{\pi}{N'(N)}.$ By \eqref{eq:denom for cassels families}, the values of $\alpha$ for which $M = N'(N)$ satisfies \eqref{eq:castle 4 options} are $N=1,2,8,12,20$. 

\item \textbf{Case $(3)$ of \Cref{T:cassels}:}
Up to equivalence, we have $\alpha = (\zeta_5+\zeta_5^4) + (\zeta_5^2 + \zeta_5^3)\zeta_N$ for some root of unity $\zeta_N$. For such $\alpha$, by \eqref{eq:house for cassels families}, we have $\house{\alpha}^2 = 1 + 4 \cos^2 \frac{\pi}{N'(2N)}$. By \eqref{eq:denom for cassels families}, the values of $\alpha$ for which $M = N'(2N)$ satisfies \eqref{eq:castle 4 options} are $N=1,4, 6,10$.

\item \textbf{Cases from \Cref{table:exceptional classes}:}
Among the entries of \Cref{table:exceptional classes}, the only cyclotomic integers $\alpha$ with castle less than $4$ appear in the three bottom boxes (with $\house{\alpha}^2 = 2, 3, 4 \cos^2 \tfrac{\pi}{14}$). 

\end{itemize}

For each $\alpha$ identified above, we identify a cyclotomic integer of one of our desired forms with the same hash as $\alpha$ in the sense of \Cref{algo:perfect hash}. See \Cref{table:matching} for the matching; see the Jupyter notebook \verb|corollary-4.5.ipynb| in the GitHub repository for verification of the asserted equivalences.
\end{proof}

\begin{table}[ht]
\caption{Matching of cyclotomic integers $\alpha$ with $\house{\alpha}^2 < 4$, as classified by \Cref{T:cassels partial classification}, with cyclotomic integers listed in \Cref{cor:robinson1}.}
\label{table:matching}
\begin{tabular}{c|c|c}
Cyclotomic Integer & Cyclotomic Hash &  Item in Robinson's List\\
\hline
$1+\zeta_1-\zeta_1^{-1}$ & $x-1$ &  $1$ \\
$1+\zeta_2-\zeta_2^{-1}$ & $x-1$ & $1$ \\
$1+\zeta_8-\zeta_8^{-1}$ & $x^2-2x+3$ & $1+i\sqrt{2}$\\
$1+\zeta_{12}-\zeta_{12}^{-1}$ & $x^2-2x+2$ & $1+i$ \\
$1+\zeta_{20}-\zeta_{20}^{-1}$ & $x^4 - 4x^3 + 9x^2 - 10x + 5$ & $1+ \frac{1}{2}i(\sqrt{5}+1)$ \\
\hline
$(\zeta_5+\zeta_5^4) + (\zeta_5^2 + \zeta_5^3)\zeta_1$ & $x-1$ & $1$ \\
$(\zeta_5+\zeta_5^4) + (\zeta_5^2 + \zeta_5^3)\zeta_4$ & $x^4-2x^3+2x^2-6x+9$ & $\frac{1}{2}(\sqrt{10}+i \sqrt{2})$ \\
$(\zeta_5+\zeta_5^4) + (\zeta_5^2 + \zeta_5^3)\zeta_6$ & $x^4-3x^3+5x^2-6x+4$ &  $\frac{1}{2}(\sqrt{5}+i \sqrt{3})$ \\
$(\zeta_5+\zeta_5^4) + (\zeta_5^2 + \zeta_5^3)\zeta_{10}$ & $x^4-5x^3+10x^2-10x+5 $& $1-\zeta_5$ \\
\hline
$1 + \zeta_7 + \zeta_7^3$ & $x^2-x+2$ & $\frac{1}{2}(1+i\sqrt{7})$  \\
$1 + \zeta_{11} + \zeta_{11}^2 + \zeta_{11}^4 + \zeta_{11}^7$ & $x^2-x+3$ & $\frac{1}{2} (1+i\sqrt{11})$  \\
$1 + \zeta_{13} + \zeta_{13}^3 + \zeta_{13}^9$ & $x^4-3x^3+5x^2-9x+9$ & $\frac{3+\sqrt{13} + i \sqrt{26-6\sqrt{13}}}{4}$ \\
$1 - \zeta_{21} - \zeta_{21}^{13}$ & \parbox[c]{1.4
in}{$x^6 - 5x^5 + 13x^4 - 22x^3 + 28x^2 - 21x + 7$} & $2 \cos \frac{2\pi}{7} + \frac{1+i\sqrt{3}}{2}$ \\
$1 - \zeta_{35} + \zeta_{35}^{7} - \zeta_{35}^{11} -\zeta_{35}^{16}$ & $x^4+x^2+9$ & $\frac{1}{2}(\sqrt{5}+i \sqrt{7})$ \\
\end{tabular}
\end{table}

\begin{remark} \label{rem:biquadratic}
In the manner of \Cref{cor:robinson1}, we note that every entry $\alpha$ of \Cref{table:exceptional classes} with $\house{\alpha}^2 \in \{4,5\}$
whose cyclotomic hash generates a quadratic or biquadratic field
is equivalent to $\tfrac{\sqrt{a} + i \sqrt{b}}{2}$ for some positive integers $a,b$ with $a +b \in \{16, 20\}$.
This covers the entries with 
$\house{\alpha}^2 = 4$
of minimal level 7, 28, 39, 55, 60,
and the entries with $\house{\alpha}^2 = 5$
of minimal level 11, 19, 24, 51, 84, 91.
This fact will reappear in the proof of \Cref{L:remove p castle 5}.
\end{remark}

We record a further corollary of \Cref{T:cassels partial classification} which will be needed in 
\S\ref{sec:higher prime powers}.
\begin{cor} \label{lem:cor of partial classification}
Let $\alpha$ be a cyclotomic integer with $\house{\alpha}^2 < 5$ and $\N(\alpha) > 2$.
\begin{enumerate}
    \item[(a)] 
    If $\calM(\alpha) \leq 2 \tfrac{1}{2}$, then up to equivalence we have
\begin{gather*}
(\house{\alpha}^2, \calM(\alpha), \alpha) \in 
\{ (2, 2, 1+\zeta_7+\zeta_7^3), (2, 2, 1- \zeta_{15} - \zeta_{15}^{12}),
(1 + 4 \cos^2 \tfrac{\pi}{5}, 2 \tfrac{1}{2}, 1 + \zeta_{20} - \zeta_{20}^{19}),
\\
(4 \cos^2 \tfrac{\pi}{14}, 2 \tfrac{1}{3}, 1 - \zeta_{21}+ \zeta_{21}^{13}),
(\tfrac{5+\sqrt{13}}{2}, 2 \tfrac{1}{2}, 1+\zeta_{13}+\zeta_{13}^4),
(\tfrac{5+\sqrt{21}}{2}, 2 \tfrac{1}{2}, 1-\zeta_{21}+\zeta_{21}^6+\zeta_{21}^{18})
\}.
\end{gather*}

\item[(b)] If $\house{\alpha}^2 = 3$, then up to equivalence we have
\begin{gather*}
\alpha \in \{ 
1+\zeta_8-\zeta_8^{-1},
1+\zeta_{11}+\zeta_{11}^2+\zeta_{11}^4 +\zeta_{11}^7,
1+\zeta_{13}+\zeta_{13}^3+\zeta_{13}^9, \\
1-\zeta_{20}+\zeta_{20}^3+\zeta_{20}^4,
1-\zeta_{35}+\zeta_{35}^7-\zeta_{35}^{11}-\zeta_{35}^{16}
\}.
\end{gather*}
\end{enumerate}
\end{cor}
\begin{proof}
By~\Cref{T:cassels partial classification}, we may build the correct lists by first collecting all relevant cases from~\Cref{table:exceptional classes}, then identifying all relevant cases from~\Cref{T:cassels}(2)--(3)
using~\eqref{eq:house for cassels families} and~\eqref{eq:denom for cassels families}
(We do not use \Cref{T:cassels}(1) because we require $\N(\alpha) > 2$).

For (a), we first observe that for $\house{\alpha}^2 = 1 +4 \cos^2 \tfrac{\pi}{N}$, 
we have $\calM(\alpha) \leq 2 \tfrac{1}{2}$ if and only if $\calM(\zeta_N + \zeta_N^{-1}) \leq 1 \tfrac{1}{2}$;
by \Cref{L:short sums} this only happens for $N = 3, 5$.
To achieve  $\house{\alpha}^2 = 1 +4 \cos^2 \tfrac{\pi}{3} = 2$,
we specialize \Cref{T:main}(2) to $N=12$ (this degenerates to $1+i$) and \Cref{T:main}(3) to $N=6$.
To achieve  $\house{\alpha}^2 = 1 +4 \cos^2 \tfrac{\pi}{5}$,
we specialize \Cref{T:main}(2) to $N=20$ and \Cref{T:main}(3) to $N=10$ (this degenerates to $\zeta_5-\zeta_5^3$).

For (b), we specialize \Cref{T:cassels}(2) to $N=8$ and \Cref{T:cassels}(3) to $N=4$.
\end{proof}

\section{Combinatorial constraints}
\label{sec:combinatorial}

We pause the proof of \Cref{T:main} to introduce a
new technique to supplement the calculus of Cassels heights. This technique applies to the situation where we are looking for a cyclotomic integer $\alpha$ whose minimal level is a nontrivial multiple of the minimal level of $\house{\alpha}^2$. For a prime $p$ dividing the quotient, the $p$-decomposition of $\alpha$ is subject to some restrictive combinatorial constraints; these ensure that the number of nonzero terms in the $p$-decomposition can be bounded \emph{below} as a function of the prime $p$. This in turns aids the calculus of Cassels heights by eliminating some of the most elaborate cases.

\subsection{Combinatorics of modular difference sets}

We first state some elementary lemmas concerning difference sets in $\ZZ/p \ZZ$ and $\ZZ/p^2 \ZZ$ for $p$ prime.
All case enumerations asserted in the proofs can be verified by running the Jupyter notebook \verb|combinatorics.ipynb| in the GitHub repository.

\begin{lemma} \label{lem:singleton difference}
Fix a positive integer $X$. 
Let $p$ be a prime such that either $p \geq p_0$ with
\[
(X, p_0) \in \{(1, 2), (2, 3), (3, 5), (4, 11), (5, 13), (6, 19)\}
\]
or  $p > 6^{(X-1)/2}$.
Then for every $X$-element subset $S$ of $\ZZ/p\ZZ$, there exists $k \in \ZZ/p\ZZ$ such that the set 
$\{(i,j) \in S^2\colon i-j = k\}$ contains a single element.
\end{lemma}
\begin{proof}
For $X=1$, take $k=0$. For $X=2$, write $S = \{i,j\}$ and take $k = i-j$. For $X = 3$, without loss of generality write $S = \{0,1,i\}$, then take $k = 2$ if $i \in \{2,p-1\}$ and $k=1$ otherwise.

For $X \geq 4$, we argue as in \cite[Lemma~16]{stan-zaharescu}.
Let $S$ be an $X$-element subset of $\ZZ/p\ZZ$ such that none of the sets
$\{(i,j) \in S^2\colon i-j = k\}$ is a singleton.
Consider the system of linear equations consisting of $x_i - x_j - x_{i'} + x_{j'} = 0$ for each two-element subset $\{(i,j), (i',j')\}$ of $S^2$ for which $i \neq j$, $i' \neq j'$, and $i-j = i'-j'$. 
The solution space of this system over $\QQ$ on one hand contains $(1,\dots,1)$,
and on the other hand is contained in some hyperplane of the form $x_i -x_j = 0$.
(Otherwise there would be a solution over $\QQ$ with all coordinates distinct, and the difference between the largest and smallest coordinates would be equated with some strictly smaller difference.)
Let $r$ be the dimension of the solution space; then $1 \leq r \leq X-1$ and the matrix $M$ corresponding to the system has rank $X-r$.

By contrast, over $\mathbb{F}_p$, taking $x_i = i$ gives a solution not satisfying $x_i - x_j = 0$ for any $i \neq j$, so the solution space over must have dimension strictly greater than $r$.
Hence the minors of $M$ of dimension $X-r$ are not all zero but are all divisible by $p$.
Pick a nonzero such minor and bound it using Hadamard's inequality: each row of $M$ has Euclidean norm 2 (if $i,j,i',j'$ are all distinct) or $\sqrt{6}$ (if not).
This yields $p \leq 6^{(X-r)/2} \leq 6^{(X-1)/2}$.

It thus remains to check the cases where $X \in \{4,5,6\}$ and $p_0 \leq p \leq 6^{(X-1)/2}$. 
We may assume without loss of generality that $0,1 \in S$; we then check by direct enumeration.
\end{proof}

\begin{lemma} \label{lem:mod p2}
For
\[
(p,X) \in \{(3,3), (5,4), (5,5), (7,4), (7,5), (11,5)\},
\]
for every $X$-element subset $S$ of $\ZZ/p^2 \ZZ$
whose elements are pairwise distinct modulo $p$,
there exist $k_1,k_2 \in \ZZ/p^2 \ZZ$ 
with $k_1 \equiv k_2 \not\equiv 0 \pmod{p}$
such that the set $\{(i,j) \in S^2 \colon i-j = k_1\}$ is empty and
the set $\{(i,j) \in S^2 \colon i-j = k_2\}$ contains a single element.
\end{lemma}
\begin{proof}
We may assume without loss of generality that $0,1 \in S$. We then check the claim by direct enumeration.
\end{proof}

\begin{lemma} \label{lem:mod p graph}
For
\[
(p,X) \in \{(7, 3), (11,4), (13,4), (17,5), (19,5), (29, 6), (31, 6)\},
\]
let $S$ be a $X$-element subset of $\ZZ/p\ZZ$
such that for each nonzero $k \in \ZZ/p\ZZ$, the set $\{(i,j) \in S^2\colon i-j = k\}$ is nonempty.
Let $\Gamma$ be the graph on $S$ with an edge $(i,j)$ 
for each pair $(i,j) \in S^2$ such that $i-j \neq i'-j'$ for any pair $(i',j') \in S^2 \setminus \{(i,j)\}$. Then $\Gamma$ is connected and not bipartite.
\end{lemma}
\begin{proof}
We may assume without loss of generality that $0,1 \in S$. We then check the claim by direct enumeration.
\end{proof}

\begin{remark}
The bound $p > 6^{(X-1)/2}$ in \Cref{lem:singleton difference} improves upon the original bound $p > 4^{X-1}$ from \cite{straus} which uses a somewhat different technique. The improvement $p > 2^{X-1}$ is claimed in \cite[Corollary of Theorem 2]{browkin} but the proof is faulty: it follows the proof of \Cref{lem:singleton difference} but fails to account for the possibility of rows of norm $\sqrt{6}$ instead of 2. The bound given in \cite{stan-zaharescu} is $p > 6^{X/2}$ instead of $p > 6^{(X-1)/2}$;
it follows the proof of \Cref{lem:singleton difference} except that the rank of $M$ is bounded by $X$ rather than $X-1$.
It is suggested in \cite{stan-zaharescu} that the optimal bound may have the form $p > c 2^{X/2}$.
\end{remark}

\subsection{Application to the prime power case}

We next identify combinatorial constraints on the $p$-decomposition of $\alpha$
for a prime $p$ dividing the minimal level of $\alpha$ to a higher power that does not also divide the minimal level of $\house{\alpha}^2$.
This situation will arise when treating Step 1 in the remaining cases of \Cref{T:main}.

\begin{lemma} \label{lem:combinatorial constraints power}
Let $\alpha \in \QQ(\zeta_N)$ be a cyclotomic integer with minimal level $N$ exactly divisible by $p^n$ for some prime $p$ and some integer $n \geq 2$,
such that the minimal level of $\house{\alpha}^2$ is not divisible by $p^n$.
Define $\eta_i, X$ as in \Cref{L:rep-prime-power-case}.
\begin{itemize}
\item
If $X = 2$, then $p = 2$.
\item
If $X = 3$, then $p = 3$.
\item
If $X = 4$, then $p \in \{5,7\}$.
\item
If $X=5$, then $p \in \{5,7,11\}$.
\end{itemize}
\end{lemma}
\begin{proof}
With notation as in \Cref{L:rep-prime-power-case} and letting $S$ denote the set of indices $i$ for which $\eta_i \neq 0$ (viewed as a subset of $\ZZ/p\ZZ$), we have
\begin{equation} \label{eq:fix norm}
\alpha \overline{\alpha} = \sum_{i,j \in S} \eta_i \overline{\eta}_j \zeta_{p^n}^{i-j}
= \sum_{i \in S} \eta_i \overline{\eta}_i + \sum_{i \neq j \in S} \eta_i \overline{\eta}_j \zeta_{p^n}^{i-j}.
\end{equation}
 Since $p^n$ does not divide the minimal level of $\house{\alpha}^2$, we may gather terms to say that 
\begin{equation} \label{eq:fix norm apart}
\mathop{\sum_{i,j \in S}}_{i-j \equiv k \pmod{p}} \eta_i \overline{\eta}_j \zeta_{p^n}^{i-j-k} = 0 \qquad (k=1,\dots,p-1).
\end{equation}
Consequently, there can be no value of $k$ for which the set $\{(i,j) \in S^2\colon i-j \equiv k \pmod{p}\}$ is a singleton. Since $S \subseteq \ZZ/p\ZZ$ has size $X$, we trivially have $p \geq X$. Now for suitably large $p$, \Cref{lem:singleton difference} guarantees the existence of $k$ for which the aforementioned set is a singleton, thereby providing an upper bound on $p$. Combining these bounds restricts $p$ to the stated possibilities for each $X \in \{2,3,4,5\}$.
\end{proof}

\begin{cor} \label{cor:mod p2}
With notation as in \Cref{lem:combinatorial constraints power}, suppose that 
\[
(p,X) \in \{(2,2), (3,3), (5,4), (5,5), (7,4), (7,5), (11,5)\}
\]
and $p^n \neq 4$.
Then at least one $\eta_i$ has minimal level divisible by $p^{n-1}$.
\end{cor}
\begin{proof}
Suppose the claim fails in some instance.
Decompose $\alpha$ as a linear combination
$\sum_{i=0}^{p^2-1} \beta_i \zeta_{p^n}^i$
with $\beta_i \in \QQ(\zeta_{N/p^2})$
and let $S'$ denote the set of indices $i$ for which $\beta_j \neq 0$. The assumption that no $\eta_i$ has minimal level divisible by $p^{n-1}$ means that the elements of $S'$ are pairwise distinct mod $p$.
Now write
\begin{equation} \label{eq:mod p2 compare sides}
\alpha \overline{\alpha} - \sum_{i \in S'} \beta_i \overline{\beta}_i = \sum_{i \neq j \in S'} \beta_i \overline{\beta}_j \zeta_{p^n}^{i-j} \in \QQ(\zeta_{N/p}).
\end{equation}
For $p > 2$, \eqref{eq:mod p2 compare sides} implies
\[
\sum_{h=0}^{p-1}
\mathop{\sum_{i,j \in S'}}_{i-j \equiv k+ph \pmod{p^2}} \eta_i \overline{\eta}_j \zeta_{p^n}^{i-j-k} = 0 \qquad (k=1,\dots,p-1).
\]
By the uniqueness of $p$-decompositions, for each $k$, the inner sums must be equal for all $h$ (and for $n>2$ the common value must be zero). This creates a contradiction after applying \Cref{lem:mod p2}: for some $k$, there exist $k_1,k_2 \in \{0,\dots,p^2-1\}$ with $k_1 \equiv k_2 \equiv k \pmod{p}$ for which the summand with $h = (k_1-k)/p$ is empty and hence zero, while the summand with $h = (k_2-k)/p$ is a singleton sum and hence nonzero.

For $p = 2$, we may assume $S' = \{0,1\}$ without loss of generality, so \eqref{eq:mod p2 compare sides} reads
\[
\alpha \overline{\alpha} - \beta_0 \overline{\beta}_0 - \beta_1 \overline{\beta}_1 = \beta_1 \overline{\beta}_0 \zeta_{2^{n}} + \beta_0 \overline{\beta}_1 \zeta_{2^n}^{-1} \in \QQ(\zeta_{N/2}).
\]
Since $n > 2$ the common value must be zero,
but this yields $\beta_1 \overline{\beta_0} \zeta_{2^{n-1}} = -\beta_0 \overline{\beta}_1$ which expresses $\zeta_{2^{n-1}}$ as an element of $\QQ(\zeta_{N/4})$, a contradiction.
\end{proof}

\begin{cor} \label{cor:mod p2 refined}
With notation as in \Cref{lem:combinatorial constraints power} and the values of the tuple $(p,X)$ constrained as in \Cref{cor:mod p2}, let $m$ be the smallest positive integer such that no $\eta_i$ has minimal level divisible by $p^m$.
Then $n = m$ unless $p = 2$ and $m = 1$, in which case $n=2$.
\end{cor}
\begin{proof}
By definition $\eta_i \in \QQ(\zeta_{N/p})$ for all $i$,
so by \Cref{L:minimal level divides} no $\eta_i$ has minimal level divisible by $p^n$.
By \Cref{cor:mod p2}, unless $p^n = 4$,
at least one $\eta_i$ has minimal level divisible by $p^{n-1}$. This proves the claim.
\end{proof}

\begin{remark}
In \Cref{cor:mod p2}, one can also obtain some algebraic information when only a small number of $\eta_i$ have minimal level divisible by $p^{n-1}$. See \S\ref{sec:X3 N3} for an example.
\end{remark}

\subsection{Application to the prime case}

We next identify combinatorial constraints on the $p$-decomposition of $\alpha$
for a prime $p$ exactly dividing the minimal level of $\alpha$ but not dividing the minimal level of $\house{\alpha}^2$.
This situation will occur when treating Step 2 in the remaining cases of \Cref{T:main}.

\begin{lemma} \label{lem:combinatorial constraints}
Let $\alpha \in \QQ(\zeta_N)$ be a cyclotomic integer with minimal level $N$ squarefree and divisible by some prime $p$ not dividing the minimal level of $\house{\alpha}^2$. 
Define $X$ as in \Cref{L:rep-prime-case}.
\begin{itemize}
\item
For $p \geq 5$, we must have $X \geq 3$.
\item 
For $p \geq 11$, we must have $X \geq 4$.
\item
For $p \geq 17$, we must have $X \geq 5$.
\item
For $p \geq 23$, we must have $X \geq 6$.
\item
For $p \geq 37$, we must have $X \geq 7$.
\end{itemize}
\end{lemma}
\begin{proof}
With notation as in \Cref{L:rep-prime-case},
using the fact that the minimal level of $\house{\alpha}^2$ is not divisible by $p$, we may write
\[
\house{\alpha}^2 = \sum_{i,j \in S} \eta_i \overline{\eta}_j \zeta_{p}^{i-j}
= \sum_{i \in S} \eta_i \overline{\eta}_i + \sum_{i \neq j \in S} \eta_i \overline{\eta}_j \zeta_{p}^{i-j}
\]
and deduce that
\begin{equation} \label{eq:sum by residue}
\mathop{\sum_{i,j \in S}}_{i-j \equiv k \pmod{p}} \eta_i \overline{\eta}_j =  c  \qquad (k=1,\dots,p-1),
\qquad c := \sum_{i \in S} \eta_i \overline{\eta}_i - \house{\alpha}^2.
\end{equation}
Any counterexample to the asserted bounds would satisfy $p > X(X-1)+1$.
In this case, there must exist $k \in \{1,\dots,p-1\}$ for which the sum in \eqref{eq:sum by residue} is empty, so the common value $c$ must be 0.
Consequently, there can be no value of $k$ for which the set $\{(i,j) \in S^2\colon i-j \equiv k \pmod{p}\}$ is a singleton. 
We may thus apply \Cref{lem:singleton difference} to rule this case out.
\end{proof}

For some small $p$ we can push this a bit further.

\begin{lemma} \label{lem:connected by 2}
Let $\Gamma$ be a finite graph which is connected and not bipartite. Let $\Gamma'$ be the graph on the same vertex set in which two distinct vertices $v,w$ are adjacent iff they have a common neighbor in $\Gamma$. Then $\Gamma'$ is also connected.
\end{lemma}
\begin{proof}
Since $\Gamma$ is connected and not bipartite, it contains a cycle $C$ of odd length; any two distinct vertices $v,w \in C$ are connected in $\Gamma'$. Any other vertex can be connected to some vertex in $C$ via a path of even length, and hence also belongs to the connected component of $C$ in $\Gamma'$.
\end{proof}

\begin{lemma} \label{lem:combinatorial constraints2}
With notation as in \Cref{lem:combinatorial constraints}, if
\[
(p,X) \in \{(7, 3), (11,4), (13,4), (17,5), (19,5), (29, 6), (31, 6)\},
\]
then $\alpha$ is equivalent to a sum of $X$ distinct $p$-th roots of unity times some cyclotomic integer $\eta$ whose minimal level is not divisible by $p$.
\end{lemma}
\begin{proof}
Set notation as in \Cref{L:rep-prime-case}.
By \Cref{lem:singleton difference}, the common difference $c$ in \eqref{eq:sum by residue} must be nonzero; hence for every $k \in \{1,\dots,p-1\}$, there exists at least one pair $(i,j) \in S^2$ such that $i-j \equiv k \pmod{p}$. 
Let $\Gamma$ be the graph from \Cref{lem:mod p graph}, which is connected but not bipartite;
for every path $i \to j \to k$ in $\Gamma$, we obtain 
an equality $\eta_i \overline{\eta}_j = c = \eta_k \overline{\eta}_j$ and deduce that $\eta_i = \eta_k$. By \Cref{lem:connected by 2}, this implies that all of the $\eta_i$ are equal.
\end{proof}

\begin{remark}
    While it is not helpful here, one can push this logic somewhat further. For instance, for $(p,X) = (23,6)$, one can show that all but one of the $\eta_i$ must be equal.
\end{remark}

\section{Constraints from Artin reciprocity}
\label{sec:class groups}

We introduce a second new technique to supplement the calculus of Cassels heights.
Here we use Artin reciprocity to describe the splitting of primes in cyclotomic fields, then exploit this to restrict the ideal generated by a cyclotomic integer with prescribed castle.

\subsection{$S$-units and fractional ideals}
\label{subsec:analyzing a single field}

We first state a proposition to make the following logic explicit:
if we want $\alpha$ to belong to a fixed number field $K$ and have a fixed castle, then we can constrain $\alpha$ by identifying possibilities for the ideal generated by $\alpha$.
We state the proposition under the additional assumption that $|\alpha|=\house{\alpha}$, which causes no problem as any $\alpha$ admits a Galois conjugate satisfying this condition.

\begin{prop} \label{prop:fixed house setup}
Let $K$ be a cyclotomic number field and fix a nonzero
$\alpha_0\in\mathcal{O}_K$ such that
\[
  |\alpha_0|^2=\house{\alpha_0}^{2}=:c.
\]
Let $S_K$ be  the set of primes of $K$ dividing $c$.
Let $T_K\subseteq S_K$ contain exactly one representative from each two-element orbit under complex conjugation and no representative from any one-element orbit (so in particular $T_K$ may be empty).
Let $N_{K,S} \subset \mathcal{O}_{K,S_K}^\times$
be the kernel of the map $\alpha \mapsto \alpha \overline{\alpha}$ and let $v_{T,K} \colon N_{K,S} \to \ZZ^{T_K}$ denote the valuation map.
Define the finite set
\begin{equation} \label{eq:set U}
U_K := \prod_{\frakp \in T_K} [-v_{\frakp}(\overline{\alpha}_0), v_\frakp(\alpha_0) ] \subset \ZZ^{T_K}.
\end{equation}
Then the map
\[
  \alpha\longmapsto
  v_{T,K}\left(\frac{\alpha_0}{\alpha}\right)
\]
induces a bijection from the quotient of the set
\[
  \left\{
    \alpha\in\mathcal{O}_K\colon
    |\alpha|^2=\house{\alpha}^{2}=c
  \right\}
\]
by the multiplicative action of $W_K$ to $U_K \cap \image(v_{T,K})$.
(In particular, any two $\alpha$ with the same image are equivalent, but not conversely.)
\end{prop}
\begin{proof}
For $\alpha$ in the indicated set, both $\alpha$ and $\alpha_0$ are $S_K$-units. Moreover, for $\beta:=\alpha_0/\alpha$, we have
\[
  \beta\overline{\beta} =\frac{\alpha_0\overline{\alpha_0}}{\alpha\overline{\alpha}}
  =\frac{|\alpha_0|^2}{|\alpha|^2}=1.
\]
Hence $\beta\in N_{K,S}$, and the map is well-defined.
If $\mathfrak p\in S_K$ is fixed by complex conjugation, then
\[
  2v_{\mathfrak p}(\beta) =v_{\mathfrak p}(\beta\overline{\beta})=0.
\]
Consequently, $\beta$ has zero valuation at every prime fixed by complex conjugation. We have $\ker(v_{T,K}) = W_K$ by Kronecker's theorem, which implies injectivity of our map at the level of $W_K$-cosets. To confirm the description of the image, interpret $u \in \ZZ^{T_K}$ as the image under $v_{T,K}$ of the fractional ideal
\begin{equation} \label{eq:ideal Ju}
J_u := \prod_{\frakp \in T_K} \frakp^{u_\frakp} \overline{\frakp}^{-u_\frakp};
\end{equation}
if $J_u$ is generated by $\beta\in N_{K,S}$, set $\alpha:=\alpha_0/\beta$. Since
$\beta\overline{\beta}=1$, every conjugate of $\beta$ has absolute value $1$. It follows that
\[
  |\alpha|^2=|\alpha_0|^2=c \qquad\text{and}\qquad \house{\alpha}^{2}=\house{\alpha_0}^{2}=c.
  \qedhere
\]
\end{proof}

\subsection{Splitting of primes}

We next make \Cref{prop:fixed house setup} more explicit using the description of splitting of primes in a cyclotomic field afforded by Artin reciprocity.

\begin{prop} \label{prop:cyclotomic splitting}
    Let $N$ be a positive integer and set $K := \QQ(\zeta_N)$. 
    Let $p$ be a prime, let $e$ be the $p$-adic valuation of $N$, and set $N' := N/p^e$.
    Via the Artin isomorphism 
    \[
    \Gal(K/\QQ) \cong (\ZZ/N \ZZ)^\times \cong (\ZZ/N'\ZZ)^\times \times (\ZZ/p^e \ZZ)^\times,
    \]
    the inertia group of $p$ corresponds to $(\ZZ/p^e \ZZ)^\times$, while the decomposition group
    is generated modulo inertia by the subgroup $\langle p \rangle$ of $(\ZZ/N'\ZZ)^\times$ generated by $p$. Consequently, the set of primes of $K$ above $p$ is a torsor for the quotient group $(\ZZ/N'\ZZ)^\times/\langle p \rangle$;
    in particular, the number of such primes is $\varphi(N')/ \#\langle p \rangle$.
\end{prop}
\begin{proof}
See for example the discussion in \cite[Chapter~2]{washington}.
\end{proof}

\begin{remark} \label{rem:effect of complex conjugation}
    In \Cref{prop:cyclotomic splitting}, the order of $\langle p \rangle$ is the lcm of the order of the subgroups of $(\ZZ/\ell^e \ZZ)^\times$ generated by $p$ over all prime powers $\ell^e$ dividing $N'$.
    
    For a subfield $K$ of $\QQ(\zeta_N)$, the decomposition and inertia groups of $p$ in $K$ are obtained by projecting the corresponding groups associated to $\QQ(\zeta_N)$ along the surjection $\Gal(\QQ(\zeta_N)/\QQ) \to \Gal(K/\QQ)$.
    The kernel of said surjection is $\Gal(\QQ(\zeta_N)/K)$; let $H$ denote the corresponding subgroup of $(\ZZ/N\ZZ)^\times$.

    The unique complex conjugation on $K$ is induced by the element of $\Gal(\QQ(\zeta_N)/\QQ)$ corresponding to $-1 \in (\ZZ/N\ZZ)^\times$. Consequently, for $N'$ as in \Cref{prop:cyclotomic splitting}, if the image of $-1$ in $(\ZZ/N' \ZZ)^\times$ belongs to the subgroup generated by the images of $p$ and $H$, then the primes above $p$ in $K$ are all fixed by complex conjugation; otherwise, these primes partition into two-element orbits for complex conjugation.
\end{remark}

\begin{remark} \label{rem:short circuit}
In this paper, we will apply  \Cref{prop:fixed house setup} with $c = p^m$ for some rational prime $p$ and some positive integer $m$.
In this case, the sets $S_K$ and $T_K$ can be described using \Cref{prop:cyclotomic splitting} and \Cref{rem:effect of complex conjugation} once we compute the order of $p$ in $(\ZZ/\ell^e \ZZ)^\times$ for all prime powers $\ell^e$ dividing the conductor of $K$. We tabulate the relevant values here.
\begin{equation} \label{eq:multiplicative order}
\begin{array}{c|ccccccccccc}
  \ell^e & 2^2 & 2^3 & 3^2 & 5 & 7 & 11 & 13 & 17 & 19 & 23 & 29 \\
  \hline
p=2 & - & - & 6 & 4 & 3 & 10 & 12 & 8 & 18 & 11 & 28\\
p=5 & 1 & 2 & 6 & - & 6 & 5 & 4 & 16 & 9 & 22 & 14
\end{array}
\end{equation}

If $T_K \neq \emptyset$, then $\#T_K = \tfrac{1}{2} \#S_K = \tfrac{1}{2} [G:G_p]$.
Let $e$ be the ramification index of $p$ in $K$.
In case $\alpha_0$ is coprime to $\overline{\alpha}_0$
(resp. $\alpha_0 = \overline{\alpha}_0$),
we have $U_K = \prod_{\frakp \in T_K} [0, em]$
(resp. $U_K = \prod_{\frakp \in T_K} [-\tfrac{em}{2}, \tfrac{em}{2}]$);
in this case, $u$ and $(em,\dots,em) - u$ 
(resp. $u$ and $-u$)
correspond to values of $\alpha$ which are complex conjugates of each other, and hence equivalent.
\end{remark}

\subsection{Sharp bounds}

In general, applying \Cref{prop:fixed house setup} requires controlling $\image(v_{T,K})$ well enough to compute its intersection with $U_K$.
We next identify some cases where no such control is needed because the
known values of $\alpha$ with $\house{\alpha}^2 = c$ suffice to account for all of $U_K$. See \Cref{prop:step3 case 4} for an elaborate variant of the same argument.

\begin{lemma} \label{lem:short circuit}
For each row of \Cref{table:short circuit}, every cyclotomic integer $\alpha$ of minimal level dividing $N$ with $\house{\alpha}^2 = c$ is equivalent to one of the listed values, and in particular is covered by \Cref{T:main}.
\end{lemma}
\begin{proof}
    For each row, we apply \Cref{prop:fixed house setup} with $\alpha_0$ taken to be the first listed value of $\alpha$.
    Using \eqref{eq:multiplicative order}, we may compute the indicated values of $\#T_K$ and $U_K$ 
    and observe that every element of $U_K$ is accounted for
    by a Galois conjugate of a listed value of $\alpha$; in particular, $U_K \subseteq \image(v_{T,K})$.
\end{proof}

\begin{table}
\caption{Cases of \Cref{lem:short circuit}. Here $*_M$ denotes an entry $\alpha$ of \Cref{table:exceptional classes} of minimal level $M$ with $\house{\alpha}^2 = c$.}
\label{table:short circuit}
\begin{tabular}{c|c|c|c|c|c}
    $c$ & $N$ & $\alpha$ & $\#T_K$ & $U_K$ & Usage\\
    \hline
    \multirow{3}{*}{$4$}
    & $3^2 \times 5$ & $2, (1-\zeta_{15}+\zeta_{15}^{12})^2$ & $1$ & $\{-1,0,1\}$ & \S\ref{sec:X3 N3} \\
    & $3 \times 13$ & $2, *_{39}$ & $1$ & $\{-1,0,1\}$ & \S\ref{subsec:castle 4 prime 13} \\
    & $5 \times 13$ & $2$ & $0$ & $0$ & \Cref{R:handle Y when 0}, \S\ref{subsec:castle 4 prime 13} \\
    \hline
    \multirow{7}{*}{$5$}
    & $2^3 \times 5$ & $\sqrt{5}, 2+i, *_{20}$ & $1$ & $\{-2,-1,0,1,2\}$ & \S\ref{subsubsec:X2 not totally real} \\
    & $2^2 \times 7$ & $2+i$ & 1 & $\{0,1\}$ & \S\ref{subsubsec:X2 not totally real} \\
    & $3^2 \times 5 \times 7$ & $\sqrt{5}$ & $0$ & $0$ & \S\ref{sec:X3 N3} \\
    & $3 \times 17$ & $*_{51}$ & $1$ & $\{0,1\}$ &\Cref{R:handle Y when 0}, \S\ref{subsec:castle 5 prime 17} \\
    & $3 \times 19$ & $*_{19}$ & $1$ & $\{0,1\}$ & \Cref{R:handle Y when 0}, \S\ref{subsec:castle 5 prime 19}\\
    & $5 \times 23$ & $\sqrt{5}$ & $0$ & 0 & \Cref{R:handle Y when 0} \\
    & $5 \times 29$ & $\sqrt{5}$ & $0$ & 0 & \Cref{R:handle Y when 0} \\
\end{tabular}
\end{table}

\begin{remark} \label{rem:class group bound}
    When \Cref{lem:short circuit} does not apply, in principle one can compute $\image(v_{T,K})$ as follows.
    See \Cref{lem:step3 case 5 remove 3} for a small example.
    
    In \Cref{prop:fixed house setup}, the group $\coker(v_{T,K})$ is always finite because the class group of $K$ is finite: for any $u \in \ZZ^{T_K}$ we can find a positive integer $n$ such that $J_{nu} = J_u^n$ is generated by some $\beta \in K^\times$, and then $J_{2nu} = J_u^{2n}$ is generated by $\beta/\overline{\beta} \in N_{K,S}$.
    That said, the class group computations may become prohibitive for large $K$.

    Now suppose that $n$ is a positive integer such that $J_{nu}$ is generated by a \emph{known} element $\beta_0 \in N_{K,S}$.
    If $J_u$ is generated by $\beta \in N_{K,S}$, then $\beta^n$ and $\beta_0$ are both elements of $N_{K,S}$ generating $J_{nu}$, and so $\beta^n/\beta_0 \in W_K$. Consequently, we can determine whether or not $\beta$ exists by testing whether $\beta_0 \zeta$ has an $n$-th root in $K$ for some $\zeta \in W_K$.
\end{remark}

\subsection{Excluding primes from the minimal level}

In general, when applying \Cref{prop:fixed house setup} we will have $U_K \not\subset \image(v_{T,K})$ and so some work will be needed to compute the intersection (\Cref{rem:class group bound}). One can simplify this work by first identifying some primes that are obstructed from dividing the minimal level.

\begin{lemma} \label{lem:fixed house descent}
    Let $N$ be a positive integer, let $p$ be an odd prime such that $p^e$ exactly divides $N$ for some positive integer $e$, and set $N' := N/p^e$. Fix a nonzero $c \in \mathcal{O}_K$.
    Suppose that $K_0$ is a subfield of $\QQ(\zeta_N)$
    containing $\QQ(\zeta_{N'})$, $c$, and 
    the decomposition field of every prime of $K$ dividing $c$,
    but not containing $\zeta_p$.
    Then every cyclotomic integer $\alpha \in \QQ(\zeta_N)$ with $\house{\alpha}^2 = c$
    is equivalent to an element of $\QQ(\zeta_N)$ which generates a Kummer extension of $K_0$ (depending on $\alpha$) of degree dividing $k :=  \gcd([\QQ(\zeta_{N}) : K_0], \lcm(2,N'))$. 
\end{lemma}
\begin{proof}
    Set $G := \Gal(\QQ(\zeta_N)/K_0)$, identified with a subgroup of $(\ZZ/p^e \ZZ)^\times$ via Artin reciprocity.
    Since $p$ is odd, $G$ is cyclic; fix a generator $g \in G$.
    Since $\zeta_p \notin K_0$, $g$ maps to a nontrivial element of $(\ZZ/p\ZZ)^\times$; hence $\zeta_{p^e}^g / \zeta_{p^e}$ is a primitive $p^e$-th root of unity.
    
    Let $\alpha\in\mathcal{O}_K$ satisfy $\house{\alpha}^{2}=c$. Since the conclusion is stated up to equivalence, we may replace $\alpha$ by a Galois conjugate and assume that
\[
  |\alpha|^2=\house{\alpha}^{2}=c.
\]
Since $c\in K_0$, every $h\in G$ fixes $c$, and hence
\[
  |\alpha^h|^2=h(\alpha\overline{\alpha})=h(c)=c.
\]
Moreover, $\house{\alpha^h}=\house{\alpha}$, so
\[
  |\alpha^h|^2=\house{\alpha^h}^{2}=c.
\]
Thus both $\alpha$ and $\alpha^h$ satisfy the normalization required in~\Cref{prop:fixed house setup}.
    Using the hypothesis that $K_0$ contains both $c$ and the decomposition field of every prime of $K$ dividing $c$, we may apply \Cref{prop:fixed house setup} to see that $\tfrac{\alpha^g}{\alpha} \in W_{\QQ(\zeta_N)}$. 
    By the previous paragraph, after multiplying $\alpha$ by a suitable power of $\zeta_{p^e}$ we may ensure that in fact
    $\tfrac{\alpha^g}{\alpha} \in W_{\QQ(\zeta_{N'})}$.

    Now the formula $\rho(h) = \tfrac{\alpha^h}{\alpha}$ defines a homomorphism $\rho \colon G \to W_{\QQ(\zeta_{N'})}$, whose order must divide 
    $\gcd(\#G, \#W_{\QQ(\zeta_{N'})}) = k$.
    We deduce that $\alpha^k \in \QQ(\zeta_{N'})$; since $\zeta_k \in \QQ(\zeta_{N'})$, $\alpha$
    generates a Kummer extension of $\QQ(\zeta_{N'})$ of degree dividing $k$.
\end{proof}

The following consequence will be crucial for Step 2 of the proof of \Cref{T:main}.
\begin{lemma} \label{L:remove p}
Let $\alpha$ be a cyclotomic integer such that $\house{\alpha}^2$ is a power of an odd prime number $p$.
Then the minimal level of $\alpha$ is not divisible by $p^2$.
\end{lemma}
\begin{proof}
Suppose by way of contradiction that the minimal level $N$ of $\alpha$ has the form $p^e N'$ with $e \geq 2$ and $N'$ coprime to $p$.
By \Cref{lem:fixed house descent}, $\alpha$ is equivalent to an element of some cyclic subextension of $\QQ(\zeta_N)/\QQ(\zeta_{N'})$ of degree dividing
$\gcd((p-1) p^{e-1}, \lcm(2,N')) = \gcd(p-1, \lcm(2,N'))$; any such subextension is contained in $\QQ(\zeta_{pN'})$.
\end{proof}

We specialize further to the cases $\house{\alpha}^2 = 4, 5$.
The resulting lemmas will be helpful in all three steps of the proof of \Cref{T:main}.

\begin{lemma} \label{L:remove 5}
Let $\alpha$ be a cyclotomic integer with odd minimal level $N$
such that $\house{\alpha}^2 = 5$.
If $\alpha$ is not equivalent to $\sqrt{5}$, then $N$ is coprime to $5$.
(This can fail for $N$ even: see \Cref{table:exceptional classes} for an example where $\alpha$ has minimal level $20$.)
\end{lemma}
\begin{proof}
By \Cref{L:remove p}, $5^2$ does not divide $N$. 
Suppose that $5$ divides $N$.
By \Cref{lem:fixed house descent}, 
we may assume that $\alpha$ belongs to a subextension of $\QQ(\zeta_N)/\QQ(\zeta_{N/5})$ of degree dividing
$\gcd(4, 2N/5) = 2$; that is, we must have
$\alpha \in K := \QQ(\zeta_{N/5}, \sqrt{5})$.
Now set notation as in \Cref{prop:fixed house setup} with $\alpha_0 = \sqrt{5}$
and set $u := v_{T,K}(\tfrac{\alpha_0}{\alpha})$.
Since the nontrivial automorphism $c$ of $K$ over $\QQ(\zeta_{N/5})$ fixes $T_K$,
we have $\epsilon := \tfrac{\alpha^c}{\alpha} \in W_K = W_{\QQ(\zeta_{N/5})}$.
Since $N$ is odd, without loss of generality we may assume $\epsilon \in \{\pm 1\}$. If $\epsilon = 1$, then $\alpha \in \QQ(\zeta_{N/5})$; 
if $\epsilon = -1$, then $\tfrac{\alpha_0}{\alpha} \in \QQ(\zeta_{N/5})$ and so $u \in (2\ZZ)^{T_K}$, but since $u \in U_K = \{-1,0,1\}^{T_K}$ (as in \Cref{rem:short circuit}) this forces $u = 0$ and so $\alpha$ is equivalent to $\sqrt{5}$.
\end{proof}

We can make a similar argument to remove some unramified primes.
This will be helpful in Steps 2 and 3 of the proof of \Cref{T:main}.

\begin{lemma} \label{L:remove p castle 5}
Choose $(\ell, c, p_0) \in \{(2,4,11), (5,5,7)\}$.
Let $p \geq p_0$ be a prime such that $\tfrac{p-1}{2}$ is also prime (this forces $p \equiv 3 \pmod{4})$.
Let $N$ be a positive integer coprime to $\ell p$
such that for every prime $q$ dividing $N$,
the order of $\ell$ in $(\ZZ/q\ZZ)^\times$ is not divisible by $\tfrac{p-1}{2}$ (this is automatic if $\tfrac{p-1}{2}$ does not divide $q-1$).
Then every $\alpha \in \QQ(\zeta_{pN})$ with $\house{\alpha}^2 = c$ is either covered by \Cref{T:main} or equivalent to a value in $\QQ(\zeta_{N})$.
\end{lemma} 
\begin{proof}
Since $p > \ell+1$, $p$ does not divide $(\ell-1)(\ell+1) = \ell^2-1$; hence the order of $\ell$ in $(\ZZ/p\ZZ)^\times$ cannot be 1 or 2, and hence must be divisible by $\tfrac{p-1}{2}$.
Consequently, the orders of both the group $(\ZZ/pN \ZZ)^\times$ and its subgroup generated by $\ell$ are exactly divisible by $\tfrac{p-1}{2}$,
so the decomposition field of $\ell$ in $\QQ(\zeta_{pN})$ is contained in $\QQ(\zeta_N, \sqrt{-p})$.
We may therefore apply \Cref{lem:fixed house descent} 
to see that $\alpha$ is equivalent to an element of $\QQ(\zeta_{pN})$ generating a Kummer extension
of $\QQ(\zeta_N, \sqrt{-p})$ of degree dividing $p-1$.
If this extension were nontrivial, it would on one hand be unramified away from places above divisors of $\ell(p-1)$, but on the other hand would be totally ramified above $p$.
This contradiction implies that $\alpha$ is in fact 
equivalent to an element of $\QQ(\zeta_N, \sqrt{-p})$.

Since the number fields $\QQ(\zeta_N)$ and of $\QQ(\sqrt{-p})$ are Galois, linearly disjoint, and have coprime discriminants,
their rings of integers together generate the ring of integers of $\QQ(\zeta_N, \sqrt{-p})$. We may thus write
\[
\alpha = \eta_0 + \eta_1 \sqrt{- p} \qquad (\eta_0, \eta_1 \in \tfrac{1}{2}\ZZ[\zeta_N], \eta_0-\eta_1 \in \ZZ[\zeta_N]).
\]
Expanding the equality 
\[
c = \alpha \overline{\alpha} = (\eta_0 + \eta_1 \sqrt{- p})(\overline{\eta}_0 -\overline{\eta}_1  \sqrt{- p})
\]
in terms of $\sqrt{- p}$ yields
\[
\eta_0 \overline{\eta}_0 + p \eta_1 \overline{\eta}_1 = c,
\qquad
\eta_0 \overline{\eta}_1 - \eta_1 \overline{\eta}_0 = 0.
\]
If $\eta_1 = 0$, then $\alpha = \eta_0 \in \QQ(\zeta_N)$ and we are done. Otherwise, $\house{2\eta_1}^2 \leq \tfrac{4c}{p}$;
this is impossible for $c=4, p \geq 19$ or $c=5, p \geq 23$. 
For $c=4, p = 11$ or $c=5, p = 11, 19$, we have $\house{2\eta_1}^2 < 2$; by \Cref{lem:cassels low height}, up to equivalence we may force $\eta_1 = \tfrac{1}{2}$. For $p = 7$, we reach the same conclusion by noting that 
none of the options with $1 < \house{2 \eta_1}^2 \leq \tfrac{4c}{p}$
offered by  \Cref{lem:cassels low height} is compatible with our restrictions on $N$.

Now $\tfrac{\eta_0}{\eta_1}$ is totally real
and $|\eta_0|^2 = \tfrac{4c-p}{4}$, so $\eta_0 = \pm \tfrac{\sqrt{4c-p}}{2}$.
We deduce that $\alpha$ is equivalent to 
$\tfrac{\sqrt{5} + \sqrt{-11}}{2}$ for $c=4, p=11$;
$\tfrac{\sqrt{13} + \sqrt{-7}}{2}$ for $c=5, p=7$;
$\tfrac{3+\sqrt{-11}}{2}$ for $c=5, p=11$;
or $\tfrac{1+\sqrt{-19}}{2}$ for $c=5, p=19$.
All of these cases are covered by \Cref{T:main} (compare \Cref{rem:biquadratic}).
\end{proof}

\section{Higher prime powers in the minimal level}
\label{sec:higher prime powers}

In this section, we prove the following result that resolves Step 1 of the proof of \Cref{T:main} in cases (b), (c), (d).

\begin{prop} \label{prop:Step 1}
Let $\alpha$ be a cyclotomic integer with $\house{\alpha}^2 < 5.01$ and $\calM(\alpha) \geq 3 \tfrac{1}{4}$.
If the minimal level of $\alpha$ is not squarefree, then $\alpha$ is covered by \Cref{T:main}.
\end{prop}

\begin{remark}
The statement of \Cref{prop:Step 1} corresponds to taking $N_0 = 1$ in each of the cases (b), (c), (d) in the proof of \Cref{T:main} (in the sense of \Cref{table:main breakdown}). In case (b) it would suffice to take $N_0 = 2^2$ because Step 3 will be handled using \Cref{prop:list 420 part1} as in case (a). We have opted not to take this shortcut because it does not save much effort, and anyway the full argument helps to explain the presence of many entries of \Cref{table:exceptional classes}.
\end{remark}

\begin{proof}[Proof of \Cref{prop:Step 1}]
Let $N$ be the minimal level of $\alpha$,
let $p$ be the smallest prime such that $p^2$ divides $N$, and let $n$ be the $p$-adic valuation of $N$.

The possible values of $\house{\alpha}^2$ are those allowed by cases (b)--(d) of \Cref{cor:RW truncation}.
In particular, $\house{\alpha}^2$ has minimal level which is squarefree and hence not divisible by $p^n$.
We may thus set notation as in \Cref{L:rep-prime-power-case},
so that $\calM(\alpha) = \sum_{i \in S} \calM(\eta_i)$. Since each summand is least 1 by \Cref{L:short sums}(a), we have $X \leq \calM(\alpha) \leq \house{\alpha}^2 \leq 5$. 

As in \Cref{P:robinson wurtz extract1}, we separate cases according to the value of $X$.
Before doing so, we articulate a number of common points (which will be restated as they arise).
\begin{itemize}
\item
For each $i \in S$, $\house{\eta_i}^2 < \house{\alpha}^2 \leq 5$ and so $\house{\eta_i}^2$ must be a value appearing in \Cref{T:robinson-wurtz}.
\item
If $\eta_i$ is equivalent to a totally real cyclotomic integer, then by \Cref{P:totally real case} the latter can be chosen to be either $2 \cos \tfrac{\pi}{m}$ for some positive integer $m$ or $\frac{\sqrt{3}+\sqrt{7}}{2}$.
\item
If $\eta_i$ is not equivalent to a totally real cyclotomic integer but $\calM(\eta_i) < 3 \tfrac{1}{4}$, we can classify $\eta_i$ using \Cref{T:cassels partial classification}. See \Cref{lem:cor of partial classification} for the key conclusion.
\item
The choice of the $\eta_i$ uniquely determines $n$ via
\Cref{cor:mod p2 refined}.
\item
By \Cref{L:remove 5},
if $\house{\alpha}^2 = 5$ and $p > 2$,
then $N$ is coprime to 5.
\end{itemize}

\subsection{Case: $X=2$}

We have $p=2$ by \Cref{lem:combinatorial constraints power}.
By \Cref{cor:mod p2},
if $n = 2$, each $\eta_i$ has odd minimal level;
if $n > 2$, some $\eta_i$ has minimal level divisible by $2^{n-1}$.

We may assume $\calM(\eta_0) \leq \calM(\eta_1)$.
By rewriting the equation $\eta_0 \overline{\eta}_1 + \eta_1 \overline{\eta}_0 \zeta_{2^{n-1}} = 0$ as
\[
\overline{\eta}_1 \overline{\eta}_0^{-1} = \zeta_{2^n}^{-2^{n-1}} \eta_1 \eta_0^{-1} \zeta_{2^n}^2,
\]
we see that $\eta_1 \eta_0^{-1} \zeta_{2^n}^{1-2^{n-2}}$ is totally real; in other words,
\begin{equation} \label{eq:x2 fix rotation}
\eta_1 = \pm \frac{|\eta_1|}{|\eta_0|} \eta_0 \zeta_{2^{n}}^{2^{n-2}-1}
 = \pm \frac{\sqrt{|\eta_0 \eta_1|^2}}{\overline{\eta}_0} \zeta_{2^n}^{2^{n-2}-1}.
\end{equation}
Consequently, $\eta_0$ is equivalent to a totally real cyclotomic integer if and only if $\eta_1$ is;
we distinguish subcases based on whether this occurs.

\subsubsection{Subcase: $\eta_i$ equivalent to totally real}
If both $\eta_i$ are equivalent to totally real cyclotomic integers,
then we may assume that $\eta_0$ is itself totally real.

If $\calM(\alpha) = 4$, then by \Cref{P:totally real case} and \eqref{eq:x2 fix rotation}, up to Galois conjugation we have
\[
\eta_0 = 2 \cos \tfrac{\pi}{m}, \qquad \eta_1 = \pm \zeta_{2^{n}}^{2^{n-2}-1} 2 \sin \tfrac{\pi}{m}
\]
for some positive integer $m$. We then
obtain $\alpha = 2 \left(\cos \tfrac{\pi}{m} \pm i \sin \tfrac{\pi}{m}\right) = 2 \zeta_{2m}^{\pm 1}$,
which is covered by \Cref{T:main}(1).

If $\calM(\alpha) \neq 4$, we may apply the constraints imposed by \Cref{P:totally real case} to solve the equation
$\eta_0 \overline{\eta}_0 + \eta_1 \overline{\eta}_1 = \house{\alpha}^2$;
we may ignore cases where $\eta_0 = 1$, as these give rise to examples of \Cref{T:cassels}(2).
We present the results in \Cref{table:house 2 way split part 1}; all of these cases are covered by either \Cref{T:cassels} or \Cref{table:exceptional classes}.

\begin{table} 
\caption{Possible values of $\house{\eta_0}, \house{\eta_1}$ when $X = 2$, $3\frac14\leq\calM(\alpha)$, $\calM(\alpha) \neq 4$, and each $\eta_i$ is equivalent to a totally real cyclotomic integer.}
\renewcommand{\arraystretch}{1.2}
\label{table:house 2 way split part 1}
\begin{tabular}{c|cc|cc|c|c}
$\calM(\alpha)$ & $\calM(\eta_0)$ & $\calM(\eta_1)$ & $\house{\eta_0}$ & $\house{\eta_1}$ & Level of $\alpha$ & Table~\ref{table:exceptional classes}? \\
\hline
$3 \tfrac{1}{4}$ & $1 \tfrac{1}{2}$ & $1 \tfrac{3}{4}$ & $2 \cos \tfrac{\pi}{5}$ & $2 \cos \tfrac{\pi}{30}$ & $2^2 \times 3 \times 5$ & No \\
\hline
$3 \tfrac{1}{3}$ & $1 \tfrac{2}{3}$ & $1 \tfrac{2}{3}$ & $2 \cos \tfrac{\pi}{7}$& $2 \cos \tfrac{\pi}{7}$ & $2^2 \times 7$ & Yes \\
\hline
\multirow{2}{*}{$3 \tfrac{1}{2}$} & $1 \tfrac{1}{2}$ & 2 & $2 \cos \tfrac{\pi}{5}$ & $\sqrt{2}$ & $2^3 \times 5$ & Yes \\
& $1 \tfrac{3}{4}$ & $1 \tfrac{3}{4}$ & $2 \cos \tfrac{\pi}{30}$ & $2 \cos \tfrac{\pi}{30}$ & $2^2 \times 3 \times 5$ & Yes \\
\hline
\multirow{3}{*}{$5$} & $2$ & 3 & $\sqrt{2}$ & $\sqrt{3}$ & $2^3 \times 3$ & Yes\\
& $2 \tfrac{1}{2}$ & $2 \tfrac{1}{2}$ & $2 \cos \tfrac{\pi}{10}$ & $2 \cos \tfrac{\pi}{10}$ & $2^2 \times 5$ & Yes \\
& $2 \tfrac{1}{2}$ & $2 \tfrac{1}{2}$ & $\tfrac{\sqrt{3}+ \sqrt{7}}{2}$ & $\tfrac{\sqrt{3}+ \sqrt{7}}{2}$ & $2^2 \times 3 \times 7$ & Yes \\
\end{tabular}
\renewcommand{\arraystretch}{1}
\end{table}

\subsubsection{Subcase: $\eta_i$ not equivalent to totally real}
\label{subsubsec:X2 not totally real}

If neither $\eta_i$ is equivalent to a totally real cyclotomic integer,
then by Lemma~\ref{L:short sums}(c),
\[
2 \leq \calM(\eta_0) \leq \calM(\eta_1) \leq 3;
\]
in particular, $\calM(\alpha) \geq 4$ and so $\house{\alpha}^2 \in \{4,5\}$.
Since $\N(\eta_0) > 2$, $\calM(\eta_0) \leq 2 \tfrac{1}{2}$, and $\house{\eta_0}^2 < 5$, we
may apply \Cref{lem:cor of partial classification} to classify $\eta_0$ up to equivalence; we exclude the case $\house{\eta_0}^2 = \tfrac{5+\sqrt{21}}{2}$ as in this case $\eta_0$ is equivalent to the totally real cyclotomic integer $\tfrac{\sqrt{3}+\sqrt{7}}{2}$.

For $\house{\alpha}^2 = 4$, we must have $|\eta_0| = |\eta_1| = 2$; we deduce from \eqref{eq:x2 fix rotation} that $\eta_1 = \eta_0 \zeta_{2^n}^{2^{n-2}-1}$ and hence $\alpha = (1+i) \eta_0$.
Since $\eta_0 \in \{1+\zeta_7+\zeta_7^3, 1-\zeta_{15}+\zeta_{15}^2\}$, these cases appear in Table~\ref{table:exceptional classes}.

For $\house{\alpha}^2 = 5$,  
by \Cref{lem:cor of partial classification}(a) we have $\house{\eta_0}^2 \in \{1 + 4 \cos^2 \tfrac{\pi}{5}, 4 \cos^2 \tfrac{\pi}{14}, 2, \tfrac{5+\sqrt{13}}{2}\}$.
In the first two cases, by \Cref{cor:mod p2 refined}, $\alpha$ must have minimal level dividing one of $2^3 \times 5$ or $2^2 \times 7$, and \Cref{lem:short circuit} applies.
In the last two cases, we have $|\eta_0 \eta_1|^2 = 6,3$, respectively, and so the right-hand side of \eqref{eq:x2 fix rotation} is not an algebraic integer.

\subsection{Case: $X=3$}
\label{subsec:x3}

We have $p=3$ by \Cref{lem:combinatorial constraints power};
in particular, $\alpha$ has odd minimal level.
We may assume without loss of generality that $\calM(\eta_0) \leq \calM(\eta_1) \leq \calM(\eta_2)$.
By \Cref{cor:mod p2}, at least one $\eta_i$ has minimal level divisible by 3;
by \Cref{L:short sums}(b), this implies $\calM(\eta_2) \geq 1 \tfrac{3}{4}$.
This already rules out all cases where $\calM(\alpha) < 4$.

For $\calM(\alpha)=4$, we have $\calM(\eta_0)+\calM(\eta_1)\leq 2\frac14$.
\Cref{L:short sums}(a,b) forces $\calM(\eta_0)=\calM(\eta_1)=1$ and
\[
  \house{\eta_2}^{2}=\house{\alpha}^{2}-2=2.
\]
Thus, $\eta_0$ and $\eta_1$ are roots of unity. Hence, by \Cref{cor:mod p2}, $\eta_2$ must have minimal level divisible by $3$. If $\N(\eta_2)=2$, then \Cref{lem:cassels low height}, together with $\house{\eta_2}^{2}=2$, implies that $\eta_2$ is equivalent to $1+i$, whose minimal level is $4$. This is a contradiction. Therefore, $\N(\eta_2)>2$.

For $\calM(\alpha) = 5$, we have $\calM(\eta_0) + \calM(\eta_1) \leq 3 \tfrac{1}{4}$. By \Cref{L:remove 5}, no $\eta_i$ has minimal level divisible by 5. By \Cref{L:short sums},  this means $\calM(\eta_i) \notin \{1 \tfrac{1}{2}, 1 \tfrac{3}{4}\}$, so $\calM(\eta_0) = 1$.

We now distinguish subcases based on $\max_i\{\N(\eta_i)\}$.

\subsubsection{Subcase: $\max_i\{\N(\eta_i)\} \leq 2$}
As noted above, we may assume that $\house{\alpha}^2 = 5$ and $\eta_0 = 1$. Since $\max_i\{\N(\eta_i)\} \leq 2$,
we must have $\eta_1 = \zeta' (1+\zeta)$, $\eta_2 = \zeta'' (1+\tilde{\zeta})$ for some roots of unity $\zeta, \tilde{\zeta}, \zeta',\zeta''$.
Adding the fact that $\eta_1 \overline{\eta}_1 + \eta_2 \overline{\eta}_2 = 5 - \eta_0 \overline{\eta}_0 = 4$, we obtain $\zeta + \zeta^{-1} = -\tilde{\zeta} - \tilde{\zeta}^{-1}$, so we may further assume that $\tilde{\zeta} = - \zeta$. 
Taking $k=1$ in \eqref{eq:fix norm apart} gives
\[
\zeta' (1+\zeta) + (\zeta')^{-1} \zeta'' (1+\zeta^{-1})(1-\zeta) + (\zeta'')^{-1} (1-  \zeta^{-1})\zeta_{3^{n-1}}^{-1} = 0.
\]
By writing $\beta_1 := \zeta' \zeta_{3^n}$, $\beta_2 := \zeta'' \zeta_{3^n}^2$, we may rewrite this equation as
\begin{equation} \label{eq:X3 no 3}
\beta_1 (1+\zeta) + \beta_1^{-1} \beta_2 (\zeta^{-1}- \zeta) + \beta_2^{-1} (1-\zeta^{-1}) = 0.
\end{equation}
We then get a contradiction by applying \Cref{lem:X3 no 3},
which forces $\zeta$ to have order divisible by 4 whereas $\alpha$ has odd minimal level.

\begin{lemma} \label{lem:X3 no 3}
For any roots of unity $\zeta, \beta_1, \beta_2$ satisfying \eqref{eq:X3 no 3}, $\zeta$ has order $4$ or $8$.
\end{lemma}
\begin{proof}
While this can be established quickly using the method of \emph{torsion closures} from \cite{tetrahedra},
we give a self-contained proof based on a theorem of Mann \cite{mann}
(see also \cite[Theorem~3.1]{poonen-rubinstein} for a concise presentation of a stronger result).
To wit, the six monomials in \eqref{eq:X3 no 3} are roots of unity with zero sum, so up to permutation they take one of the following forms:
\begin{itemize}
    \item[(i)] $\eta_1, -\eta_1, \eta_2, -\eta_2, \eta_3, -\eta_3$ for some $\eta_1, \eta_2, \eta_3$;
    \item[(ii)] $\eta_1, \eta_1 \zeta_3, \eta_1 \zeta_3^2, \eta_2, \eta_2 \zeta_3, \eta_2 \zeta_3^2$ for some $\eta_1, \eta_2$;
    \item[(iii)] $\eta \zeta_3, \eta \zeta_3^2, -\eta \zeta_5, -\eta \zeta_5^2, -\eta \zeta_5^3, -\eta \zeta_5^4$ for some $\eta$.
\end{itemize}

Before continuing, we identify some additional structure among our six monomials. First, they partition into two sets of three that each have product $1$: 
\begin{equation} \label{eq:partition into two sets of three}
\{\beta_1, -\beta_1^{-1} \beta_2 \zeta, -\beta_2^{-1} \zeta^{-1}\} \cup \{\beta_1 \zeta, \beta_1^{-1} \beta_2 \zeta^{-1}, \beta_2^{-1}\}.
\end{equation}
Second, they are permuted by a group action generated by two commuting involutions:
\begin{equation}\label{symmetry between partitions}
    (\zeta, \beta_1, \beta_2) \mapsto (-\zeta^{-1}, \beta_2^{-1}, \beta_1^{-1}),
\end{equation}
which interchanges the two parts of the partition in \eqref{eq:partition into two sets of three};
and
\begin{equation} \label{preserving involution} 
    (\zeta, \beta_1, \beta_2) \mapsto (-\zeta, -\beta_2^{-1} \zeta^{-1}, \beta_1^{-1} \zeta^{-1}),
\end{equation}
which preserves the two parts of the partition in \eqref{eq:partition into two sets of three}. 

We now proceed backwards through the cases of Mann's theorem. 
In each case, we start with an arbitrary matching of the six listed terms with the six monomials in \eqref{eq:partition into two sets of three} and let $\pi_1, \pi_2$ be the products over the two parts; we must then have $\pi_1 = \pi_2 = 1$.

In case (iii), if $\eta \zeta_3$ and $\eta \zeta_3^2$ appear in one part of the partition, then $\pi_1/\pi_2 = \zeta_5^{\pm(1+2+3+4-2i)}$ for some $i \in \{1,2,3,4\}$, which cannot equal 1. Otherwise, $\pi_1/\pi_2 = \zeta_3^{\pm 1} \zeta_5^i$ for some $i\in \{0,1,2,3,4\}$, which again cannot equal 1.

In case (ii), 
if $\eta_1, \eta_1\zeta_3, \eta_1 \zeta_3^2$ appear in one part of the partition, then $\pi_1 = \eta_1^3, \pi_2 = \eta_2^3$ and so $\eta_1^3 = \eta_2^3 = 1$.
Otherwise, $\pi_1 \pi_2 = (\eta_1 \eta_2)^3$ and $(\pi_1/\pi_2)^3 = (\eta_1/\eta_2)^{\pm 3}$, so $\eta_1^3 = \eta_2^3 = \sigma \in \{\pm 1\}$ and each $\eta_i$ has the form $\sigma \zeta_3^j$ for some $j \in \{0,1,2\}$;
now $\pi_1 = \sigma^3 \zeta_3^j$ for some $j \in \{0,1,2\}$
and so $\sigma=1$. In both cases we conclude that $\eta_1, \eta_2$ are cube roots of 1, as are each of the six monomials in \eqref{eq:X3 no 3}. Now each of $\beta_1, -\beta_2^{-1} \zeta^{-1}, \beta_1 \zeta, \beta_2^{-1}$ is a cube root of 1
but 
$(\beta_1)^{-1} (-\beta_2^{-1} \zeta^{-1}) (\beta_1 \zeta) (\beta_2^{-1})^{-1} = -1$ is not, a contradiction.

In case (i), if $\eta_1, \eta_2, \eta_3$ appear in one part of the partition, then $\pi_1/\pi_2 = -1$, a contradiction.
We can thus choose $\eta \in \{\eta_1,\eta_2,\eta_3\}$ such that $\eta$ and $-\eta$ appear in one part of the partition; using that $\pi_1 = \pi_2$, we may write the partition as
\begin{equation}\label{eq :partion with cancelling in pairs}
    \{\eta, -\eta, -\eta^{-2}\} \cup \{ \eta^{-2}, i\eta, -i\eta\}.
\end{equation}
If neither $\beta_1$ nor $\beta_1 \zeta$ matches with $\pm \eta^{-2}$, then $\beta_1 = \pm \eta$, $\beta_1 \zeta = \pm i \eta$, and $\zeta = (\beta_1 \zeta) (\beta_1)^{-1} = \pm i$ has order 4; similarly, if neither $\beta_2^{-1}$ nor $-\beta_2^{-1} \zeta^{-1}$ matches with $\pm \eta^{-2}$, then $\zeta$ has order 4. Otherwise, the monomials matching with $\pm \eta^{-2}$ are either $\beta_1, \beta_2^{-1}$ or $\beta_1 \zeta, -\beta_2^{-1} \zeta^{-1}$; using  \eqref{preserving involution} we may reduce to the former case. We have
\[
\beta_1 = -\eta^{-2}, \qquad \beta_2^{-1} = \eta^{-2},
\qquad
-\beta_2^{-1} \zeta^{-1} = \pm \eta, \qquad \beta_1 \zeta = \pm i \eta;
\]
by writing
\[
\zeta^2 = (\beta_1 \zeta) (\beta_1)^{-1} (\beta_2^{-1} \zeta^{-1})^{-1} (\beta_2^{-1}) 
 = (\pm i \eta) (-\eta^2) (\pm \eta^{-1}) (\eta^{-2}) = \pm i,
\]
we deduce that $\zeta$ has order 8.
\end{proof}

\subsubsection{Subcase: $\max_i\{\N(\eta_i)\} > 2$}
\label{sec:X3 N3}

As noted earlier, we may assume that $\eta_0 = 1$.
By \Cref{L:short sums}(c) and the condition $\max_i\{\N(\eta_i)\} > 2$, either $\calM(\eta_2) = 2$ or $\calM(\eta_2) \geq 2 \tfrac{1}{4}$.
When $\house{\alpha}^2 = 4$ this forces $\calM(\eta_1) = 1, \calM(\eta_2) = 2$.
When $\house{\alpha}^2 = 5$, we deduce that  $\calM(\eta_1) = 2$ or $\calM(\eta_1) \leq 1 \tfrac{3}{4}$; since $\eta_1$ has minimal level coprime to 5 (by \Cref{L:remove 5}),
\Cref{L:short sums}(b) only leaves the options $\calM(\eta_1) \in \{1, 1 \tfrac{2}{3}, 2\}$.

For each choice of $\calM(\eta_1), \calM(\eta_2)$, we identify $\house{\eta_1}^2, \house{\eta_2}^2$ as follows.
When $\calM(\eta_1) < 2$, we apply \Cref{L:short sums} to identify $\eta_1$ up to equivalence, then solve for $\house{\eta_2}^2$ by writing $|\eta_1|^2 + |\eta_2|^2 = \house{\alpha}^2 - 1$. When $\calM(\eta_1) = \calM(\eta_2) = 2$, we combine the condition $\max_i\{\N(\eta_i)\} > 2$ with
\Cref{lem:cor of partial classification}(a) to deduce that one of $\house{\eta_1}^2, \house{\eta_2}^2$ equals 2, as then does the other.

We then identify $\eta_1$ and $\eta_2$ up to equivalence using \Cref{lem:cor of partial classification}, keeping in mind that at least one has minimal level divisible by 3 (by \Cref{cor:mod p2}), but neither has minimal level divisible by 5 if $\house{\alpha}^2 = 5$. 
The results are listed in Table~\ref{table:X3}.

In the two remaining cases, up to equivalence we have 
$\eta_1 = \zeta \beta_1, \eta_2 = \zeta' (\beta_0 + \zeta_3^2 \beta_2)$ where $\zeta,\zeta'$ are roots of unity of order prime to 3 and $\beta_0, \beta_1, \beta_2$ belong to $\QQ(\zeta_{M})$ for $M = 5$ (if $\house{\alpha}^2 = 4$)
or $M = 7$ (if $\house{\alpha}^2 = 5$).
Now
\[
\alpha = 1 + \eta_1 \zeta_9  + \eta_2 \zeta_9^2 = 
1 + \zeta \beta_1 \zeta_9 + \zeta' \beta_0 \zeta_9^2 + \zeta' \beta_2 \zeta_9^8.
\]
Since the expansion of $\alpha \overline{\alpha} \in \{4,5\}$ in powers of $\zeta_9$ does not include $\zeta_9^4$ or $\zeta_9^5$,
the coefficients of $\zeta_9$ and $\zeta_9^{2}$ in that expansion are zero:
\begin{align}
\label{eq:compare beta012}
0 &= (\zeta')^{-1} \overline{\beta}_2 +  \zeta \beta_1 + \zeta^{-1} \zeta' \beta_0 \overline{\beta}_1 \\
\label{eq:substitute beta012}
&= \zeta (\zeta')^{-1} \beta_1 \overline{\beta}_2 + \zeta' \beta_0.
\end{align}
Multiplying \eqref{eq:compare beta012} by $\zeta \beta_1$ and subtracting \eqref{eq:substitute beta012} yields:
\begin{equation} \label{eq:substitute beta012 consequence}
\zeta^2 \beta_1^2  = \zeta' \beta_0 (1 - \beta_1 \overline{\beta}_1).
\end{equation}
Since $\beta_0, \beta_1, \beta_2 \in \QQ(\zeta_M)$,
by \eqref{eq:substitute beta012}, \eqref{eq:substitute beta012 consequence}
we have $\zeta^2 (\zeta')^{-1}, \zeta^{-1} (\zeta')^2 \in \QQ(\zeta_{M})$.
We thus have $\zeta^3, (\zeta')^3 \in \QQ(\zeta_M)$ and hence $\zeta,\zeta' \in \QQ(\zeta_M)$ because $\zeta,\zeta'$ have orders coprime to 3. We deduce that $\alpha$ has minimal level dividing $3^2 \times M$, which is covered by \Cref{lem:short circuit}.

\begin{table}
\caption{Possible values of $\calM(\eta_1), \calM(\eta_2), \house{\eta_1}^2, \house{\eta_2}^2$
and $\eta_1, \eta_2$ up to equivalence when $X=3$ and $\max_i\{\N(\eta_i)\} > 2$.}
\label{table:X3}
\begin{tabular}{c|c|c|c|c|c}
$\calM(\eta_1)$ & $\calM(\eta_2)$ & $\house{\eta_1}^2$ & $\house{\eta_2}^2$ & $\eta_1$ & $\eta_2$ \\
\hline
$1$ & $2$ & $1$ & $2$ & $1$ & $1 - \zeta_{15} + \zeta_{15}^{12}$ \\
\hline
$1$ & $3$ & $1$ & $3$ & --- & --- \\
$1 \tfrac{2}{3}$ & $2 \tfrac{1}{3}$ & $4 \cos^2 \tfrac{\pi}{7}$ & $4 \cos^2 \tfrac{\pi}{14}$ & $1+\zeta_7$ & $1-\zeta_{21} + \zeta_{21}^{13}$ \\
$2$ & $2$ & $2$ & $2$ & --- & ---
\end{tabular}
\end{table}

\subsection{Case: $X>3$}

\Cref{lem:combinatorial constraints power} implies that $p \in \{5,7,11\}$;
in particular, $\alpha$ has odd minimal level.
By \Cref{cor:mod p2}, we cannot have $\calM(\eta_i) = 1$ for all $i$; this forces $X = 4, \calM(\alpha) = 5$.
By \Cref{L:short sums}, each nonzero $\eta_i$ is equivalent to $1$ or $1+\zeta_5$; we now reach a contradiction by comparing \Cref{cor:mod p2} (which forces $p=5$) with \Cref{L:remove 5} (which excludes $p=5$).
\end{proof}

\section{Large primes in the minimal level}
\label{sec:large primes}

In this section, we resolve Step 2 of the proof of \Cref{T:main} in cases (b)--(d) of \Cref{cor:RW truncation}; see respectively \Cref{prop:step 2 case b}, \Cref{prop:step 2 case c}, and \Cref{prop:step 2 case d}.

\subsection{Initial setup}

We introduce the general setup for resolving Step 2, following along the lines of \Cref{P:robinson wurtz extract2} but incorporating the combinatorial constraints identified in \S\ref{sec:combinatorial} to eliminate smaller values of $X$, for which the calculus of Cassels heights gets prohibitively complicated.
This improvement alone will suffice to handle those $\alpha$ covered by \Cref{cor:RW truncation}(b);
we will need some further ingredients to manage the other cases.

\begin{remark} \label{R:apply combinatorial constraints}
For $\alpha$ covered by \Cref{cor:RW truncation}(b)--(d),
suppose that
the minimal level $N$ of $\alpha$ is squarefree.
Let $p$ denote the largest prime factor of $N$ and assume that $p > N_1$.
Set notation as in \Cref{L:rep-prime-case}.
As in the proof of \Cref{P:robinson wurtz extract2}, let $\alpha_1,\dots,\alpha_X$ denote the nonzero values of $\eta_i$ sorted so that $0 < \N(\alpha_1) \leq \cdots \leq \N(\alpha_X)$.

We will separate cases according to $p$ and $X$.
Since $p \geq 11$, $p$ cannot divide the minimal level of $\house{\alpha}^2$,
so we may exclude all pairs $(p,X)$
covered by \Cref{lem:combinatorial constraints}.

As in \Cref{P:robinson wurtz extract2}, we write $S$ for the right-hand side of \eqref{eq:rep-prime-case} (again overriding the notation of \Cref{L:rep-prime-case}) and tabulate lower bounds on $S$ based on \Cref{L:short sums}, the inequality $\N(\alpha) \leq \sum_{i=1}^X \N(\alpha_i)$, and other available constraints (including \Cref{R:handle Y when 0}). In these tables,
we write $\Sigma$ as shorthand for $\sum_{1 \leq i < j \leq X} \calM(\alpha_i-\alpha_j)$,
and write $\times n$ to indicate $n$-fold repetition.
\end{remark}

\begin{remark} \label{rem:minimal level bounded below}
Before continuing, we note that the minimal level of $\house{\alpha}^2$ must divide $N$.
When $\calM(\alpha) = 3 \tfrac{1}{4}, 3\tfrac{1}{3}, 3\frac{1}{2}$,
by \eqref{eq:rw cutoff} this implies that $N$ is divisible by $15, 7, 5$, respectively.
We will take advantage of this in certain situations by showing $N$ equals $p$ times a small cofactor incompatible with this constraint; 
see for example \Cref{R:handle Y when 0}. For a distinct but related observation, see \Cref{rem:small cofactor}.
\end{remark}

\begin{remark} \label{R:handle Y when 0}
When the $\alpha_i$ are all known to be equal (e.g., by \Cref{lem:combinatorial constraints2}),
we may proceed as follows. Here we assume $p \leq 31$, as all larger $p$ will get ruled out 
solely using \Cref{R:apply combinatorial constraints}
(see \eqref{eq:lower bound X7}).

Factor $\alpha$ as $\beta \gamma$ where $\beta$ is a sum of $X$ distinct $p$-th roots of unity and $\gamma$ is a cyclotomic integer of minimal level coprime to $p$.
Since we can conjugate $\beta$ and $\gamma$ independently, 
we have $\house{\alpha} = \house{\beta} \times \house{\gamma}$;
in particular, both $\house{\beta}^2$ and $\house{\gamma}^2$
are restricted by \Cref{T:robinson-wurtz}. 
\begin{itemize}
    \item 
We cannot have $\house{\beta} = 1$, as then $N$ could not be the minimal level of $\alpha$.
\item
 If $\house{\gamma} = 1$, then $\alpha = \beta$ and hence $N = p$;
 using \Cref{rem:minimal level bounded below},
 this forces $\house{\alpha}^2 \in \{4,5\}$.
Now using \Cref{lem:exhaust} (for $p = 31$) and \Cref{lem:short circuit} (for $p < 31$), we deduce that all such $\alpha$ are covered by \Cref{T:main}.
\item
If $\house{\gamma} > 1$, then by \Cref{lem:cassels low height} we have
\[
\house{\beta}^2, \house{\gamma}^2 \in \{2, 4 \cos^2 \tfrac{\pi}{5}\} \cup [3, \infty);
\]
since $2 \times 4\cos^2 \tfrac{\pi}{5} > 5$,
the only way to ensure that $\house{\alpha}^2 = \house{\beta}^2 \house{\gamma}^2 \leq 5$
is to take $\house{\alpha}^2 = 4, \house{\beta}^2 = \house{\gamma}^2 = 2$.
Since $\beta$ and $\gamma$ are then classified by \Cref{lem:cassels low height},
we again conclude that all such $\alpha$ are covered by \Cref{T:main}. 
\end{itemize}

We may thus assume throughout this section that the $\alpha_i$ are not all equal.
This means on one hand that the pair $(p,X)$ is not covered by either \Cref{lem:combinatorial constraints} or \Cref{lem:combinatorial constraints2},
and on the other hand $\Sigma \geq X-1$ by \eqref{eq:prime case next Y0}.
\end{remark}

\subsection{$\calM(\alpha) \leq 3 \tfrac{1}{2}$}

\begin{prop} \label{prop:step 2 case b}
Let $\alpha$ be a cyclotomic integer 
with $\house{\alpha}^2 < 5.01$ and $3 \tfrac{1}{4} \leq \calM(\alpha) \leq 3 \tfrac{1}{2}$.
Assume further that the minimal level of $\alpha$ is squarefree and divisible by some prime $p$ greater than $7$.
Then $\alpha$ is covered by \Cref{T:main}.
\end{prop}
\begin{proof}
Set notation as in \Cref{R:apply combinatorial constraints}.
We separate cases based on $p$.

\subsubsection{Case: $p = 11$}
By \Cref{L:rep-prime-case}(d),(e), we have $X \leq \frac{p+1}{2} = 6$.
Meanwhile, by \Cref{R:handle Y when 0} we may assume $X \geq 5$.
From \Cref{L:rep-prime-case}  and the condition $\calM(\alpha) \leq 3 \tfrac{1}{2}$, we have the inequality:
\begin{equation} \label{CasselsBoundfor11}
35 \geq (p-X) \sum_{i=1}^X \calM(\alpha_i) + \sum_{1\leq i<j\leq X} \calM(\alpha_i - \alpha_j) =: S.
\end{equation} 

For $X = 5$, after applying \Cref{lem:exhaust} to the pair $(2^2 \times 3 \times 5 \times 7 \times 11, 5)$, we can assume that $\N(\alpha) \geq 6$ as in the proof of \Cref{P:robinson wurtz extract2}. 
We obtain a contradiction by tabulating cases:

\begin{center}
\setlength{\tabcolsep}{6 pt}
\begin{tabular}{cc|cc|c|c}
\hline\hline
\multicolumn{2}{c}{$\N(\alpha_i)$} & \multicolumn{2}{c}{$\calM(\alpha_i)$} & $\Sigma$ & $S$ \\
\hline
$1$ $(\times 4)$ & $\geq 2$ & $1$ $(\times 4)$ & $\geq 3/2$ & $\geq 4$ &  $\geq 37$ \\

$\geq 1$ $(\times 3)$ & $\geq 2$ $(\times 2)$ & $\geq 1$ $(\times 3)$ & $\geq 3/2$ $(\times 2)$ & & $\geq 36$ \\
\hline \hline
\end{tabular}
\end{center}

For $X = 6$, we have $X>5=\tfrac{p-1}{2}$, and so $\calM(\alpha)\geq \tfrac{p+3}{4} = 3\tfrac{1}{2}$ by \Cref{L:rep-prime-case}(e). From our hypothesis, we deduce that in fact $\calM(\alpha)= \tfrac{p+3}{4}=3\tfrac{1}{2}$. 
Since $X=6=\tfrac{p+1}{2}$, we are in the equality case of  \Cref{L:rep-prime-case}(d), which is covered by \Cref{R:handle Y when 0}.

\subsubsection{Case: $p = 13$}
By \Cref{L:rep-prime-case}(e),
\begin{equation} \label{eq:case b upper bound on X}
p \geq 13 \Longrightarrow  \calM(\alpha) \leq 3 \tfrac{1}{2} < \tfrac{p+3}{4}
\Longrightarrow X \leq \tfrac{p-1}{2}.
\end{equation}
For $p=13$, we thus have $X \leq 6$.
Meanwhile, by \Cref{R:handle Y when 0} we may assume $X \geq 5$.
From \Cref{L:rep-prime-case}(b)  and the bound $\calM(\alpha) \leq 3 \tfrac{1}{2}$, we have the inequality:
\begin{equation}\label{CasselsBoundfor13}
42 \geq (13-X) \sum_{i=1}^X \calM(\alpha_i) + \sum_{1\leq i<j\leq X} \calM(\alpha_i - \alpha_j) =: S.
\end{equation}

For $X = 5$, we obtain a contradiction by tabulating cases:

\begin{center}
\setlength{\tabcolsep}{6 pt}
\begin{tabular}{cc|cc|c|c}
\hline\hline
\multicolumn{2}{c}{$\N(\alpha_i)$} & \multicolumn{2}{c}{$\calM(\alpha_i)$} & $\Sigma$ & $S$ \\
\hline
$1$ $(\times 5)$ & & $1$ $(\times 5)$ & & $\geq 4$ & $\geq 44$ \\

$\geq 1$ $(\times 4)$ & $\geq 2$ & $\geq 1$ $(\times 4)$ & $\geq 3/2$  & & $\geq 44$ \\
\hline \hline
\end{tabular}
\end{center}

For $X = 6$, \Cref{L:rep-prime-case}(b) implies $\calM(\alpha) \geq 3 \tfrac{1}{2}$. This puts us in the equality case of  \Cref{L:rep-prime-case}(b), which is covered by \Cref{R:handle Y when 0}.

\subsubsection{Case: $p \geq 17$}
By \eqref{eq:case b upper bound on X}, $X(p-X)$ grows monotonically with $X$.
If $X \geq 5$, then \Cref{L:rep-prime-case}(b) gives the contradiction
\[
\calM(\alpha) \geq \frac{X(p-X)}{p-1} \geq \frac{5(p-5)}{p-1} \geq \frac{5(17-5)}{17-1} > 3 \tfrac{1}{2}.
\]
Hence $X \leq 4$, contrary to \Cref{lem:combinatorial constraints}.
\end{proof}

\begin{remark}
At this point, we can already extend \Cref{T:cassels partial classification} to case (b) of 
\Cref{cor:RW truncation} by combining \Cref{prop:Step 1},
\Cref{prop:step 2 case b}, and \Cref{prop:list 420 part1}.
\end{remark}

\subsection{$\calM(\alpha) = 4$}

When $\house{\alpha}^2 \in \{4,5\}$, 
the combinatorics of Cassels heights starts to become somewhat complicated. We tame it firstly by raising $N_1$ (thus pushing some work into Step 3), and secondly by bringing in some additional tools.

\begin{remark} \label{rem:small cofactor}
When $\house{\alpha}^2 \in \{4,5\}$, \Cref{rem:minimal level bounded below} ceases to be useful. Instead, in some situations where we show that $N$ divides $p$ times a small cofactor, we can appeal to \Cref{lem:short circuit} to confirm that $\alpha$ is covered by \Cref{T:main}.
\end{remark}

\begin{remark} \label{R:short sums 11}
When $\house{\alpha}^2 = 4$ and $p \leq 13$,
we can apply \Cref{L:remove p castle 5} to reduce to the case where the minimal level $N$ of $\alpha$ is not divisible by 11;
that is, we may assume that $N$ divides $3 \times 5 \times 7 \times 13$.
For a corresponding statement with $\house{\alpha}^2 = 5$, see \Cref{R:short sums 7}.
\end{remark}

\begin{prop} \label{prop:step 2 case c}
Let $\alpha$ be a cyclotomic integer 
with $\house{\alpha}^2 = 4$.
Assume further that the minimal level of $\alpha$ is squarefree and divisible by some prime greater than $11$.
Then $\alpha$ is covered by \Cref{T:main}.
\end{prop}
\begin{proof}
Set notation as in \Cref{R:apply combinatorial constraints}. 
We separate cases based on $p$.

\subsubsection{Case: $p = 13$}
\label{subsec:castle 4 prime 13}

By \Cref{L:rep-prime-case}(d), $X \leq \tfrac{p+1}{2} = 7$.
Meanwhile, by \Cref{R:handle Y when 0} we may assume $X \geq 5$.
By \Cref{R:short sums 11}, we may assume that
$N$ is not divisible by 11. From \Cref{L:rep-prime-case}  and the condition $\calM(\alpha) = 4$, we have the equality:
\begin{equation}\label{CasselEqualityfor 4,13}
48 = (p-X) \sum_{i=1}^X \calM(\alpha_i) + \sum_{1\leq i<j\leq X} \calM(\alpha_i - \alpha_j) =: S.
\end{equation} 

For $X = 5$, 
after applying \Cref{lem:exhaust} to the pair
$(3 \times 5 \times 7 \times 13, 5)$,
we may assume that $\N(\alpha) > 5$. We now tabulate cases:
\begin{center}
\setlength{\tabcolsep}{6 pt}
\begin{tabular}{cc|cc|c|c}
\hline\hline
\multicolumn{2}{c}{$\N(\alpha_i)$} & \multicolumn{2}{c}{$\calM(\alpha_i)$} & $\Sigma$ & $S$ \\
\hline
$1$ $(\times 4)$ & $2$ & $1$ $(\times 4)$ & $\geq 3/2$ & $\geq 4$ &   $\geq 48$ \\
$1$ $(\times 3)$ & $\geq 2$ $(\times 2)$ & $1$ $(\times 3)$ & $\geq 3/2$ $(\times 2)$ & $\geq 6$ & $\geq 54$ \\

$\geq 1$ $(\times 4)$ & $\geq 3$ & $\geq 1 (\times 4)$ & $\geq 2$  & $\geq 4$ & $\geq 52$ \\
\hline \hline
\end{tabular}
\end{center}
By \eqref{CasselEqualityfor 4,13}, we may assume that $\alpha_1 = \dots = \alpha_4 = 1$, and then $\alpha_5$ satisfies $\N(\alpha_5) = 2$, $\calM(\alpha_5) = \tfrac{3}{2}$,
and $\calM(\alpha_5-1) = 1$.
By \Cref{L:short sums}, we may assume that $\alpha_5 = 1 + \zeta_5$; thus $N$ divides $5 \times 13$, which is covered by \Cref{lem:short circuit}.

For $X = 6$, we tabulate cases:
\begin{center}
\setlength{\tabcolsep}{6 pt}
\begin{tabular}{cc|cc|c|c}
\hline\hline
\multicolumn{2}{c}{$\N(\alpha_i)$} & \multicolumn{2}{c}{$\calM(\alpha_i)$} & $\Sigma$ & $S$ \\
\hline
$1$ $(\times 6)$ & & $1$ $(\times 6)$ & & $\geq 5$ & $\geq 47$ \\
$1$ $(\times 5)$ & $\geq 2$ & $1$ $(\times 5)$ & $\geq 3/2$ & $\geq 5$ & $\geq 50.5$ \\

$\geq 1$ $(\times 4)$ & $\geq 2$ $(\times 2)$ & $\geq 1$ $(\times 4)$ & $\geq 3/2$ $(\times 2)$ & & $\geq 49$ \\
\hline \hline
\end{tabular}
\end{center}
By \eqref{CasselEqualityfor 4,13}, the $\alpha_i$ are all roots of unity and $\Sigma=6 < 2X-2$.
By \eqref{eq:prime case next Y1}, we may assume that $\alpha_1 = \cdots = \alpha_5 = 1$. Now by \Cref{L:short sums}, $\Sigma = 5 \calM(\alpha_6-1)$ is either equal to 5 or at least  $5 \times \frac{3}{2}=7.5 > 6$, a contradiction.

For $X = 7$, we have $\calM(\alpha) = 4 = \tfrac{p+3}{4}$ and $X = \tfrac{p+1}{2}$, so we are in the equality case of \Cref{L:rep-prime-case}(d). Hence up to equivalence,
we may take the $\alpha_i$ to be $1$ $(\times 6)$, $-\zeta_3$.
Thus $N$ divides $3 \times 13$, which is covered by \Cref{lem:short circuit}.

\subsubsection{Case: $p \geq 17$}
By \Cref{L:rep-prime-case}(e),
\begin{equation} \label{eq:case c upper bound on X}
p \geq 17 \Longrightarrow  \calM(\alpha) = 4 < \tfrac{p+3}{4}
\Longrightarrow X \leq \tfrac{p-1}{2};
\end{equation}
so $X(p-X)$ grows monotonically with $X$.
If $X \geq 6$, then \Cref{L:rep-prime-case}(b) gives the contradiction
\[
\calM(\alpha) \geq \frac{X(p-X)}{p-1} \geq \frac{6(p-6)}{p-1} \geq \frac{6(17-6)}{17-1} > 4.
\]
Hence $X \leq 5$; for $p \geq 23$ this contradicts \Cref{lem:combinatorial constraints}, while for $p = 17, 19$ it gives a case covered by \Cref{lem:combinatorial constraints2} and thus addressed by \Cref{R:handle Y when 0}.
\end{proof}

\subsection{$\calM(\alpha) = 5$}

When $\house{\alpha}^2 = 5$, the calculus of Cassels heights becomes rather more complicated even compared to $\house{\alpha}^2 = 4$; we react by raising $N_1$ again.
In addition, we avail ourselves of a much stronger version of \Cref{R:short sums 11}.

\begin{remark} \label{R:short sums 7}
When $\house{\alpha}^2 = 5$, 
by \Cref{L:remove 5}, we may assume that the minimal level $N$ of $\alpha$ is not divisible by $5$.
For $p \leq 23$, by \Cref{L:remove p castle 5}, we may further assume that $N$ is not divisible by 11 or 23.

We can go somewhat further by comparing with \eqref {eq:multiplicative order}: since  $5$ has order $4$ in $(\ZZ/13\ZZ)^\times$, 
when $p \leq 17$ we can apply \Cref{L:remove p castle 5} to reduce to the case where $N$ is not divisible by 7.
This logic fails for $p=19$; however, in this case we can turn around and argue that $N$ \emph{must} be divisible by 7, as otherwise \Cref{L:remove p castle 5} excludes the prime factor 19.

In particular, when tabulating lower bounds on $S$, if the contribution of $\calM(\eta)$ is greater than 1, then it is  at least $1 \tfrac{5}{6}$ if $p=17$ or at least $1 \tfrac{2}{3}$ if $p = 19$.
\end{remark}

One more tool we need is a lemma that allows us to pit two large primes in the minimal level 
in opposition to each other. In lieu of stating the strongest possible lemma of this form, we stick to a simpler statement which will suffice for our purposes.

\begin{lemma} \label{L:rep two primes}
Let $p' < p$ be two odd primes.
Let $\alpha \in \QQ(\zeta_N)$ be a cyclotomic integer whose minimal level $N$ is exactly divisible by both $p$ and $p'$.
With notation as in \Cref{L:rep-prime-case} (for the prime $p$),
suppose that $\calM(\eta_i) = 1$ for each $i \in S$.
Define $X'$ by applying \Cref{L:rep-prime-case} for the prime $p'$.
Then 
\begin{equation} \label{eq:rep-two-primes}
(p-1) \calM(\alpha) \geq X(p-X) + \left(2 - \frac{2}{p'-1} \right) \left( \binom{X}{2} - \binom{X-\min\{X,X'\}+1}{2} \right).
\end{equation}
\end{lemma}
\begin{proof}
Define an equivalence relation $\sim$ on $S$ by specifying that $i \sim j$ if $\eta_i/ \eta_j \in \QQ(\zeta_{N/(pp')})$. 
By the definition of $X'$, $S$ must separate into at least $\min\{X,X'\}$ equivalence classes.

For $i < j \in S$ with $i \not\sim j$, we may apply \Cref{L:rep-prime-case} with $\alpha,p$ replaced by $\eta_i-\eta_j, p'$ to deduce that
$\calM(\eta_i-\eta_j) \geq \tfrac{2(p'-2)}{p'-1}$. 
The number of such pairs $i,j$ is minimized 
when all but one equivalence class is a singleton
(as in the proof of \eqref{eq:rep-prime-case extended}),
yielding a lower bound of $ \binom{X}{2} - \binom{X-\min\{X,X'\}+1}{2}$.
Applying \eqref{eq:rep-prime-case} now yields the claim.
\end{proof}

\begin{prop} \label{prop:step 2 case d}
Let $\alpha$ be a cyclotomic integer 
with $\house{\alpha}^2 = 5$.
Assume further that the minimal level of $\alpha$ is squarefree and divisible by some prime greater than $13$.
Then $\alpha$ is covered by \Cref{T:main}.    
\end{prop}
\begin{proof}
Set notation as in \Cref{R:apply combinatorial constraints}. 
We separate cases based on $p$.

\subsubsection{Case: $p = 17$}
\label{subsec:castle 5 prime 17}

By \Cref{L:rep-prime-case}(d), $X \leq \tfrac{p+1}{2} = 9$.
Meanwhile, by \Cref{R:handle Y when 0} we may assume $X \geq 6$.
By \Cref{R:short sums 7}, we may assume that $N$ divides $3 \times 13 \times 17$;
we may also assume that 13 divides $N$, as otherwise \Cref{lem:short circuit} applies. From \Cref{L:rep-prime-case}  and the condition $\calM(\alpha) = 5$, we have the equality:
\begin{equation}\label{CasselEqualityfor 5,17}
80 = (p-X) \sum_{i=1}^X \calM(\alpha_i) + \sum_{1\leq i<j\leq X} \calM(\alpha_i - \alpha_j) =: S.
\end{equation} 

For $X = 6$, note that if $\N(\alpha) = 6$,
then we may apply \Cref{L:rep two primes} with $p=17, p'=13$, and $X' \geq 4$ by \Cref{lem:combinatorial constraints}, but now \eqref{eq:rep-two-primes} yields the contradiction
\[
80 = (p-1) \calM(\alpha) \geq 6 \times 11 + \left(2 - \tfrac{2}{12} \right) (15 - 3)
= 88.
\]
We may thus assume now that $\N(\alpha) > 6$.  We obtain a contradiction by tabulating cases (now noting \Cref{R:short sums 7}):
\begin{center}
\setlength{\tabcolsep}{6 pt}
\begin{tabular}{cc|cc|c|c}
\hline\hline
\multicolumn{2}{c}{$\N(\alpha_i)$} & \multicolumn{2}{c}{$\calM(\alpha_i)$} & $\Sigma$ & $S$ \\
\hline
$1$ $(\times 5)$ & $\geq 2$ & $1$ $(\times 5)$ & $\geq 11/6$ & $\geq 5$ &  $\geq 80.16$ \\
$\geq 1$ $(\times 4)$ & $\geq 2$ $(\times 2)$ & $\geq 1$ $(\times 4)$ & $\geq 11/6$ $(\times 2)$ & & $\geq 84.33$ \\
\hline \hline
\end{tabular}
\end{center}

For $X = 7$, we tabulate cases:
\begin{center}
\setlength{\tabcolsep}{6 pt}
\begin{tabular}{cc|cc|c|c}
\hline\hline
\multicolumn{2}{c}{$\N(\alpha_i)$} & \multicolumn{2}{c}{$\calM(\alpha_i)$} & $\Sigma$ & $S$ \\
\hline
$1$ $(\times 7)$ & & $1$ $(\times 7)$ & & $\geq 6$ & $\geq 76$ \\
$1$ $(\times 6)$ & $\geq 2$ & $1$ $(\times 6)$ & $\geq 11/6$ & $\geq 6$ & $\geq 84.33$ \\

$\geq 1$ $(\times 5)$ & $\geq 2$ $(\times 2)$ & $\geq 1$ $(\times 5)$ & $\geq 11/6$ $(\times 2)$ & & $\geq 86.66$ \\
\hline \hline
\end{tabular}
\end{center}
By \eqref{CasselEqualityfor 5,17}, the $\alpha_i$ are all roots of unity and $\Sigma=10 < 2X-2$.
By \eqref{eq:prime case next Y1}, we may assume that $\alpha_1 = \cdots = \alpha_6 = 1$. Now by \eqref{CasselEqualityfor 5,17}, $\calM(\alpha_7-1)=1\frac{2}{3}$, but this contradicts \Cref{R:short sums 7}.

For $X = 8$, we again tabulate cases:
\begin{center}
\setlength{\tabcolsep}{6 pt}
\begin{tabular}{cc|cc|c|c}
\hline\hline
\multicolumn{2}{c}{$\N(\alpha_i)$} & \multicolumn{2}{c}{$\calM(\alpha_i)$} & $\Sigma$ & $S$ \\
\hline
$1$ $(\times 8)$ & & $1$ $(\times 8)$ & & $\geq 7$ &  $\geq 79$ \\
$1$ $(\times 7)$ & $\geq 2$ & $1$ $(\times 7)$ & $\geq 11/6$ & $\geq 7$ & $\geq 86.5$ \\

$\geq 1$ $(\times 6)$ & $\geq 2$ $(\times 2)$ & $\geq 1$ $(\times 6)$ & $\geq 11/6$ $(\times 2)$ & & $\geq 87$ \\
\hline \hline
\end{tabular}
\end{center}
By \eqref{CasselEqualityfor 5,17}, the $\alpha_i$ are all roots of unity and $\Sigma =8 < 2X-2$.
By \eqref{eq:prime case next Y1}, we may assume that $\alpha_1 = \cdots = \alpha_7 = 1$. Now by 
\Cref{L:short sums}, $\Sigma = 7 \calM(\alpha_8-1)$ is either equal to 7 or at least  $7 \times \frac{3}{2}=10.5 > 8$, a contradiction.

For $X = 9$, we have $\calM(\alpha) = 5 = \tfrac{p+3}{4}$ and $X = \tfrac{p+1}{2}$, so we are in the equality case of \Cref{L:rep-prime-case}(d). Hence up to equivalence,
we may take the $\alpha_i$ to be $1$ $(\times 8)$, $-\zeta_3$.
Thus $N$ divides $3 \times 17$, which is covered by \Cref{lem:short circuit}.

\subsubsection{Case: $p = 19$}
\label{subsec:castle 5 prime 19}

By \Cref{L:rep-prime-case}(e),
\begin{equation} \label{eq:case d upper bound on X}
p \geq 19 \Longrightarrow  \calM(\alpha) = 5 < \tfrac{p+3}{4}
\Longrightarrow X \leq \tfrac{p-1}{2}.
\end{equation}
For $p = 19$ this yields $X \leq 9$.
Meanwhile, by \Cref{R:handle Y when 0} we may assume $X \geq 6$.
By \Cref{R:short sums 7}, we may assume that $N$ divides $3 \times 7 \times 13 \times 17 \times 19$ and is divisible by $7$. From \Cref{L:rep-prime-case}  and the condition $\calM(\alpha) = 5$, we have the equality:
\begin{equation}\label{CasselEqualityfor 5,19}
90 = (p-X) \sum_{i=1}^X \calM(\alpha_i) + \sum_{1\leq i<j\leq X} \calM(\alpha_i - \alpha_j) =: S.
\end{equation} 

For $X = 6$, we note that if $\N(\alpha) = 6$,
 then we may apply \Cref{L:rep two primes} with $p=19, p'=7$, and $X' \geq 3$ by \Cref{lem:combinatorial constraints}, but now \eqref{eq:rep-two-primes} yields the contradiction
\[
90 = (p-1) \calM(\alpha) \geq 6 \times 13 + \left(2 - \tfrac{2}{6} \right) (15 - 6)
= 93.
\]
We may thus assume now that $\N(\alpha) > 6$. We obtain a contradiction by tabulating cases (now noting \Cref{R:short sums 7}):
\begin{center}
\setlength{\tabcolsep}{6 pt}
\begin{tabular}{cc|cc|c|c}
\hline\hline
\multicolumn{2}{c}{$\N(\alpha_i)$} & \multicolumn{2}{c}{$\calM(\alpha_i)$} & $\Sigma$ & $S$ \\
\hline
$1$ $(\times 5)$ & $\geq 2$ & $1$ $(\times 5)$ & $\geq 5/3$ & $\geq 5$ &  $\geq 91.66$ \\
$\geq 1$ $(\times 4)$ & $\geq 2$ $(\times 2)$ & $\geq 1$ $(\times 4)$ & $\geq 5/3$ $(\times 2)$ & & $\geq 95.33$ \\
\hline \hline
\end{tabular}
\end{center}

For $X = 7$,  we tabulate cases:
\begin{center}
\setlength{\tabcolsep}{6 pt}
\begin{tabular}{cc|cc|c|c}
\hline\hline
\multicolumn{2}{c}{$\N(\alpha_i)$} & \multicolumn{2}{c}{$\calM(\alpha_i)$} & $\Sigma$ & $S$ \\
\hline
$1$ $(\times 7)$ & & $1$ $(\times 7)$ & & $\geq 6$ & $\geq 90$ \\
$\geq 1$ $(\times 6)$ & $\geq 2$ & $\geq 1$ $(\times 6)$ & $\geq 5/3$ & & $\geq 92$ \\
\hline \hline
\end{tabular}
\end{center}
By \eqref{CasselEqualityfor 5,19}, the $\alpha_i$ are all roots of unity, and we may assume that $\alpha_1 = \cdots = \alpha_6 = 1$ and $\calM(\alpha_7-1)=1$. Hence up to equivalence,
we may take $\alpha_7 = -\zeta_3$; then $N$ divides $3 \times 19$, which is covered by \Cref{lem:short circuit}.

For $X = 8$, we obtain a contradiction by tabulating cases:
\begin{center}
\setlength{\tabcolsep}{6 pt}
\begin{tabular}{cc|cc|c|c}
\hline\hline
\multicolumn{2}{c}{$\N(\alpha_i)$} & \multicolumn{2}{c}{$\calM(\alpha_i)$} & $\Sigma$ & $S$ \\
\hline
$1$ $(\times 8)$ & & $1$ $(\times 8)$ & & $\geq 7$ & $\geq 95$ \\
$\geq 1$ $(\times 7)$ & $\geq 2$ & $\geq 1$ $(\times 7)$ & $\geq 5/3$ & & $\geq 95.33$ \\
\hline \hline
\end{tabular}
\end{center}

For $X = 9$, we have $\calM(\alpha)=\frac{X(p-X)}{p-1}$, so we are in the equality case of \Cref{L:rep-prime-case}(b). Hence this case is covered by \Cref{R:handle Y when 0}.

\subsubsection{Case: $p = 23$}

As per \Cref{R:short sums 7}, we may cover this case using \Cref{L:remove p castle 5}.

\subsubsection{Case: $p \geq 29$}

For $p > 19$, \eqref{eq:case d upper bound on X} implies that $X(p-X)$ grows monotonically with $X$. We may thus rule out $X \geq 7$ by computing that
\begin{equation} \label{eq:lower bound X7}
p \geq 23, X \geq 7 \Longrightarrow \calM(\alpha) \geq \frac{X(p-X)}{p-1} \geq \frac{7(p-7)}{p-1} \geq \frac{7(23-7)}{23-1} > 5.
\end{equation}
\Cref{lem:combinatorial constraints} gives $X\ge 6$, and hence, by~\eqref{eq:lower bound X7} we have $X=6$. For $p\ge 37$, this contradicts \Cref{lem:combinatorial constraints}; while for $p=29,31$, the case $X=6$ is covered by \Cref{lem:combinatorial constraints2} and thus addressed by \Cref{R:handle Y when 0}.
\end{proof}

\section{Small primes in the minimal level}
\label{sec:small primes}

In this section, we resolve Step 3 of the proof of \Cref{T:main} in cases (c), (d) of \Cref{cor:RW truncation}, following the approach introduced in \S\ref{sec:class groups};
see respectively \Cref{prop:step3 case 4} and \Cref{prop:step3 case 5}.
We skip case (b) because it was already covered in \Cref{prop:list 420 part1}.

\subsection{$\calM(\alpha) = 4$}

When $\house{\alpha}^2 = 4$, we took $N_1 = 11$ and so we must consider the case where $\alpha$ has minimal level dividing $3 \times 5 \times 7 \times 11$. We do this by removing the prime 11
using \Cref{L:remove p castle 5}, then continuing with a variant of \Cref{lem:short circuit}.

\begin{prop} \label{prop:step3 case 4}
For $N = 3 \times 5 \times 7 \times 11$, 
every cyclotomic integer $\alpha \in \QQ(\zeta_N)$ with $\house{\alpha}^2 = 4$ is covered by \Cref{T:main}.  
\end{prop}
\begin{proof}
By applying \Cref{L:remove 5} as in \Cref{R:short sums 11}, we may exclude the prime factor 11 from the minimal level of $\alpha$; that is, we may assume that $\alpha \in \QQ(\zeta_{N'})$ with $N' := 3 \times 5 \times 7$.
We now apply \Cref{prop:fixed house setup} with $K = \QQ(\zeta_{N'}), \alpha_0 = 2$.
Define 
\[
\alpha_1 := 1+\zeta_7 + \zeta_7^3, \qquad \alpha_2 := 1-\zeta_{15} + \zeta_{15}^{12}.
\]
By \eqref{eq:multiplicative order} we have $\#T_K = 2$;
we choose $T_K = \{\frakp_1, \frakp_2\}$
with $\frakp_1 = (\alpha_1, \alpha_2)$ and $\frakp_2 = (\alpha_1, \overline{\alpha}_2)$.
For each $u \in U_K = [-1, 0, 1] \times [-1, 0, 1]$, the fractional ideal $J_u$ defined in~\eqref{eq:ideal Ju} admits a generator as indicated in this table:
\[
\begin{array}{c|c|c|c}
 & -1 & 0 & 1 \\
\hline
-1 & \overline{\alpha}_1^2 & \overline{\alpha}_1 \overline{\alpha}_2 & \overline{\alpha}_2^2 \\
\hline
0 & \overline{\alpha}_1 \alpha_2 & 2 & \alpha_1 \overline{\alpha}_2 \\
\hline
1 & \alpha_2^2 & \alpha_1 \alpha_2 & \alpha_1^2
\end{array}
\]
The entries $2, \alpha_2^2, \overline{\alpha}_2^2$ appear in the infinite families of \Cref{T:cassels}, while the others are equivalent to entries of \Cref{table:exceptional classes}.
\end{proof}

\subsection{$\calM(\alpha) = 5$}

When $\house{\alpha}^2 = 5$, we took $N_1 = 13$ and so we must consider the case where $\alpha$ has minimal level dividing $3 \times 5 \times 7 \times 11 \times 13$.
For this, we remove some prime factors using \Cref{L:remove 5} and \Cref{L:remove p castle 5}, and handle the rest as follows.

\begin{lemma} \label{lem:step3 case 5 remove 3}
For $N = 3 \times 13$,
there is no cyclotomic integer $\alpha \in \QQ(\zeta_{N})$ with $\house{\alpha}^2 = 5$.
\end{lemma}
\begin{proof}
Let $K_0$ be the index-3 subfield of $\QQ(\zeta_N)$
and set $\alpha_1 := \Norm_{\QQ(\zeta_N)/K_0}(\alpha)$,
so that $\house{\alpha_1}^2 = 5^3$.
We apply \Cref{prop:fixed house setup}
with $K = K_0(\sqrt{5}), \alpha_0 = 5^{3/2}$.
By \eqref{eq:multiplicative order} we have $\#T_K = 1$ (so we may identify $\ZZ^{T_K}$ with $\ZZ$) and $U_K = \{-3,\dots,3\}$. 
For $u \in \ZZ$, the fractional ideal $(\alpha_0)/J_u$ is extended from $K_0$ if and only if $u$ is odd, so $\alpha_1$ must correspond to some $u \in \{\pm 1, \pm 3\}$.

To rule out these possibilities, we apply the procedure described in~\Cref{rem:class group bound}. For $u = 2$, $J_u$ is principal generated by $\beta_0 = 2\sqrt{-3} \pm \sqrt{13}$, but neither $\pm \beta_0$ has a square root in $K$: such a square root would generate an abelian extension of $\QQ$,
which would imply that one of $\pm \beta_0 / \overline{\beta}_0 = \pm \beta_0^2/5$ is a square in $\QQ(\beta_0)$, which they are not.
We conclude that none of $\pm 1, \pm 3$ vanishes in $\coker(v_{T,K})$, proving the claim.
\end{proof}

\begin{prop} \label{prop:step3 case 5}
For $N = 3 \times 5 \times 7 \times 11 \times 13$,
every cyclotomic integer $\alpha \in \QQ(\zeta_N)$ with $\house{\alpha}^2 = 5$ is covered by \Cref{T:main}. 
\end{prop}
\begin{proof}
By applying \Cref{L:remove 5} as in \Cref{R:short sums 7}, we may exclude the prime factors 5, 7, and 11 from the minimal level of $\alpha$; that is, we may assume that $\alpha \in \QQ(\zeta_{3 \times 13})$. 
This case is ruled out by \Cref{lem:step3 case 5 remove 3}.
\end{proof}

We may finally deduce \Cref{T:main} by assembling the ingredients listed in \Cref{table:main breakdown}.

\section*{Acknowledgements}
The authors thank Mathilde Gerbelli-Gauthier for suggesting the term ``castle''.
This paper is a result of the virtual workshop ``Rethinking Number Theory 6'' funded by NSF grants DMS-2201085 and DMS-2418528. In addition,  Bajpai was supported by the European Research Council (ERC) under the European Union’s Horizon Europe research and innovation programme (grant agreement No. 101163794, GroupHype). Das was supported by Prime Minister's Research Fellowship, Government of India (ID: 0701990). Leudière was supported in part by the Pacific Institute for the Mathematical Sciences. Kedlaya was supported by NSF grant DMS-2401536, the UC San Diego Warschawski Professorship, and during fall 2025 by the Lodha Mathematical Sciences Institute.


\end{document}